\documentclass[A4]{amsart}
\usepackage[square,sort,comma,numbers]{natbib}
\usepackage{amsmath, amssymb, amstext, amsfonts, textcomp, amsxtra, amsbsy, amsgen, amsopn, amscd, mathrsfs, amsthm, latexsym, array}
\usepackage[textwidth=16cm,textheight=22cm,centering]{geometry}

\usepackage{enumerate}

\usepackage{mathdots, bbm, dsfont}
\usepackage{color}
\usepackage{graphicx}

\usepackage{xcolor}
\definecolor{cite}{rgb}{0.00,0.60,1.00}
\definecolor{url}{rgb}{1.00,0.10,0.80}
\definecolor{link}{rgb}{0.00,0.00,1.00}
\usepackage[colorlinks,linkcolor=link,urlcolor=url,citecolor=cite,pagebackref,breaklinks]{hyperref}

\def\leq{\leqslant}

\def\geq{\geqslant}

\allowdisplaybreaks

\usepackage[all]{xy}

\newtheorem{theorem}             {Theorem}  [section]
\newtheorem{definition} [theorem] {Definition}
\newtheorem{lemma}      [theorem]{Lemma}
\newtheorem{corollary}  [theorem]{Corollary}
\newtheorem{proposition}[theorem]{Proposition}
\newtheorem{remark} [theorem] {Remark}

\numberwithin{equation}{section} \everymath{\displaystyle}

\newcommand{\Cont}{{\rm C}}

\newcommand{\Sch}{\mathcal{S}}

\newcommand{\intL}{{\rm L}}

\newcommand{\Nr}{{\rm Nr}}
\newcommand{\Tr}{{\rm Tr}}

\newcommand{\gp}[1]{\mathbf{#1}}
\newcommand{\GL}{{\rm GL}}
\newcommand{\PGL}{{\rm PGL}}
\newcommand{\SL}{{\rm SL}}

\newcommand{\Z}{\mathbb{Z}}

\newcommand{\Mat}{{\rm M}}
\newcommand{\id}{\mathbbm{1}}

\makeatletter
\def\legendre@dash#1#2{\hb@xt@#1{%
  \kern-#2\p@
  \cleaders\hbox{\kern.5\p@
    \vrule\@height.2\p@\@depth.2\p@\@width\p@
    \kern.5\p@}\hfil
  \kern-#2\p@
  }}
\def\@legendre#1#2#3#4#5{\mathopen{}\left(
  \sbox\z@{$\genfrac{}{}{0pt}{#1}{#3#4}{#3#5}$}%
  \dimen@=\wd\z@
  \kern-\p@\vcenter{\box0}\kern-\dimen@\vcenter{\legendre@dash\dimen@{#2}}\kern-\p@
  \right)\mathclose{}}
\newcommand\legendre[2]{\mathchoice
  {\@legendre{0}{1}{}{#1}{#2}}
  {\@legendre{1}{.5}{\vphantom{1}}{#1}{#2}}
  {\@legendre{2}{0}{\vphantom{1}}{#1}{#2}}
  {\@legendre{3}{0}{\vphantom{1}}{#1}{#2}}
}
\def\dlegendre{\@legendre{0}{1}{}}
\def\tlegendre{\@legendre{1}{0.5}{\vphantom{1}}}
\makeatother

\newcommand{\Kl}{\mathrm{Kl}}

\newcommand{\ud}{\mathrm{d}}

\newcommand{\Q}{\mathbb{Q}}
\newcommand{\R}{\mathbb{R}}
\newcommand{\C}{\mathbb{C}}
\newcommand{\E}{\mathbf{E}}
\newcommand{\F}{\mathbf{F}}
\newcommand{\fF}{\mathbb{F}}
\newcommand{\bL}{\mathbf{L}}
\newcommand{\A}{\mathbb{A}}

\newcommand{\vO}{\mathcal{O}}

\newcommand{\vP}{\mathcal{P}}
\newcommand{\vp}{\mathfrak{p}}

\newcommand{\norm}[1][\cdot]{\lvert #1 \rvert}
\newcommand{\extnorm}[1]{\left\lvert #1 \right\rvert}

\newcommand{\Pairing}[2]{\langle #1, #2 \rangle}

\newcommand{\OFour}{\mathfrak{F}}
\newcommand{\invOFour}{\overline{\mathfrak{F}}}

\newcommand{\Rem}{\mathrm{r}}
\newcommand{\Ext}{\mathrm{e}}
\newcommand{\Inv}{\mathrm{i}}
\newcommand{\Mult}{\mathfrak{m}}
\newcommand{\Vor}{\mathcal{V}}
\newcommand{\VorH}{\mathcal{VH}}

\newcommand{\Trans}{\mathfrak{t}}

\newcommand{\Res}{{\rm Res}}

\newcommand{\Whi}{\mathcal{W}}

\newcommand{\Cond}{\mathbf{C}}
\newcommand{\cond}{\mathfrak{c}}
\newcommand{\condL}{\mathfrak{\rho}}

\newcommand{\BesselF}{\mathrm{j}}

\newcommand{\RamCst}{\vartheta}

\newcommand{\Vol}{{\rm Vol}}
\newcommand{\Ecp}{\boldsymbol{\mathcal{E}}}
\makeatletter
\newcommand{\rmnum}[1]{\romannumeral #1}
\newcommand{\Rmnum}[1]{\expandafter\@slowromancap\romannumeral #1@}
\makeatother

\title{On a Generalization of Motohashi's Formula: Non-archimedean Weight Functions}
\author{Han Wu}
\address{School of Mathematical Sciences, University of Scinece and Technology of China, 230026 Hefei, P. R. China}
\email{wuhan1121@yahoo.com}

\date{\today}

\begin{document}

\begin{abstract}
	This is a continuation of the adelic version of Kwan's formula. At non-archimedean places we give a bound of the weight function on the mixed moment side, when the weight function on the $\PGL_3 \times \PGL_2$ side is nearly the characteristic function of a short family. Our method works for any tempered representation $\Pi$ of $\PGL_3$, and reveals the structural reason for the appearance of Katz's hypergeometric sums in a previous joint work with P.Xi.
\end{abstract}

	\maketitle
	
	\tableofcontents
	
\section{Introduction}

	\subsection{Main Results}
	
	This paper is the first follow-up of our previous \cite{Wu24+}. Let $\F$ be a number field with adele ring $\A$. Let $\psi$ be the additive character of $\F \backslash \A$ \`a la Tate. Let $S_{\F}$ be the set of places of $\F$, and $S_{\infty}$ be the subset of archimedean places. Fix a cuspidal automorphic representation $\Pi$ of $\PGL_3(\A)$. We established an adelic version of Kwan's spectral reciprocity identity (see \cite[Theorem 1.1 \& (5.19)-(5.21)]{Wu24+}), which we recall in the special case of trivial central characters as follows.
	
\begin{theorem}
	Let $S_{\infty} \subset S \subset S_{\F}$ be any finite subset. At every place $v \in S$ there is a pair of weight functions $h_v$ and $\widetilde{h}_v$ with the auxiliary \emph{normalized} ones at finite places $\vp < \infty$
\begin{equation} \label{eq: NWt}
	H_{\vp}(\pi_{\vp}) := h_{\vp}(\pi_{\vp}) \frac{L(1, \pi_{\vp} \times \pi_{\vp})}{L(1/2, \Pi_{\vp} \times \pi_{\vp})}, \quad \widetilde{H}_{\vp}(\chi_{\vp}) := \frac{\widetilde{h}_{\vp}(\chi_{\vp})}{L(1/2, \Pi_{\vp} \times \chi_{\vp}^{-1}) L(1/2, \chi_{\vp})}, 
\end{equation}
	so that the following equation holds (where $\pi$ runs through cuspidal automorphic forms of $\PGL_2(\A)$)
\begin{multline*}
	\sum_{\pi} \frac{L(1/2, \Pi \times \pi)}{2 \Lambda_{\F}(2) L(1,\pi,\mathrm{Ad})} \cdot \prod_{v \mid \infty} h_v(\pi_v) \cdot \prod_{\vp \in S} H_{\vp}(\pi_{\vp}) + \\ 
	\sum_{\chi \in \widehat{\R_+ \F^{\times} \backslash \A^{\times}}} \int_{-\infty}^{\infty} \frac{\extnorm{L(1/2+i\tau, \Pi \times \chi)}^2}{2\Lambda_{\F}(2) \extnorm{L(1+2i\tau,\chi^2)}^2}
	\cdot \prod_{v \mid \infty} h_v(\pi(\chi_v,i\tau)) \cdot \prod_{\vp \in S} H_{\vp}(\pi(\chi_{\vp},i\tau)) \frac{\ud \tau}{2\pi} \\
	= \frac{1}{\zeta_{\F}^*} \sum_{\chi \in \widehat{\R_+ \F^{\times} \backslash \A^{\times}}} \int_{-\infty}^{\infty} L(1/2-i\tau, \Pi \times \chi^{-1}) L(1/2+i\tau, \chi) \cdot \prod_{v \mid \infty} \widetilde{h}_v(\chi_v \norm_v^{i\tau}) \cdot \prod_{\vp \in S} \widetilde{H}_{\vp}(\chi_{\vp} \norm_{\vp}^{i\tau}) \frac{\ud \tau}{2\pi} + \\
	\frac{1}{\zeta_{\F}^*} \sum_{\pm} \Res_{s_1 = \pm \frac{1}{2}} L(1/2-s_1, \Pi) \zeta_{\F}(1/2+s_1) \cdot \prod_{v \mid \infty} \widetilde{h}_v(\norm_v^{s_1}) \cdot \prod_{\vp \in S} \widetilde{H}_{\vp}(\norm_{\vp}^{s_1}),
\end{multline*}
	where we have used the abbreviation $\pi(\chi_v,s) := \pi(\chi_v \norm_v^s, \chi_v^{-1} \norm_v^{-s})$.
\end{theorem}
	
\noindent We only mention that the weight functions $h_v$ and $\widetilde{h}_v$ are integrals of a smooth Whittaker function $W_v \in \Whi(\Pi_v^{\infty}, \psi_v)$. More details from \cite{Wu24+} will be recalled when we need them.

	In this paper we focus on the weight functions at a \emph{non-archimedean} place $\vp < \infty$. We omit the subscript $\vp$ for simplicity. Let $\F$ be a non-archimedean local field of characteristic $0$ with valuation ring $\vO_{\F}$, and write $q=\Nr(\vp)$ from now on. In the case $\chi = \norm_{\F}^s$, the function $\widetilde{H}(\norm_{\F}^s)$ is a polynomial in $\C[q^s, q^{-s}]$, hence is entire. We introduce its Taylor expansion at any point $s_0$ as
\begin{equation} \label{eq: DNDWtTaylor}
	\widetilde{H}(\norm_{\F}^s) = \sideset{}{_{k \geq 0}} \sum \widetilde{H}(k; s_0) \left( s - s_0 \right)^k.
\end{equation}
	We summarize our main results (namely Proposition \ref{prop: DWtBde=1}, \ref{prop: DNDWtBde=1}, \ref{prop: DWtBde=2} \& \ref{prop: DNDWtBde=2}) as follows.

\begin{theorem} \label{thm: Main}
	Suppose the residual characteristic of $\F$ is not $2$. Let $\Pi \in \widehat{\PGL_3(\F)}$ be generic and tempered. Let $\pi_0 \in \widehat{\PGL_2(\F)}$ with $\cond(\pi_0) > 1$. We can choose $W \in \Whi(\Pi^{\infty},\psi)$ so that:
\begin{itemize}
	\item[(1)] The weight function satisfies $h(\pi) \geq 0$ for any $\pi \in \widehat{\PGL_2(\F)}$ and $h(\pi_0) > 0$;
	\item[(2)] For any unitary $\chi \in \widehat{\F^{\times}}$ and $\epsilon > 0$ the dual weight function satisfies 
	$$ h(\pi_0)^{-1} \widetilde{h}(\chi) \ll_{\epsilon} \Cond(\Pi)^{2+\epsilon} \cdot \left\{ \id_{\leq \max(\lfloor \frac{\cond(\pi_0)}{2} \rfloor, 6\cond(\Pi))}(\cond(\chi)) + q^{\frac{1}{2}} \id_{(-1)^{\frac{q-1}{2}} = \varepsilon_0} \id_{2 \nmid \cond(\chi) = \frac{\cond(\pi_0)}{2} \geq 3} \id_{\Ecp(\pi_0)}(\chi^2) \right\}, $$
	where $\varepsilon_0 = -1$ if $\pi_0$ is dihedral supercuspidal, and $\varepsilon_0 = 1$ otherwise; and the exceptional set of characters $\Ecp(\pi_0)$ of $\vO_{\F}^{\times}$ has size $O(q^{-1} \Cond(\pi_0)^{\frac{1}{2}})$ and is given below in Lemma \ref{lem: JacTypeSRC} (2);
	\item[(3)] For any $k \in \Z_{\geq 0}$ and $\epsilon > 0$ the normalized unramified dual weight function satisfies the bounds 
	$$ h(\pi_0)^{-1} \widetilde{H}(k; \pm 1/2) \ll_{\epsilon} \Cond(\Pi)^{4+\epsilon} q^{\lceil \frac{\cond(\pi_0)}{2} \rceil + \epsilon}. $$
\end{itemize}
\end{theorem}

\begin{remark}
	Our primary goal is to \emph{reveal the structural reason} for the bounds of the dual weight functions. Our main discovery is the \emph{quadratic elementary functions} given in \eqref{eq: QEleF}, which are ``building blocks'' of the Bessel functions of the relevant representations. See \S \ref{sec: Outline} for more details. Further extension of our method to the non-dihedral supercuspidal and Steinberg $\pi_0$ requires only plugging in the integral representation of the Bessel functions of such $\pi_0$, analogous to \cite[Theorem 1.6]{Wu24+}.
\end{remark}

\begin{remark}
	Specializing to $\F = \Q_p$ and $\Pi = \id \boxplus \id \boxplus \id$, Theorem \ref{thm: Main} corresponds to the main local non-archimedean computation of the recent work of Hu--Petrow--Young \cite{HPY25+} in the case $p \neq 2$.
\end{remark}

	\subsection{Notation and Convention}
	
	For a locally compact group $G$, let $\widehat{G}$ be the topological dual of unitary irreducible representations. For $\pi \in \widehat{G}$, we write $V_{\pi}$ for the underlying Hilbert space, and write $V_{\pi}^{\infty} \subset V_{\pi}$ for the subspace of smooth vectors if $G$ carries extra structure to make sense of the notion.
	
	Throughout the paper $\F$ is a local field of characteristic $0$, with residual characteristic $\neq 2$. Let $\norm_{\F}$ (resp. $v_{\F}$) be the valuation (resp. normalized additive valuation) of $\F$. Fix $\psi$ an additive character of conductor exponent $0$ and normalize the measures accordingly. The valuation ring of $\F$ is $\vO_{\F}$, while the valuation ideal is $\vP_{\F}$. We choose a uniformizer $\varpi_{\F} \in \vP_{\F}-\vP_{\F}^2$. Different choices of $\varpi_{\F}$ give different \emph{ramified} quadratic extensions of $\F$. Write $\gp{G}_d := \GL_d$ for simplicity, introduce some compact open subgroups of $\gp{G}_2(\F)$ as
	$$ \gp{K} := \GL_2(\vO_{\F}); \quad \gp{K}_0[\vP_{\F}^n] := \left\{ \begin{pmatrix} a & b \\ c & d \end{pmatrix} \in \gp{K} \ \middle| \ c \in \vP_{\F}^n \right\}, \ \forall n \in \Z_{\geq 0}; $$
	and some algebraic subgroups of $\gp{G}_2(\F)$ as
	$$ \gp{Z} = \gp{Z}_2(\F) := \left\{ z \id_2 \ \middle| \ z \in \F^{\times} \right\}, \quad  \gp{N}_2(\F) := \left\{ n(x) := \begin{pmatrix} 1 & x \\ & 1 \end{pmatrix} \ \middle| \ x \in \F \right\}, $$
	$$ \gp{A}_2(\F) := \left\{ \begin{pmatrix} t_1 & \\ & t_2 \end{pmatrix} \ \middle| \ t_1,t_2 \in \F^{\times} \right\}, \quad \gp{B}_2(\F) := \gp{A}_2(\F) \gp{N}_2(\F). $$
	For integers $n,m \geq 1$ we write $\Sch(n \times m, \F)$ for the space of Schwartz--Bruhat functions on $\Mat(n \times m, \F)$ the $n \times m$ matrices with entries in $\F$. The (inverse) $\psi$-Fourier transform is denoted and defined by
	$$ \widehat{\Psi}(X) = \OFour_{\psi}(\Psi)(-X) = \int_{\Mat(n \times m, \F)} \Psi(Y) \psi \left( \Tr(XY^T) \right) \ud Y. $$
	If no confusion occurs, we may omit $\psi$ from the notation. If this is the case, then the inverse Fourier transform is denoted by $\invOFour = \OFour_{\overline{\psi}} = \OFour_{\psi^{-1}}$. We introduce some elementary operators on the space of functions on $\F^{\times}$:
\begin{itemize}
	\item For functions $\phi$ on $\F^{\times}$, its \emph{extension} by $0$ to $\F$ is denoted by $\Ext(\phi)$, and its \emph{inverse} is $\mathrm{Inv}(\phi)(t) := \phi(t^{-1})$; for functions $\phi$ on $\F$, its \emph{restriction} to $\F^{\times}$ is denoted by $\Rem(\phi)$, and the operator $\Inv$ is 
	$$ \Inv = \Ext \circ \mathrm{Inv} \circ \Rem. $$
	\item For $s \in \C$, $\mu \in \widehat{\F^{\times}}$ and functions $\phi$ on $\F$, we introduce the operator $\Mult_s(\mu)$ by
	$$ \Mult_s(\mu)(\phi)(t) = \phi(t) \mu(t) \norm[t]_{\F}^s. $$
	\item For $\delta \in \F^{\times}$ we introduce the operator $\Trans(\delta)$ by
	$$ \Trans(\delta)(\phi)(y) = \phi(y \delta). $$
	\item Let $I_n = I_{n,\F}: \Cont(\F^{\times}) \to \Cont(\GL_n(\F))$ be given by $I_n(h)(g) := h(\det g)$.
\end{itemize}
	These notation apply to finite (field) extensions of $\F$.
	
	We introduce the standard involution of \emph{inverse-transpose} on $\GL_n(R)$ as $g^{\iota} := {}^tg^{-1}$.
	
	If $(U,\ud u)$ is a measured space with finite total mass, we introduce the \emph{normalized integral} as
\begin{equation} \label{eq: NormInt} 
	\oint_U f(u) \ud u := \frac{1}{\Vol(U, \ud u)} \int_U f(u) \ud u. 
\end{equation}
	For $n \in \Z_{\geq 1}$ we will frequently perform the following \emph{process of regularization} to an integral
\begin{equation} \label{eq: LnReg}
	\int_{\F} f(y) \ud y = \int_{\F} \left( \oint_{\vO_{\F}} f(y(1+\varpi_{\F}^nx)) \ud x \right) \ud y,
\end{equation}
	which we shall refer to as the \emph{level $n$ regularization (with respect) to $\ud y$}.

	\subsection{Outline of Proof}
	\label{sec: Outline}
	
	From the local main result \cite[Theorem 1.4]{Wu24+} the weight $h(\pi)$ and the dual weight $\widetilde{h}(\chi)$ functions are related to each other by a (hidden) relative orbital integral $H(y)$ via
\begin{equation} \label{eq: WtFviaBOI}
	h(\pi) = \int_{\F^{\times}} H(y) \cdot \BesselF_{\widetilde{\pi},\psi^{-1}} \begin{pmatrix} & -y \\ 1 & \end{pmatrix} \frac{\ud^{\times}y}{\norm[y]_{\F}}, \quad \widetilde{h}(\chi) = \int_{\F^{\times}} \psi(-y) \chi^{-1}(y) \norm[y]_{\F}^{-\frac{1}{2}} \cdot \widetilde{\Vor}_{\Pi}(H)(y) \ud^{\times} y
\end{equation}
	where $\BesselF_{\widetilde{\pi},\psi^{-1}}$ is the Bessel function of the contragredient representation $\widetilde{\pi}$ in the sense of \cite[\S 3.5]{BM05}, and $\widetilde{\Vor}_{\Pi}$ is the extended Voronoi transform characterized by the equations
	$$ \widetilde{\VorH}_{\Pi} :=\widetilde{\Vor}_{\Pi} \circ \Mult_{1}, \quad \Mult_{-2}(\beta) \circ I_3 \circ \widetilde{\VorH}_{\Pi} = \invOFour \circ \Mult_{-1}(\beta^{\iota}) \circ I_3 $$
	for all smooth matrix coefficients $\beta$ of $\Pi$.
	
	We need to choose $H(y)$ so that $h(\pi) \geq 0$ selects a \emph{short family} containing $\pi_0$. In other words the weight function $h(\pi)$ should be an approximation of the characteristic functions of $\{ \pi_0 \}$. From the first formula of \eqref{eq: WtFviaBOI} and the orthogonality of Bessel functions, one should expect that any such $H(y)$ is an approximation of $\BesselF_{\pi_0,\psi}\begin{pmatrix} & -y \\ 1 & \end{pmatrix}$. If $\pi_0$ is supercuspidal, the asymptotic analysis of both functions (in \S \ref{sec: ROI} and Lemma \ref{lem: TestFAsymp}) show that the two functions can be \emph{equal}. In general, we construct $H(y)$ from some test function $\phi_0^{\iota}*\phi_0$ on $\PGL_2(\F)$ of positive type in order to ensure $h(\pi) \geq 0$; we also require $\phi_0$ to be \emph{left covariant} with respect to some character of an open compact subgroup, which determines the (minimal $\gp{K}$-)\emph{type} of $\pi_0$ in the sense of \cite{MP94}. We use $\phi_0$ to deduce some integral representation \eqref{eq: TestFROI} of $H(y)$ which approximates the one of $\BesselF_{\pi_0,\psi}\begin{pmatrix} & -y \\ 1 & \end{pmatrix}$ obtained in \cite[Theorem 1.6]{Wu24+}. The integral representation of $H(y)$ in all cases is summarized in Corollary \ref{cor: TestFROI}, and is the departure point of our \emph{refined} analysis of the dual weight function $\widetilde{h}(\chi)$ described below.
	
	Our key observation on the dual weight function is the following decomposition into three parts. The first part is the contribution from $H_{\infty}(y)$, essentially the restriction of $H(y)$ to $v_{\F}(y) \leq - \cond(\pi_0)$ (see \eqref{eq: StabPar}). In this region $H_{\infty}(y)$ has a \emph{stable} behavior regardless $\pi_0$, i.e., the \emph{germ} of any Bessel relative orbital integral at infinity. To treat the corresponding $\widetilde{h}_{\infty}(\chi)$ we crucially rely on the extension of the Voronoi--Hankel transform developed in our previous paper \cite[Theorem 1.3]{Wu24+}. The remaining part $\widetilde{h}_c(\chi)$ corresponding to $H_c := H-H_{\infty}$ can be written as $\widetilde{h}_c(\chi) = \widetilde{h}_c^+(\chi) + \widetilde{h}_c^-(\chi)$, where 
	$$ \widetilde{h}_c^+(\chi) = \int_{\vO_{\F}} \chi^{-1}(y) \norm[y]_{\F}^{-\frac{1}{2}} \cdot \widetilde{\Vor}_{\Pi}(H)(y) \ud^{\times} y, \quad \widetilde{h}_c^-(\chi) = \int_{\F-\vO_{\F}} \psi(-y) \chi^{-1}(y) \norm[y]_{\F}^{-\frac{1}{2}} \cdot \widetilde{\Vor}_{\Pi}(H)(y) \ud^{\times} y. $$
	The bound of $\widetilde{h}_c^+$ is offered by the local functional equations via Lemma \ref{lem: LaurentTruncF}. The method via the local functional equations can also offer a crude bound in Lemma \ref{lem: DWt-TrivBd} of $\widetilde{h}_c^-$ when $\cond(\pi_0) \ll \cond(\Pi)$. 
	
	To \emph{refine} the bound of $\widetilde{h}_c^-$ in the case $\cond(\pi_0) \gg \cond(\Pi)$ we observe that $H_c(y)$ is a linear combination of translations of the following \emph{quadratic elementary functions} (see Definition \ref{def: QEleF1} \& \ref{def: QEleF2}, \eqref{eq: FPTestFNR0}-\eqref{eq: FPTestFNR2Bis} \& \eqref{eq: FPTestFR0}-\eqref{eq: FPTestFR2} for more details)
\begin{equation} \label{eq: QEleF}
	F_n, G_n(y^2) = \id_{v(y)=-n} \cdot \sideset{}{_{\pm}} \sum \eta_{\bL/\F}(\pm y) \psi(\pm y), 
\end{equation}
where $\bL/\F$ is the quadratic algebra extension associated with the type of $\pi_0$, and $\eta_{\bL/\F}$ is the corresponding quadratic character of $\F^{\times}$. We essentially change the order of integrations by first computing the dual weight of (the translates of) the quadratic elementary functions. In particular, the \emph{translation pattern} of the above quadratic elementary functions is responsible for the appearance of Katz's hypergeometric sums in our previous joint work with Xi \cite{WX23+}.

\begin{remark}
	Xi \cite{Xi23} has yet another transformation of the algebraic exponential sums in Petrow--Young's work \cite{PY19_CF}. This is mysterious and seems to lie beyond the framework of $\GL_2$ or $\GL_3$.
\end{remark}

\begin{remark}
	In the case of principal series $\pi_0$, our $\phi_0$ coincides with Nelson's test function in \cite{Ne20}. But we do not view it within the theory of micro-localized vectors, nor does our method of bounding the dual weight function rely on anything in that theory. Our choice of test function follows the idea of a previous work \cite{BFW21+} in the real case by analogy in terms of ``minimal $\gp{K}$-type'' \cite{MP94}.
\end{remark}

	\subsection{Acknowledgement}
	
	We thank Zhi Qi and Ping Xi for discussions related to the topics of the paper.

\section{Local Weight Transforms Revisited}

	We have expressed the local weight transforms in terms of the extended Voronoi transforms in \cite[Theorem 1.4]{Wu24+}. In that version, a test function $H(y)$ is some integral of a Kirillov function of a generic \emph{unitary} irreducible $\Pi$ (namely the restriction of a $W \in \Whi(\Pi^{\infty}, \psi)$ to the left-upper embedding of $\gp{G}_2(\F)$). As explained in \cite[(6.16)]{Wu24+}, the space of test functions $H(y)$ contains the Bessel orbital integrals of $\Cont_c^{\infty}(\gp{G}_2(\F))$. We shall restrict to the latter subspace of test functions and get more refined information on the local weight transforms in the case of trivial central characters.

	\subsection{Relative Orbital Integrals}
	\label{sec: ROI}
	
	For any $m \in \Z_{\geq 0}$ we introduce a function $E_m \in \Cont_c^{\infty}(\F^{\times})$ supported in the subset of \emph{square} elements of $\F^{\times}$ given by
\begin{equation} \label{eq: EleF}
	E_m(y^2) := \id_{v_{\F}(y)=-m} \cdot \norm[y]_{\F} \sum_{\pm} \int_{\pm1 + \vP_{\F}^{\lfloor \frac{m}{2} \rfloor}} \psi \left( y (u+u^{-1}) \right) \ud u.
\end{equation}
	Recall that for any $f \in \Cont_c^{\infty}(\GL_2(\F))$, the Bessel orbital integral \cite[(5.17)]{Wu24+} is given by
\begin{equation} \label{eq: BesselOrbInt}
	h(y) = \int_{\F^{\times}}  \int_{\F^2} f \left( \begin{pmatrix} 1 & x_1 \\ & 1 \end{pmatrix} \begin{pmatrix} & -y \\ 1 & \end{pmatrix} \begin{pmatrix} 1 & x_2 \\ & 1 \end{pmatrix} \begin{pmatrix} z & \\ & z \end{pmatrix} \right) \psi(-x_1-x_2) \ud x_1 \ud x_2 \ud^{\times} z.
\end{equation}
\begin{proposition} \label{prop: BesselOrbInt}
	As $f$ traverses $\Cont_c^{\infty}(\gp{G}_2(\F))$, the Bessel orbital integrals $h$ traverses
	$$ \Cont_c^{\infty}(\F^{\times}) \bigoplus \C E_{\geq n} $$
where $n \in \Z_{\geq 0}$ can be chosen arbitrarily and we have written 
\begin{equation} \label{eq: EleFBis}
	E_{\geq n}(y^2) := \sideset{}{_{m=n}^{\infty}} \sum E_m(y^2) = \id_{v_{\F}(y) \leq -n} \cdot \norm[y]_{\F} \sideset{}{_{\pm}} \sum \int_{\pm1 + \vP_{\F}^{\lfloor -\frac{v_{\F}(y)}{2} \rfloor}} \psi \left( y (u+u^{-1}) \right) \ud u. 
\end{equation}
\end{proposition}
\begin{proof}
	If $f$ traverses $\Cont_c^{\infty}(\gp{B}_2(\F) w \gp{N}_2(\F))$, then clearly $h$ traverses $\Cont_c^{\infty}(\F^{\times})$. Consider a function $f \in \Cont_c^{\infty}(\gp{B}_2(\F) \gp{N}_2(\F)^t)$ and let $\phi \in \Cont_c^{\infty}(\F^{\times} \times \F)$ be defined by
	$$ \phi(y,x) := \int_{\F^{\times}}  \int_{\F^2} f \left( \begin{pmatrix} 1 & x_1 \\ & 1 \end{pmatrix} \begin{pmatrix} y & \\ & 1 \end{pmatrix} \begin{pmatrix} 1 & \\ x & 1 \end{pmatrix} \begin{pmatrix} z & \\ & z \end{pmatrix} \right) \psi(-x_1) \ud x_1 \ud^{\times} z. $$
	From the equation of matrices
	$$ \begin{pmatrix} & -y \\ 1 & \end{pmatrix} \begin{pmatrix} 1 & x_2 \\ & 1 \end{pmatrix} = \begin{pmatrix} 1 & -y/x_2 \\ & 1 \end{pmatrix} \begin{pmatrix} y/x_2 & \\ & x_2 \end{pmatrix} \begin{pmatrix} 1 & \\ 1/x_2 & 1 \end{pmatrix} $$
	we easily deduce that the relative orbital integral \eqref{eq: BesselOrbInt} is given by
	$$ h(y) = \int_{\F} \phi \left( \frac{y}{x_2^2}, \frac{1}{x_2} \right) \psi \left( - \frac{y}{x_2} - x_2 \right) \ud x_2 = \int_{\F} \phi(yu^2, u) \psi(-yu - u^{-1}) \norm[u]_{\F}^{-2} \ud u. $$
	Let $\tau \in \{ 1, \varepsilon \}$, we obtain by an obvious change of variables
	$$ h(\tau y^2) = \norm[y]_{\F} \int_{\F} \phi(\tau u^2, y^{-1}u) \psi(-y(\tau u + u^{-1})) \norm[u]_{\F}^{-2} \ud u. $$
	Choose $i_0 \in \Z_{\geq 1}$ and $k_0, m_0 \in \Z$ such that:
\begin{itemize}
	\item $\phi(y(1+\delta_1), x(1+\delta_2)) = \phi(y,x), \quad \forall \ \delta_1, \delta_2 \in \vp^{i_0}$;
	\item for any $y \in \vp^{2+2k_0}$ and any $x \in \F$, we have $\phi(y,x) = 0$;
	\item for any $x \in \vp^{m_0}$ and any $y \in \F^{\times}$, we have $\phi(y,x) = \phi(y,0)$.
\end{itemize}
	Let $v_{\F}(y) = -m$ for some $m \geq \max(m_0, 2i_0+k_0, 2i_0)$, and take any $n \in \Z_{\geq 1}$ satisfying
	$$ 2n \geq m+k_0, \quad n \geq i_0, \quad m-n \geq i_0. $$
	Performing the level $n$ regularization to $\ud y$ we get
\begin{align*}
	h(\tau y^2) &= \norm[y]_{\F} \int_{\F-\vP_{\F}^{1+k_0}} \phi(\tau u^2, y^{-1}u) \left[ \oint_{\vO_{\F}} \psi \left( -y \left( \tau u(1+\varpi_{\F}^n x) + u^{-1}(1+\varpi_{\F}^n x)^{-1} \right) \right) \ud x \right] \norm[u]_{\F}^{-2} \ud u \\
	&= \norm[y]_{\F} \int_{\F-\vP_{\F}^{1+k_0}} \phi(\tau u^2, y^{-1}u) \psi(-y(\tau u + u^{-1})) \left[ \oint_{\vO_{\F}} \psi \left( -y \varpi_{\F}^n (\tau u - u^{-1}) x \right) \ud x \right] \norm[u]_{\F}^{-2} \ud u.
\end{align*}
	The non-vanishing of the inner integral implies
	$$ v_{\F}(\tau u - u^{-1}) \geq m-n \geq i_0, $$
which can be satisfied only if $\tau$ is a square modulo $\vP_{\F}$, i.e., $\tau = 1$. Moreover, we have $v_{\F}(u-u^{-1}) \geq i_0 \Leftrightarrow u \in \pm 1 + \vP_{\F}^{i_0}$. We therefore get $h(\varepsilon y^2) = 0$ and
\begin{align*} 
	h(y^2) &= \id_{k_0 \geq 0} \phi(1, 0) \cdot \norm[y]_{\F} \sum_{\pm} \int_{\pm 1 + \vP_{\F}^{i_0}} \psi(-y(u + u^{-1})) \ud u \\
	&= \id_{k_0 \geq 0} \phi(1, 0) \cdot \norm[y]_{\F} \sum_{\pm} \int_{\pm 1 + \vP_{\F}^{i_0}} \psi(-y(u + u^{-1})) \left[ \oint_{\vO_{\F}} \psi \left( -y \varpi_{\F}^n (u - u^{-1}) x \right) \ud x \right] \ud u,
\end{align*}
where we have performed the level $n = \lceil m/2 \rceil \geq i_0$ regularization to $\ud u$. Again the inner integral is non-vanishing only if
	$$ v_{\F}(u - u^{-1}) \geq m-n = \lfloor m/2 \rfloor \quad \Leftrightarrow \quad u \in \pm 1 + \vP_{\F}^{\lfloor m/2 \rfloor}. $$
	We have obtained
	$$ h(y^2) = \id_{k_0 \geq 0} \phi(1, 0) \cdot \norm[y]_{\F} \sum_{\pm} \int_{\pm 1 + \vP_{\F}^{\lfloor -\frac{v_{\F}(y)}{2} \rfloor}} \psi(-y(u + u^{-1})) \ud u, $$
	hence $h$ lies in the desired space of functions. Applying a smooth partition of unity to the open covering 
	$$ \gp{G}_2(\F) = \gp{B}_2(\F) w \gp{N}_2(\F) \bigcup \gp{B}_2(\F) \gp{N}_2(\F)^t $$
	we conclude the proof.
\end{proof}

\begin{definition} \label{def: EBOI}
	We call $E_{\geq n}$ the \emph{elementary Bessel orbital integrals}, abbreviated as {\rm EBOI}s.
\end{definition}

	\subsection{Voronoi--Hankel Transforms of EBOIs}
	
	We notice that the Mellin transform of $E_m$ is simple.
\begin{lemma} \label{lem: MellinStEBOI}
	Let $m \geq 2$ and $\chi$ be a quasi-character. We have
	$$ \int_{\F^{\times}} E_m(y) \chi(y) \ud^{\times} y = \id_{\cond(\chi)=m} \cdot \zeta_{\F}(1) \gamma(1/2,\chi^{-1},\psi)^2. $$
\end{lemma} 
\begin{proof}
	Applying the change of variables $y \mapsto y^2$ and the level $\lceil m/2 \rceil$ regularization to $\ud u$ we get
\begin{multline*}
	\int_{\F^{\times}} E_m(y) \chi(y) \ud^{\times} y = \frac{1}{2} \int_{\varpi_{\F}^{-m}\vO_{\F}^{\times}} E_m(y^2) \chi^2(y) \ud^{\times}y \\
	= q^m \int_{1+\vP_{\F}^{\lfloor \frac{m}{2} \rfloor}} \left( \int_{\varpi_{\F}^{-m} \vO_{\F}^{\times}} \psi \left( y(u+u^{-1}) \right) \chi^2(y) \ud^{\times} y \right) \ud u \\
	= q^m \int_{\vO_{\F}^{\times}} \left( \int_{\varpi_{\F}^{-m} \vO_{\F}^{\times}} \psi \left( y(u+u^{-1} \right) \chi^2(y) \ud^{\times} y \right) \ud u \\
	= \zeta_{\F}(1) q^m \int_{\vO_{\F}^{\times}} \int_{\vO_{\F}^{\times}} \psi \left( \frac{y^2u+u^{-1}}{\varpi_{\F}^m} \right) \chi^2 \left( \frac{y}{\varpi_{\F}^m} \right) \ud y \ud u.
\end{multline*}
	While the level $\lceil m/2 \rceil$ regularization to $\ud u$ also implies
	$$ \int_{\vO_{\F}^{\times}} \int_{\vO_{\F}^{\times}} \psi \left( \frac{\varepsilon y^2u+u^{-1}}{\varpi_{\F}^m} \right) \chi \left( \frac{\varepsilon y^2}{\varpi_{\F}^{2m}} \right) \ud y \ud u = 0, $$
	we can sum the above two equations to get
\begin{multline*} 
	\int_{\F^{\times}} E_m(y) \chi(y) \ud^{\times} y = \zeta_{\F}(1) q^m \int_{\vO_{\F}^{\times}} \int_{\vO_{\F}^{\times}} \psi \left( \frac{yu+u^{-1}}{\varpi_{\F}^m} \right) \chi \left( \frac{y}{\varpi_{\F}^{2m}} \right) \ud y \ud u \\
	= \zeta_{\F}(1) q^m \left( \int_{\vO_{\F}^{\times}} \psi \left( \frac{t}{\varpi_{\F}^m} \right) \chi \left( \frac{t}{\varpi_{\F}^m} \right) \ud t \right)^2.
\end{multline*}
	The last integral was studied in \cite[Proposition 4.6]{Wu14}, which is non-vanishing if and only if $\cond(\chi)=m$. Its relation with the local gamma factor 
	$$ \int_{\vO_{\F}^{\times}} \psi \left( \frac{t}{\varpi_{\F}^m} \right) \chi \left( \frac{t}{\varpi_{\F}^m} \right) \ud t = \gamma(1, \chi^{-1}, \psi) $$
	is the content of \cite[Exercise 23.5]{BuH06}.
\end{proof}

\begin{proposition} \label{prop: StabRang}
	There is $a = a(\Pi) \in \Z_{\geq 2}$, called the \emph{stability barrier} of $\Pi$, such that for any $m \geq a$
	$$ \VorH_{\Pi,\psi} \circ \Mult_{-1}(E_m)(t) = \psi(t) \cdot \id_{\varpi^{-m} \vO_{\F}^{\times}}(t), \quad \widetilde{\VorH}_{\Pi,\psi} \circ \Mult_{-1}(E_{\geq m})(t) = \psi(t) \cdot \id_{\vP_{\F}^m}(t^{-1}). $$
	Moreover, we have $a \leq \max(2\cond(\Pi), 1)$. More precisely, we have:
\begin{itemize}
	\item[(1)] If $\Pi$ is equal to or is included in $\mu_1 \boxplus \mu_2 \boxplus \mu_3$ we can take $a = \max(2\cond(\mu_1), 2\cond(\mu_2), 2\cond(\mu_3),1)$;
	\item[(2)] If $\Pi = \pi \boxplus \mu$ for a supercuspidal $\pi$ of $\gp{G}_2(\F)$ we can take $a = \max(\cond(\pi), 2\cond(\mu))$;
	\item[(3)] If $\Pi$ is supercuspidal we have $a \leq 2 \cond(\Pi)$.
\end{itemize}
\end{proposition}
\begin{proof}
	By Lemma \ref{lem: MellinStEBOI} and the local functional equation, the integral
	$$ \int_{\F^{\times}} \VorH_{\Pi,\psi} \circ \Mult_{-1}(E_m)(t) \chi^{-1}(t) \norm[t]_{\F}^{-s} \ud^{\times}t = \gamma(s, \Pi \times \chi, \psi) \int_{\F^{\times}} E_m(y) \chi(y) \norm[y]_{\F}^{s-1} \ud^{\times}y $$
is vanishing for any $\chi$ with $\cond(\chi) \neq m$. By the \emph{stability} of the local gamma factors (see \cite[Proposition (2.2)]{JS85}, \cite[Theorem 23.8]{BuH06} and \cite[Exercise 23.5]{BuH06}), there is $a \in \Z_{\geq 2}$ depending only on $\Pi$ so that the factors 
	$$ \gamma(s, \Pi \times \chi, \psi) = \gamma(s, \chi, \psi)^3 $$ 
depend only on $\omega_{\Pi} = \id$ if $\cond(\chi) \geq a$. Hence for $\cond(\chi)=m \geq a$ we have by Lemma \ref{lem: MellinStEBOI}
	$$ \gamma(s, \Pi \times \chi, \psi) \int_{\F^{\times}} E_m(y) \chi(y) \norm[y]_{\F}^{s-1} \ud^{\times}y = \gamma(s, \chi, \psi)^3 \cdot \zeta_{\F}(1) \gamma(3/2-s,\chi^{-1},\psi)^2 = \zeta_{\F}(1) \gamma(s+1,\chi,\psi). $$
	One verifies easily the desired formula for $\VorH_{\Pi,\psi}(E_m)$ by comparing their Mellin transforms. The desired formula for $E_{\geq m}$ follows easily by taking limits in the sense of tempered distributions on $\Mat_3(\F)$ via $I_3: \Cont(\F^{\times}) \to \Cont(\GL_3(\F))$ by \cite[Theorem 1.3]{Wu24+}. In the ``moreover'' part, (1) and (2) follow from the effective versions of the stability theorems for $\gp{G}_1$ and $\gp{G}_2$ \cite[Theorem 23.8 \& 25.7]{BuH06} together with the multiplicativity of the local gamma factors. For (3) we can take $a = \max(2\cond(\Pi),6)$ by examining the proof in \cite[\S 2]{JS85}. Now that a supercuspidal $\Pi$ of $\gp{G}_3$ has $\cond(\Pi) \geq 3$, we conclude the last assertion.
\end{proof}
\begin{remark}
	It would be interesting to know a sharp bound for $a$, say in general for stability theorems for $\gp{G}_n \times \gp{G}_t$. It could follow from the work of Bushnell--Henniart--Kutzko \cite{BHK98}.
\end{remark}

\section{Local Weight Functions}

	\subsection{Choice of Test Functions}
	
	The target representation $\pi_0 \in \widehat{\PGL_2(\F)}$ are:
\begin{itemize}
	\item[(1)] (Split) $\pi(\chi_0, \chi_0^{-1})$ with a \emph{ramified} and $\RamCst_3$-tempered quasi-character $\chi_0$ of $\F^{\times}$;
	\item[(2)] (Special) the \emph{quadratic} twists $\mathrm{St}_{\eta}$ of the Steinberg representation $\mathrm{St}$;
	\item[(3)] (Dihedral) $\pi_{\beta}$ with a \emph{unitary} regular character $\beta$ of $\E^{\times}$, $\E/\F$ being a quadratic field extension.
\end{itemize}
	Note that $\mathrm{St}_\eta$ is a sub-representation of $\pi(\eta \norm_{\F}^{1/2}, \eta \norm_{\F}^{-1/2})$.
	
	Let $\bL$ be a separable quadratic algebra extension of $\F$. The non-trivial element of the group $\mathrm{Aut}_{\F}(\bL)$ is denoted by $\bL \to \bL, x \mapsto \bar{x}$. Define 
	$$ \Tr=\Tr_{\bL/\F}: \bL \to \F, \ x \mapsto x + \bar{x}; \quad \Nr=\Nr_{\bL/\F}: \bL^{\times} \to \F^{\times}, \ x \mapsto x \bar{x}; \quad \norm[x]_{\bL} := \norm[\Nr(x)]. $$
	We associate to each target $\pi_0$ a \emph{parameter} $(\bL, \beta)$ by:
\begin{itemize}
	\item[(1)] $\pi(\chi_0, \chi_0^{-1})$: Let $\bL \simeq \F \oplus \F$ and $\beta: \bL^{\times} \simeq \F^{\times} \times \F^{\times} \to \mathrm{S}^1, (t_1,t_2) \mapsto \chi_0(t_1t_2^{-1})$;
	\item[(2)] $\mathrm{St}_{\eta}$: Let $\bL \simeq \F \oplus \F$ and $\beta: \bL^{\times} \simeq \F^{\times} \times \F^{\times} \to \C^{\times}, (t_1,t_2) \mapsto \eta(t_1t_2^{-1}) \norm[t_1t_2^{-1}]_{\F}^{1/2}$;
	\item[(3)] $\pi_{\beta}$: Let $\bL = \E$ and $\beta$ be the obvious one.
\end{itemize}

\noindent We equip $\bL$ with the self-dual Haar measure $\ud z$ with respect to $\psi_{\bL} := \psi \circ \Tr$. Write 
	$$ \vO_{\bL} := \left\{ x \in \bL \ \middle| \ \Tr(x), \Nr(x) \in \vO_{\F} \right\}. $$
\begin{definition}
	Define $\vP_{\bL} := \varpi_{\bL} \vO_{\bL}$ and $\varpi_{\bL} := \varpi_{\F}$ if $\bL$ is split, otherwise $\varpi_{\bL}$ is a uniformizer of $\bL$. For any (quasi-)character $\beta$ of $\bL^{\times}$ we define its \emph{$\F$-norm-$1$ conductor} as
	$$ \cond_1(\beta) := \min \left\{ n \ \middle| \ \beta \left( \bL^1 \cap (1+\vP_{\bL}^n) \right) = \{ 1 \} \right\}. $$
	If $\cond_1(\beta) > 0$ we say that $\beta$ is a \emph{regular} character.
\end{definition}
\begin{remark}
	Directly from the definition we have for any (quasi-)character $\beta$ of $\bL^{\times}$
\begin{itemize}
	\item $\cond_1(\beta) \leq \cond(\beta)$;
	\item $\cond_1(\beta) = \cond_1(\beta \cdot (\chi \circ \Nr))$ for any (quasi-)character $\chi$ of $\F^{\times}$.
\end{itemize}
\end{remark}
	
\begin{lemma} \label{lem: CondReInterp}
	 If $\beta$ is a regular character of $\bL^{\times}$ so that the restriction $\beta \mid_{\F^{\times}} = \eta_{\bL/\F}$ coincides with the quadratic character associated with the quadratic extension $\bL/\F$, then we have $\cond_1(\beta) = \cond(\beta)$.
\end{lemma}
\begin{proof}
	Write $n_0 := \cond(\beta) (\geq \cond_1(\beta) \geq 1)$. We shall prove that $\cond_1(\beta) < n_0$ is impossible.
	
\noindent (1) If $\bL \simeq \F \oplus \F$ is split then $\eta_{\bL/\F}=\id$. The assertion follows from $\cond(\chi_0) = \cond(\chi_0^2)$, because taking square is a group automorphism of $1+\vP_{\F}$ and the regularity of $\beta$ is equivalent with $\cond(\chi_0^2) > 0$.
	
\noindent (2) Assume $\bL/\F$ is non-split. Then $\beta$ coincides with $\eta_{\bL/\F}$ on $\F^{\times}$. If $n_0=1$ then the assertion follows from the condition $\beta \mid_{\E^1} \neq \id$ (equivalent to $\beta$ being \emph{regular}). Assume $n_0 \geq 2$ from now on. We choose the uniformizers $\varpi_{\F}$ and $\varpi_{\bL}$ of $\vO_{\F}$ and $\vO_{\bL}$ respectively so that 
\begin{equation} \label{eq: UnifChoice}
	\begin{cases}
		\varpi_{\bL} = \varpi_{\F} & \text{if } e=e(\bL/\F)=1 \\
		\overline{\varpi_{\bL}} = - \varpi_{\bL}, \varpi_{\F} = \Nr(\varpi_{\bL}) = - \varpi_{\bL}^2 & \text{if } e=e(\E/\F)=2
	\end{cases}.
\end{equation}

\noindent \textbf{Claim:} We have $2 \mid n_0$ in the ramified case.
\begin{proof}[Proof of Claim:] 
	In fact, if $\beta$ is trivial on $1+\vP_{\bL}^{2n+1}$ with $n \geq 1$, then for any element $\alpha \in 1+\vP_{\bL}^{2n}$ we can find $u_0 \in \vO_{\F}$ and $u_1 \in \vO_{\bL}$ such that $\alpha = 1 + \varpi_{\bL}^{2n} u_0 + \varpi_{\bL}^{2n+1}u_1$, since the residue class fields of $\bL$ and $\F$ are isomorphic. Then $\beta(\alpha) = \eta_{\bL/\F}(1+(-\varpi_{\F})^n u_0) = 1$. Hence $\beta$ is also trivial on $1+\vP_{\bL}^{2n}$.
\end{proof}
	
\noindent With these reductions, we find $\alpha \in 1+\vP_{\bL}^{n_0-1}$ such that $\beta(\alpha) \neq 1$. But $\Nr(\alpha) \in \F \cap (1+\vP_{\bL}^{n_0-1}) = 1+\vP_{\F}^{\lceil (n_0-1)/e \rceil}$ is a square. Hence $\Nr(\alpha) = k^2$ for some $k \in 1+\vP_{\F}^{\lceil (n_0-1)/e \rceil} \subset 1+\vP_{\bL}^{n_0-1}$, and $\beta(k^{-1}\alpha) = \beta(\alpha) \neq 1$ with $k^{-1} \alpha \in \bL^1 \cap (1+\vP_{\bL}^{n_0-1})$, proving the assertion.
\end{proof}
\begin{corollary} \label{cor: DihTrivCenterIsTwistedMin}
	Any \emph{dihedral supercuspidal representations} $\pi_{\beta}$ with trivial central character (the case (3) in the beginning of this subsection) is twisted minimal.
\end{corollary}
\begin{proof}
	For any (quasi-)character $\chi$ of $\F^{\times}$ we have $\cond(\beta) = \cond_1(\beta) = \cond_1 \left(\beta \cdot (\chi \circ \Nr) \right) \leq \cond \left(\beta \cdot (\chi \circ \Nr) \right)$. For any regular (quasi-)character $\beta$ of $\E^{\times}$ we have by \cite[Theorem 4.7]{JL70} with $e=e(\E/\F)$
\begin{equation} \label{eq: CondRelDih} 
	\cond(\pi_{\beta}) = \cond(\beta) f(\E/\F) + \cond(\psi_{\E}) = \tfrac{2n_0}{e}+e-1, 
\end{equation}
	where $f(\E/\F)$ is the residual field index and $\psi_{\E} = \psi_{\F} \circ \Tr_{\E/\F}$. The desired assertion follows readily.
\end{proof}

\begin{proposition} \label{prop: MeasNorm}
	Let $\bL^1$ be the kernel of the norm map $\Nr$.
\begin{itemize}
	\item[(1)] For any $x \in \bL^{\times}$, the following map
	$$ \iota_x: \F^{\times} \times \bL^1 \to \bL^{\times}, \quad (r,\alpha) \mapsto xr\alpha $$
	is a $2$-to-$1$ covering map onto an open subset of $\bL^{\times}$. 
	\item[(2)] (Polar decomposition) The Haar measure $\ud \alpha$ on $\bL^1$ determined by
\begin{equation} \label{eq: MeasComp}
	\norm[z]_{\bL}^{-1} \ud z = \norm[r]_{\F}^{-1} \ud r \ud \alpha, \quad z = \iota_x(r,\alpha) = x r \alpha,
\end{equation}
	where $\ud r$ is the self-dual Haar measure of $\F$ with respect to $\psi$, is independent of $x$, and satisfies
	$$ \Vol(\bL^1 \cap \vO_{\bL}, \ud \alpha) = 2^{e-1} \frac{\Vol(\vO_{\bL}^{\times}, \ud z)}{\Vol(\vO_{\F}^{\times}, \ud r)}, $$
where the number $e = e(\bL/\F)$ is the \emph{generalized ramification index} of $\bL/\F$ given by
\begin{equation} \label{eq: GenRamInd}
	e = \begin{cases} 1 & \text{if } \bL/\F \text{ is not ramified} \\ 2 & \text{if } \bL/\F \text{ is ramified} \end{cases}. 
\end{equation}
	\item[(3)] For any $n \in \Z_{\geq 1}$ we have $\Vol(\bL^1 \cap (1+\vP_{\bL}^n), \ud \alpha) = q^{-\lfloor \frac{n}{e} \rfloor - \frac{e-1}{2}}$.
\end{itemize}
\end{proposition}
\begin{proof}
	The proof of (1) is easy and omitted. The relation $\Im(\iota_x) = x \Im(\iota_1)$ implies the independence of $x$ in \eqref{eq: MeasComp}, since the measure $\norm[z]_{\bL}^{-1} \ud z$ is invariant by multiplication by elements in $\bL^{\times}$. To show that \eqref{eq: MeasComp} implies the stated formula for $\Vol(\bL^1 \cap \vO_{\bL}, \ud \alpha)$, it suffices to show 
	$$ \Vol(\iota_1(\vO_{\F}^{\times} \times (\vO_{\bL} \cap \bL^1))) = 2^{e-2} \Vol(\vO_{\bL}^{\times}). $$
	Now that $\iota_1(\vO_{\F}^{\times} \times (\vO_{\bL} \cap \bL^1)) = \vO_{\bL}^{\times} \cap \Nr^{-1}((\vO_{\F}^{\times})^2)$, we get
	$$ [\vO_{\bL}^{\times} : \iota_1(\vO_{\F}^{\times} \times (\vO_{\bL} \cap \bL^1))] = [\Nr(\vO_{\bL}^{\times}) : (\vO_{\F}^{\times})^2] = \frac{[\vO_{\F}^{\times} : (\vO_{\F}^{\times})^2]}{[\vO_{\F}^{\times} : \Nr(\vO_{\bL}^{\times})]} = \frac{2}{2^{e-1}} $$
	and conclude the proof of (2). To prove (3) we note that the map
	$$ \sigma: \vO_{\bL}^{\times} \to \bL^1 \cap \vO_{\bL}, \quad x \mapsto x / \bar{x} $$
	is surjective onto $\bL^1 \cap (1+\vP_{\bL}^n)$, whose pre-image is $\vO_{\F}^{\times}(1+\vP_{\bL}^n)$. It is also easy to see that the image of $\sigma$ is a subgroup of $\bL^1 \cap \vO_{\bL}$ with index $2^{e-1}$, by a refinement of Hilbert's 90. Therefore we get
	$$ \frac{\Im(\sigma)}{\bL^1 \cap (1+\vP_{\bL}^n)} \simeq \frac{\vO_{\bL}^{\times}}{\vO_{\F}^{\times}(1+\vP_{\bL}^n)} \simeq \frac{(\vO_{\bL}/\vP_{\bL}^n)^{\times}}{(\vO_{\F}/\vP_{\F}^{\lceil n/e \rceil})^{\times}} $$
	and deduce from it the measure relation
	$$ \frac{\Vol(\bL^1 \cap \vO_{\bL}, \ud \alpha)}{\Vol(\bL^1 \cap (1+\vP_{\bL}^n), \ud \alpha)} = 2^{e-1} \frac{q^{2n/e}}{q^{\lceil n/e \rceil}} \cdot \frac{\Vol(\vO_{\bL}^{\times}, \ud z)}{\Vol(\vO_{\F}^{\times}, \ud r)} \cdot \frac{\Vol(\vO_{\F}, \ud r)}{\Vol(\vO_{\bL}, \ud z)}. $$
	We conclude the proof of (3) by comparing with (2) and $\Vol(\vO_{\bL}, \ud z) = q^{-\frac{e-1}{2}}$.
\end{proof}
	
	Let $\eta_{\bL/\F}$ be the quadratic character of $\F^{\times}$ associated with $\bL/\F$, which is trivial on $\Nr(\bL^{\times})$. We have $\beta \mid_{\F^{\times}} = \eta_{\bL/\F}$ since $\pi_0$ has trivial central character. Write $\lambda(\bL/\F, \psi)$ for the Weil index, which is equal to $1$ if $\bL/\F$ is not ramified (see \cite[Corollary 3.2]{CC07}). Let $\tau$ run through a system of representatives of $\vO_{\F}^{\times}/(\vO_{\F}^{\times})^2$ (split), resp. $\Nr(\bL^{\times})/(\F^{\times})^2$ (non-split), and choose any $x_{\tau} \in \bL^{\times}$ s.t. $\Nr(x_{\tau})=\tau$; if $\bL/\F$ is split we also require $\Tr(x_{\tau}) = 1+\tau$. Define a function $H$ on $\F^{\times}$ with support contained in $\Nr(\bL^{\times})$ by 
\begin{equation} \label{eq: TestFROI}
	H(\tau y^2) = \lambda(\bL/\F,\psi) \norm[\tau]_{\F}^{\frac{1}{2}} \cdot \norm[y]_{\F} \cdot \eta_{\bL/\F}(y) \cdot \int_{\bL^1 \cap \vO_{\bL}} \beta(x_{\tau} \alpha) \psi(\Tr(x_{\tau} \alpha) y) \ud \alpha.
\end{equation}
	It is clear that $H$ is independent of any choice made in the definition. For definiteness, we fix a $\varepsilon \in \vO_{\F}^{\times} - (\vO_{\F}^{\times})^2$ and choose a system $\{ \tau \}$ of representatives of $\vO_{\F}^{\times}/(\vO_{\F}^{\times})^2$, resp. $\Nr(\bL^{\times})/(\F^{\times})^2$ as
\begin{itemize}
	\item[(1)] $\{ 1, \varepsilon \}$ if $\bL/\F$ is split or unramified (in the latter case $\bL = \F[\sqrt{\varepsilon}]$),
	\item[(2)] $\{ 1, -\varpi_{\F} \}$ if $\bL/\F$ is ramified (with $\bL = \F[\sqrt{\varpi_{\F}}]$ and we fix a uniformizer $\varpi_{\bL} = \sqrt{\varpi_{\F}}$ of $\bL$).
\end{itemize}
	The (partial) Mellin transform of $H$ will be fundamental for our analysis. We record it as follows.
\begin{lemma} \label{lem: TestFMellin}
	Let $\chi \in \widehat{\F^{\times}}$ and $n \in \Z_{\geq 1}$. We have
\begin{multline*} 
	\varepsilon_n(\chi,H) := \int_{\varpi_{\F}^{-n} \vO_{\F}^{\times}} H(y) \chi(y) \ud^{\times} y = \zeta_{\F}(1) \cdot \\
	\begin{cases}
		\id_{2 \mid n, \cond(\chi_0\chi) = \cond(\chi_0^{-1}\chi) = n/2} \cdot \varepsilon(1/2,\chi_0 \chi^{-1}, \psi) \varepsilon(1/2,\chi_0^{-1} \chi^{-1}, \psi) & \text{if } \bL/\F \text{ split and } \cond(\chi_0\chi), \cond(\chi_0^{-1}\chi) \geq 1 \\
		- \id_{n = 2, \cond(\chi_0)=1, \cond(\chi_0^2) \neq 0} \cdot q^{-1} \chi_0\chi(\varpi_{\F}) \varepsilon(1/2,\chi_0^{-1} \chi^{-1}, \psi) & \text{if } \bL/\F \text{ split and } \cond(\chi_0\chi) = 0 \neq \cond(\chi_0^{-1}\chi) \\
		- \id_{n = 2, \cond(\chi_0)=1, \cond(\chi_0^2) \neq 0} \cdot q^{-1} \chi_0^{-1}\chi(\varpi_{\F}) \varepsilon(1/2,\chi_0 \chi^{-1}, \psi) & \text{if } \bL/\F \text{ split and } \cond(\chi_0^{-1}\chi) = 0 \neq \cond(\chi_0\chi) \\
		- \id_{n = 2, \cond(\chi_0)=1, \cond(\chi_0^2)=0} \cdot q^{-2} \chi^2(\varpi_{\F}) & \text{if } \bL/\F \text{ split and } \cond(\chi_0\chi) = 0 = \cond(\chi_0^{-1}\chi) \\
		\id_{2 \mid en, \cond(\beta \cdot (\chi \circ \Nr)) = en/2+1-e} \cdot \varepsilon(1/2, \pi_{\beta^{-1}} \otimes \chi^{-1}, \psi) & \text{if } \bL/\F \text{ non-split}
	\end{cases}.
\end{multline*}
	In particular, we have $\extnorm{\varepsilon_n(\chi,H)} \leq \zeta_{\F}(1)$.
\end{lemma}
\begin{proof}
	We divide the domain of integration into cosets of square elements. 

\noindent \textbf{(A)} In the split case, we have
\begin{align*}
	\int_{\varpi_{\F}^{-n} \vO_{\F}^{\times}} H(y) \chi(y) \ud^{\times} y &= \id_{2 \mid n} \cdot \frac{1}{2} \sum_{\tau \in \{ 1, \varepsilon \}} \int_{\varpi_{\F}^{-\frac{n}{2}} \vO_{\F}^{\times}} \chi(\tau y^2) \norm[y] \left( \int_{\vO_{\F}^{\times}} \chi_0(\tau \alpha^2) \psi \left( (\tau \alpha + \alpha^{-1}) y \right) \ud \alpha \right) \ud^{\times} y \\
	&= \id_{2 \mid n} \cdot \zeta_{\F}(1)^{-1} \int_{\varpi_{\F}^{-\frac{n}{2}} \vO_{\F}^{\times} \times \varpi_{\F}^{-\frac{n}{2}} \vO_{\F}^{\times}} \chi(t_1t_2) \norm[t_1t_2]^{\frac{1}{2}} \cdot \chi_0(t_1t_2^{-1}) \psi(t_1+t_2) \ud^{\times}t_1 \ud^{\times}t_2 \\
	&= \id_{2 \mid n} \cdot \zeta_{\F}(1) q^{\frac{n}{2}} \left( \int_{\varpi_{\F}^{-\frac{n}{2}} \vO_{\F}^{\times}} \psi(t) \cdot \chi \chi_0(t) \frac{\ud t}{\norm[t]} \right) \left( \int_{\varpi_{\F}^{-\frac{n}{2}} \vO_{\F}^{\times}} \psi(t) \cdot \chi \chi_0^{-1}(t) \frac{\ud t}{\norm[t]} \right),
\end{align*}
	where we have taken into account that for each $\tau$ the map
	$$ \varpi_{\F}^{-\frac{n}{2}} \vO_{\F}^{\times} \times \vO_{\F}^{\times} \to \varpi_{\F}^{-\frac{n}{2}} \vO_{\F}^{\times} \times \varpi_{\F}^{-\frac{n}{2}} \vO_{\F}^{\times}, \quad (y,\alpha) \mapsto (\tau y \alpha, y \alpha^{-1}) $$
	is $2$-to-$1$. The stated formulae then follow from the relation between the Gauss integrals and the local epsilon-factors given in \cite[Exercise 23.5]{BuH06}, and a direct computation in the unramified case. 
		
\noindent \textbf{(B)} Similarly in the non-split case we have
\begin{align*}
	\int_{\varpi_{\F}^{-n}\vO_{\F}^{\times}} H(y) \chi(y) \ud^{\times}y &= \id_{2 \mid en} \cdot \frac{\zeta_{\F}(1)}{\Vol(\bL^1)} \int_{\varpi_{\bL}^{-\frac{en}{2}}\vO_{\bL}^{\times}} H(\Nr(z)) \chi(\Nr(z)) \frac{\ud z}{\norm[z]_{\bL}} \\
	&= \id_{2 \mid en} \cdot \zeta_{\F}(1) q^{\frac{n}{2}} \lambda(\bL/\F,\psi) \int_{\varpi_{\bL}^{-\frac{en}{2}}\vO_{\bL}^{\times}} \beta(z) \chi(\Nr(z)) \cdot \psi_{\bL}(z) \frac{\ud z}{\norm[z]_{\bL}} \\
	&= \id_{2 \mid en} \cdot \zeta_{\F}(1) q^{\frac{n}{2}} \lambda(\bL/\F,\psi) \cdot \id_{\cond \left( \beta \cdot (\chi \circ \Nr) \right) + e-1 = \frac{en}{2}} \varepsilon(1, \beta^{-1} \cdot (\chi^{-1} \circ \Nr), \psi_{\bL}),
\end{align*}	
	where we have applied \cite[Exercise 23.5]{BuH06} in the last line. Now that \cite[Theorem 4.7]{JL70} implies
	$$ \lambda(\bL/\F,\psi) \varepsilon(s, \beta \cdot (\chi \circ \Nr), \psi_{\bL}) = \varepsilon(s, \pi_{\beta} \otimes \chi, \psi) = q^{n \cdot (\frac{1}{2}-s)} \varepsilon(\tfrac{1}{2}, \pi_{\beta} \otimes \chi, \psi), $$
	we conclude the desired formula. Note that in this case $\beta \cdot (\chi \circ \Nr)$ is never unramified.
\end{proof}

	We give a first asymptotic analysis of $H$ as follows.

\begin{lemma} \label{lem: CongCondEquiv}
	Let $\alpha \in \bL^1$ and $n_1 \in \Z_{\geq 1}$. Then $\alpha - \overline{\alpha} \in \vP_{\bL}^{n_1}$ if and only if $\alpha \in \pm 1 + \vP_{\bL}^{n_1}$.
\end{lemma}
\begin{proof}
	If $\alpha \in \pm 1 + \vP_{\bL}^{n_1}$, then it clear that $\alpha - \overline{\alpha} \in \vP_{\bL}^{n_1}$. The converse is justified as follows.

\noindent (a) If $\bL/\F$ is split, we write $\alpha = (t,t^{-1})$ with $t \in \F^{\times}$. Then $t-t^{-1} \in \vP_{\F}^{n_1}$. We must have $t \in \vO_{\F}^{\times}$, and $t \equiv t^{-1} \pmod{\vP_{\F}}$. Hence $t \in \pm 1 + \vP_{\F}$. Writing $t = \pm 1 + \varpi_{\F}^k u$ for some $k \in \Z_{\geq 1}$ and $u \in \vO_{\F}^{\times}$ we infer that $t-t^{-1} \in 2 \varpi_{\F}^k u + \vP_{\F}^{k+1} \subset \vP_{\F}^k-\vP_{\F}^{k+1}$. Hence $k \geq n_1$, and $t \in \pm 1 + \vP_{\F}^{n-1}$.

\noindent (b) If $\bL/\F$ is a field, then $\overline{\alpha} = \alpha^{-1}$. We repeat the above argument with (subscript) $\F$ replaced by $\bL$.
\end{proof}

\begin{lemma} \label{lem: AddParMultChar}
	There is $c_{\beta} \in \vO_{\F}^{\times}$, called the \emph{additive parameter} and is uniquely determined up to multiplication by elements in $1+\vP_{\F}^{\lceil \frac{n_0}{2e} \rceil}$, so that for all $u \in \vP_{\bL}^{\lfloor \frac{n_0}{2} \rfloor}$
	$$ \beta(1+u) = \begin{cases}
		\psi \left( \varpi_{\F}^{-n_0} c_{\beta} (u-\bar{u}) \right) & \text{if } \bL/\F \text{ split and } 2 \mid n_0, \\
		\psi \left( \varpi_{\F}^{-n_0} c_{\beta} \left( (u-\bar{u}) - 2^{-1}(u^2-\bar{u}^2) \right) \right) & \text{if } \bL/\F \text{ split and } 2 \nmid n_0; \\
		\psi \left( \varpi_{\F}^{-n_0} \sqrt{\varepsilon} c_{\beta} (u-\bar{u}) \right) & \text{if } \bL/\F \text{ unramified and } 2 \mid n_0, \\
		\psi \left( \varpi_{\F}^{-n_0} \sqrt{\varepsilon} c_{\beta} \left( (u-\bar{u}) - 2^{-1}(u^2-\bar{u}^2) \right) \right) & \text{if } \bL/\F \text{ unramified and } 2 \nmid n_0; \\
		\psi \left( \varpi_{\bL}^{-n_0-1} c_{\beta} (u-\bar{u}) \right) & \text{if } \bL/\F \text{ ramified}.
	\end{cases} $$
\end{lemma}
\begin{proof}
	We only treat the cases $2 \nmid n_0 := 2m+1$ if $\bL/\F$ is not ramified, the even ones being simpler. The following map
	$$ \log_{\F} : (1+\vP_{\F}^m) / (1+\vP_{\F}^{2m+1}) \to \vP_{\F}^m/\vP_{\F}^{2m+1}, \quad 1+x \mapsto x - 2^{-1} x^2 $$
	is a group isomorphism. If $\bL/\F$ is split there is $c_{\beta}' \in \vO_{\F}^{\times}/(1+\vP_{\F}^{m+1})$ such that
\begin{equation} \label{eq: AddParSplit}
	\beta(1+\varpi_{\F}^m u) = \psi \left( c_{\beta}' \left( \varpi_{\F}^{-(m+1)} u - 2^{-1} \varpi_{\F}^{-1} u^2 \right) \right), \quad \forall u \in \vO_{\bL} \left( \simeq \vO_{\F} \times \vO_{\F} \right). 
\end{equation}
	One checks easily that $c_{\beta} := c_{\beta}'$ is the required one. If $\bL/\F$ is unramified we get an analogue of \eqref{eq: AddParSplit}
\begin{equation} \label{eq: AddParUnr}
	\beta(1+\varpi_{\F}^m u) = \psi \left( \Tr \left( c_{\beta}' \left( \varpi_{\F}^{-(m+1)} u - 2^{-1} \varpi_{\F}^{-1} u^2 \right) \right) \right), \quad \forall u \in \vO_{\bL}. 
\end{equation}

\noindent Now that $\beta$ is trivial on $\vO_{\F}^{\times}$ and $\log_{\F}$ is surjective, implying $\psi \circ \Tr(c_{\beta}' \vP_{\F}^{-m-1}) = 1$, i.e., $c_{\beta}' + \overline{c_{\beta}'} \in \vP_{\F}^{m+1}$. Therefore we may write $c_{\beta}' \in c_{\beta} \sqrt{\varepsilon} + \vP_{\F}^{m+1}$ for some $c_{\beta} \in \vO_{\F}^{\times}/(1+\vP_{\F}^{m+1})$. If $\bL/\F$ is ramified, then $2 \mid n_0$ by the claim proved in Lemma \ref{lem: CondReInterp} and $\psi_{\bL} := \psi \circ \Tr_{\bL/\F}$ has conductor exponent $-1$. We similarly get
\begin{equation} \label{eq: AddParR}
	\beta(1+u) = \psi \left( \varpi_{\F}^{-\frac{n_0}{2}} \varpi_{\bL}^{-1} (c_{\beta}' u - \overline{c_{\beta}'} \bar{u}) \right), \quad \forall u \in \vP_{\bL}^{\frac{n_0}{2}} 
\end{equation}
	for some $c_{\beta}' \in \vO_{\bL}^{\times}/(1+\vP_{\bL}^{\frac{n_0}{2}})$. Note that $\beta$ is trivial on $1+\vP_{\F}$, implying $\psi \left( \varpi_{\F}^{-\frac{n_0}{2}} \varpi_{\bL}^{-1} (c_{\beta}' - \overline{c_{\beta}'}) \vP_{\F}^{\lceil \frac{n_0}{4} \rceil} \right) = 1$,  i.e., $c_{\beta}' - \overline{c_{\beta}'} \in \vP_{\bL}^{2 \lfloor \frac{n_0}{4} \rfloor+1}$. Therefore we may write $c_{\beta}' \in c_{\beta} + \vP_{\bL}^{2 \lfloor \frac{n_0}{4} \rfloor+1}$, hence may take $c_{\beta}'=c_{\beta}$ in \eqref{eq: AddParR}, for some $c_{\beta} \in \vO_{\F}^{\times}/(1+\vP_{\F}^{\lceil \frac{n_0}{4} \rceil})$.
 \end{proof}
 
 \begin{remark} \label{rmk: AddParMultChar}
 	In the split case the character $\beta$ is constructed from a character $\chi_0$ of $\F^{\times}$. We also call $c_{\beta} = c_{\chi_0}$ the additive parameter of $\chi_0$.
 \end{remark}

\begin{lemma} \label{lem: TestFAsymp}
	Let $e$ be the generalized ramification index given in \eqref{eq: GenRamInd}, and write $n_0 := \cond(\beta)$.
\begin{itemize}
	\item[(1)] In the domain $v_{\F}(y) \leq -4n_0/e-2(e-1)$ we have $H(y) = E_{\geq 2n_0/e+e-1}(y)$.
	\item[(2)] \emph{Assume $\beta$ is regular.} For any $\tau$ we have $H(\tau y^2) = 0$ if $v_{\F}(y) \geq (2-n_0)e^{-1}-1$.
	\item[(3)] If $\tau \neq 1$, then $H(\tau y^2) = 0$ unless $v_{\F}(y) = 1-e-e^{-1}n_0$.
	\item[(4)] Suppose $e=1, n_0 \geq 2$. If $\tau = 1$ and $v_{\F}(y) = -m$ with $n_0 \leq m \leq 2n_0-1$, then we may put the following condition in the domain of integration in \eqref{eq: TestFROI}:
	$$ \Tr(\alpha) \in 
	\begin{cases} 
		\pm 2(1+\vP_{\F}^{2(n_0-1)}) & \text{if } m=2n_0-1 \\
		\pm 2(1+\varpi_{\F}^{2(m-n_0)} \vO_{\F}^{\times}) & \text{if } n_0<m<2n_0-1 \\
		\vO_{\F} - \sideset{}{_{\pm}} \bigcup \pm 2(1+\vP_{\F}) & \text{if } m=n_0
	\end{cases}. $$
	\item[(5)] Suppose $e=2$. We may put the following condition in the domain of integration in \eqref{eq: TestFROI}:
	$$ \begin{cases}
		\Tr(\alpha) \in \pm 2(1+\vP_{\F}^{n_0-1}) & \text{if } \tau=1, \text{ and } -v_{\F}(y):=m = n_0 \\
		\Tr(\alpha) \in \pm 2(1+\varpi_{\F}^{2m-n_0+1} \vO_{\F}^{\times}) & \text{if } \tau=1, \text{ and } \tfrac{n_0}{2}+1 \leq -v_{\F}(y):=m \leq n_0-1
	\end{cases}. $$
\end{itemize}
\end{lemma}
\begin{proof}
	(1) We only consider non-split $\bL$, leaving the simpler split case as an exercise. Let $n = -v_{\F}(y) \geq 4n_0/e+2(e-1)$. For $\chi \in \widehat{\vO_{\F}^{\times}}$, by Lemma \ref{lem: TestFMellin} the integral $\int_{\varpi_{\F}^{-n}\vO_{\F}^{\times}} H(y) \chi(y) \ud^{\times}y \neq 0$ is non-zero only if
	$$ \tfrac{en}{2} = \cond(\beta (\chi \circ \Nr)) + \cond(\psi_{\bL}) \quad \Leftrightarrow \quad \max(\cond(\beta), \cond(\chi \circ \Nr)) \geq \cond(\beta (\chi \circ \Nr)) = \begin{cases} n-1 & \text{if } e=2 \\ n/2 & \text{if } e=1 \end{cases}. $$
	By \cite[Proposition \Rmnum{5}.2.3 \& Corollay \Rmnum{5}.3.3]{Se79} we have
\begin{equation} \label{eq: CondCompNormF} 
	\cond(\chi \circ \Nr) = \begin{cases} \max(2\cond(\chi)-1,0) & \text{if } e=2 \\ \cond(\chi) & \text{if } e=1 \end{cases}. 
\end{equation}
	In view of the assumption on $n$ (implying $\cond(\beta (\chi \circ \Nr)) > n_0$), the non-vanishing condition is equivalent to $2 \mid n \ \& \ \cond(\chi)=n/2$. Now that by \cite[Theorem 4.7]{JL70} $\ell(\chi) = \cond(\chi)-1 = n/2-1 > 2\ell(\pi_{\beta}) = \cond(\pi_{\beta})-2 = 2\cond(\beta)/e+e-3 = 2n_0/e+e-3$, we can apply the stability theorem \cite[Theorem 25.7]{BuH06} and obtain
\begin{equation} \label{eq: DHRel}
	\int_{\varpi_{\F}^{-n}\vO_{\F}^{\times}} H(y) \chi(y) \ud^{\times}y = \id_{2 \mid n, \cond(\chi)=\frac{n}{2}} \cdot \zeta_{\F}(1) \varepsilon(1/2,\chi^{-1},\psi)^2. 
\end{equation}
	We conclude by comparing with Lemma \ref{lem: MellinStEBOI} and applying the Mellin inversion on $\vO_{\F}^{\times}$.
	
\noindent (2) We average over the change of variables $\alpha \mapsto \alpha \delta$ for $\delta \in \bL^1 \cap (1+\vP_{\bL}^{n_0-1}) =: U$. Note that
	$$ y x_{\tau} \alpha (\delta - 1) \in \vP_{\bL}^{1-e}, \quad \Tr(\vP_{\bL}^{1-e}) \subset \vP_{\bL}^{1-e} \cap \F = \vO_{\F} \quad \Rightarrow \quad \psi(\Tr(x_{\tau} \alpha \delta) y) = \psi(\Tr(x_{\tau} \alpha) y) $$
	for $y$ satisfying the stated condition. By Lemma \ref{lem: CondReInterp}, the character $\beta$ is non-trivial on $U$, hence
	$$ H(\tau y^2) = \lambda(\bL/\F,\psi) \norm[\tau]_{\F}^{\frac{1}{2}} \cdot \norm[y]_{\F} \cdot \eta_{\bL/\F}(y) \cdot \int_{\bL^1 \cap \vO_{\bL}} \beta(x_{\tau} \alpha) \psi(\Tr(x_{\tau} \alpha) y) \cdot \left( \oint_{U} \beta(\delta) \ud \delta \right) \ud \alpha = 0. $$
	
\noindent (3) The idea is to average over similar change of variables $\alpha \mapsto \alpha \delta$ for $\delta \in \bL^1 \cap (1+\vP_{\bL}^{n}) =: U_n$ with
\begin{equation} \label{eq: RegCond}
	n \geq \max(n_0, (1-e-ev_{\F}(y))/2)
\end{equation} 
	and perform a further change of variables $\delta = \delta(u)$
	$$ \delta = \frac{1+\varpi_{\bL}^n u}{1+\overline{\varpi_{\bL}^n u}}, \ u \in \vO_{\bL} \quad \Rightarrow \quad \delta - 1 \in \varpi_{\bL}^n u - \overline{\varpi_{\bL}^n u} + y^{-1}\vP_{\bL}^{1-e}. $$
	Then we have $\beta(\delta)=1$, $\psi(\Tr(x_{\tau} \alpha(\delta-1)) y) = \psi_{\bL}(y x_{\tau} \alpha (\varpi_{\bL}^n u - \overline{\varpi_{\bL}^n u}))$ and get
\begin{multline} \label{eq: RegI}
	H(\tau y^2) = \lambda(\bL/\F,\psi) \norm[\tau]_{\F}^{\frac{1}{2}} \cdot \norm[y]_{\F} \cdot \eta_{\bL/\F}(y) \cdot \int_{\bL^1 \cap \vO_{\bL}} \beta(x_{\tau} \alpha) \psi(\Tr(x_{\tau} \alpha) y) \cdot \\
	\left[ \oint_{\vO_{\bL}} \psi \left( y (x_{\tau} \alpha - \overline{x_{\tau} \alpha}) (\varpi_{\bL}^n u - \overline{\varpi_{\bL}^n u}) \right) \ud u \right] \ud \alpha.
\end{multline}
	For any $m \in \Z$ we introduce $\vP_{\bL}^{m,-} := \left\{ x \in \vP_{\bL}^m \ \middle| \ \bar{x} = -x \right\}$. We see that
	$$ \vO_{\bL} \to \vP_{\bL}^{n,-}, \quad u \mapsto \varpi_{\bL}^n u - \overline{\varpi_{\bL}^n u} $$
	is a surjective group homomorphism. Therefore the inner integral in \eqref{eq: RegI} is non-vanishing only if
\begin{equation} \label{eq: RegINVCond} 
	y (x_{\tau} \alpha - \overline{x_{\tau} \alpha}) \in \vP_{\bL}^{1-e-n,-} \quad \Leftrightarrow \quad v_{\bL}(x_{\tau} \alpha - \overline{x_{\tau} \alpha}) \geq 1-e-n-ev_{\F}(y). 
\end{equation}
	We choose $n = 1-2e-ev_{\F}(y)$, which is consistent with the condition \eqref{eq: RegCond} if 
\begin{equation} \label{eq: VCond}
	v_{\F}(y) \leq \lfloor \min(e^{-1}-2-e^{-1}n_0, e^{-1}-3) \rfloor = -e-e^{-1}n_0, 
\end{equation}
	so that \eqref{eq: RegINVCond} becomes $v_{\bL}(x_{\tau} \alpha - \overline{x_{\tau} \alpha}) \geq e$ or $v_{\F}((x_{\tau} \alpha - \overline{x_{\tau} \alpha})^2) \geq 2$. In view of the equality
\begin{equation} \label{eq: TauDecomp}
	\tau = x_{\tau} \alpha \cdot \overline{x_{\tau} \alpha} = \tfrac{1}{4} \left\{ (x_{\tau} \alpha + \overline{x_{\tau} \alpha})^2 - (x_{\tau} \alpha - \overline{x_{\tau} \alpha})^2 \right\} 
\end{equation}
	and $v_{\F}(\tau) \in \{0,1\}$, we deduce that $v_{\F}(\tau) = v_{\F}((x_{\tau} \alpha + \overline{x_{\tau} \alpha})^2) \in 2\Z$. Hence $v_{\F}(\tau)=0$ and $\tau$ is a square in $\vO_{\F}^{\times}$. This is possible only if $\tau=1$. In other words, if $\tau \neq 1$ and $v_{\F}(y)$ satisfies \eqref{eq: VCond} then $H(\tau y^2) = 0$. We conclude because the only integer not satisfying \eqref{eq: VCond} nor the inequality in (1) is $v_{\F}(y) = 1-e-e^{-1}n_0$.
	
\noindent (4) Suppose $\bL/\F$ is unramified. Let $n := \lceil \tfrac{m}{2} \rceil$. Recall $c_{\beta}$ defined in Lemma \ref{lem: AddParMultChar} (1). We average over the change of variables $\alpha \mapsto \alpha (1+u)(1+\bar{u})^{-1}$ for $u \in \vP_{\bL}^n$ and get
\begin{multline*} 
	\int_{\bL^1 \cap \vO_{\bL}} \beta(\alpha) \psi(y \Tr(\alpha)) \ud \alpha = \int_{\bL^1 \cap \vO_{\bL}} \beta(\alpha) \psi(y \Tr(\alpha)) \left\{ \oint_{\vP_{\bL}^n} \psi \left( \tfrac{2 c_{\beta} \sqrt{\varepsilon} (u-\bar{u})}{\varpi_{\F}^{n_0}} +y \Tr \left( \alpha \cdot \tfrac{u-\bar{u}}{1+\bar{u}} \right) \right) \ud u \right\} \ud \alpha \\
	= \int_{\bL^1 \cap \vO_{\bL}} \beta(\alpha) \psi(y \Tr(\alpha)) \left\{ \oint_{\vO_{\bL}} \psi \left( (2 \sqrt{\varepsilon} c_{\beta} \varpi_{\F}^{n-n_0} + y \varpi_{\F}^n (\alpha - \bar{\alpha}))(u-\bar{u}) \right) \ud u \right\} \ud \alpha.
\end{multline*}
	The non-vanishing of the inner integral implies
\begin{multline*} 
	\left( 2 \sqrt{\varepsilon} c_{\beta} \varpi_{\F}^{n-n_0} + y \varpi_{\F}^n (\alpha - \bar{\alpha}) \right) \sqrt{\varepsilon} \in \vO_{\F} \ \Rightarrow \\
	(\alpha - \bar{\alpha}) \sqrt{\varepsilon} \in 
	\begin{cases} 
		\varpi_{\F}^{m-n_0} \vO_{\F}^{\times} & \text{if } m < 2n_0-1 \\
		\vP_{\F}^{n_0-1} & \text{if } m = 2n_0-1
	\end{cases}
	\Rightarrow \ (\alpha + \bar{\alpha})^2 \in 
	\begin{cases} 
	4 + \vP_{\F}^{2(n_0-1)} & \text{if } m=2n_0-1 \\
	4 + \varpi_{\F}^{2(m-n_0)} \vO_{\F}^{\times} & \text{if } n_0<m<2n_0-1 \\
	\vO_{\F} - 4(1+\vP_{\F}) & \text{if } m=n_0
	\end{cases}
\end{multline*}
	from which we conclude. Replacing $\sqrt{\varepsilon}$ by $1$ we obtain the proof in the split case.
	
\noindent (5) In the case $\tau = 1$ and $\tfrac{n_0}{2}+1 \leq m \leq n_0-1$, we average over the change of variables $\alpha \mapsto \alpha (1+u)(1+\bar{u})^{-1}$ for $u \in \vP_{\bL}^m$, noting that $\alpha \tfrac{u-\bar{u}}{1+\bar{u}} \in \alpha(u-\bar{u})+y^{-1}\vO_{\bL}$, and get
\begin{multline*}
	\int_{\bL^1} \beta(\alpha) \psi(y \Tr(\alpha)) \ud \alpha = \int_{\bL^1} \beta(\alpha) \psi(y \Tr(\alpha)) \left\{ \oint_{\vP_{\bL}^m} \psi \left( \tfrac{2 c_{\beta} (u-\bar{u})}{\varpi_{\bL}^{n_0+1}} + y \Tr \left( \alpha \cdot \tfrac{u-\bar{u}}{1+\bar{u}} \right) \right) \ud u \right\} \ud \alpha \\
	= \int_{\bL^1} \beta(\alpha) \psi(y \Tr(\alpha)) \left\{ \oint_{\vP_{\bL}^m} \psi \left( \left( \tfrac{2c_{\beta}}{\varpi_{\bL}^{n_0+1}} + (\alpha - \bar{\alpha}) y \right) (u-\bar{u}) \right) \ud u \right\} \ud \alpha \\
	= \int_{\bL^1} \beta(\alpha) \psi(y \Tr(\alpha)) \left\{ \oint_{\vP_{\F}^{\lfloor \frac{m}{2} \rfloor}} \psi \left( \left( \tfrac{2c_{\beta}}{\varpi_{\F}^{n_0/2}} + \varpi_{\bL} (\alpha - \bar{\alpha}) y \right) u \right) \ud u \right\} \ud \alpha.
\end{multline*}
	The non-vanishing of the inner integral implies
	$$ 2c_{\beta} \varpi_{\F}^{-\frac{n_0}{2}} + \varpi_{\bL} (\alpha - \bar{\alpha}) y \in \vP_{\F}^{-\lfloor \frac{m}{2} \rfloor} \ \Rightarrow \ 
	\begin{cases}
	(\alpha - \bar{\alpha}) \varpi_{\bL}^{-1} \in \varpi_{\F}^{m-\frac{n_0}{2}-1} \vO_{\F}^{\times} & \text{if } m < n_0 \\
	(\alpha - \bar{\alpha}) \varpi_{\bL}^{-1} \in \varpi_{\F}^{\frac{n_0}{2}-1} \vO_{\F} & \text{if } m = n_0
	\end{cases}. $$
	We conclude by noting that for any $m \in \Z_{\geq 0}$ we have
	$$ (\alpha - \bar{\alpha}) \varpi_{\bL}^{-1} \in \vP_{\F}^m \ \Leftrightarrow \ \Tr(\alpha) \in \pm2(1+\vP_{\F}^{2m+1}), $$
and that $\Tr(\alpha)^2-4 = (\alpha-\bar{\alpha})^2$ never has even valuation.	
\end{proof}

\begin{remark} \label{rmk: TestFROI}
	If $\beta$ is not regular, which can happen only if $\bL/\F$ is split with $\cond(\chi_0^2)=0$, then (2) fails. In fact, it is easy to see that $H(\tau y^2)$ is proportional to $\chi_0(\tau) \norm[y]$ as $\norm[y] \to 0$. In this case we modify the definition of $H$ by \emph{truncating it} so that (2) still holds.
\end{remark}

\begin{corollary} \label{cor: TestFROI}
	The test function $H$ defined by \eqref{eq: TestFROI} and Remark \ref{rmk: TestFROI} is a Bessel orbital integral.
\end{corollary}
\begin{proof}
	This is a direct consequence of Lemma \ref{lem: TestFAsymp} (1) \& (2) and Proposition \ref{prop: BesselOrbInt}.
\end{proof}

	\subsection{Weight Functions}
	
	To every target representation $\pi_0$ we have associated a parameter $(\bL,\beta)$ and a test function $H$ in \eqref{eq: TestFROI}.

\begin{lemma} \label{lem: WtFNonSC}
	(1) In the split case $\pi_0 = \pi(\chi_0,\chi_0^{-1})$ consider the function $\phi_0 \in \Cont_c^{\infty}(\PGL_2(\F))$ given by
	$$ \phi_0 \begin{pmatrix} x_1 & x_2 \\ x_3 & x_4 \end{pmatrix} = \chi_0 \left( \frac{x_4}{x_1} \right) \id_{\gp{Z}\gp{K}_0[\vP_{\F}^{n_0}]} \begin{pmatrix} x_1 & x_2 \\ x_3 & x_4 \end{pmatrix}. $$
	The partial integral $I_0$ defined by
	$$ I_0(g) := \int_{\F} \phi_0 \left( g \begin{pmatrix} 1 & x \\ & 1 \end{pmatrix} \right) \psi(-x) \ud x $$
has support contained in $\gp{Z}\gp{K}_0[\vP_{\F}^{n_0}] \gp{N}(\F)$, and satisfies
	$$ I_0(\kappa n(x)) = \phi_0(\kappa) \psi(x), \quad \forall \kappa \in \gp{Z}\gp{K_0}, x \in \F. $$	

\noindent (2) We can identify $H$ as a Bessel orbital integral
	$$ H(y) = \int_{\F^2} \phi_0 \left( \begin{pmatrix} 1 & x_1 \\ & 1 \end{pmatrix} \begin{pmatrix} & -y \\ 1 & \end{pmatrix} \begin{pmatrix} 1 & x_2 \\ & 1 \end{pmatrix} \right) \psi(-x_1-x_2) \ud x_1 \ud x_2. $$
	
\noindent (3) The weight function $h(\pi)$ is non-negative. We have $h(\pi) \neq 0$ only if $\cond(\pi \otimes \chi_0^{-1}) \leq n_0$, in which case we have a lower bound $h(\pi) \gg q^{-n_0}$.
\end{lemma}
\begin{proof}
	(1) Note that $\phi_0$ is a character upon restriction to $\gp{Z}\gp{K}_0[\vP_{\F}^{n_0}]$. Hence $I_0(g)$ satisfies
	$$ I_0(\kappa g n(x)) = \phi_0(\kappa) \psi(x) I_0(g), \quad \forall \kappa \in \gp{Z}\gp{K_0}, x \in \F. $$
	From the Cartan decomposition $\sideset{}{_{n \in \Z_{\geq 0}}} \bigsqcup \gp{Z} \gp{K} \begin{pmatrix} \varpi_{\F}^n & \\ & 1 \end{pmatrix} \gp{N}(\F)$ we deduce $\mathrm{supp}(I_0) \subset \gp{Z} \gp{K} \gp{N}(\F)$. We have
	$$ \gp{Z} \gp{K} \gp{N}(\F) = \gp{Z}\gp{K}_0[\vP_{\F}^{n_0}] \begin{pmatrix} & 1 \\ 1 & \end{pmatrix} \gp{N}(\F) \sqcup \sideset{}{_{u \in \vP_{\F}/\vP_{\F}^{n_0}}} \bigcup \gp{Z}\gp{K}_0[\vP_{\F}^{n_0}] \begin{pmatrix} 1 & \\ u & 1 \end{pmatrix} \gp{N}(\F) $$
	by the Bruhat decomposition over the residue field of $\F$. We easily verify for $u \notin \vP_{\F}^{n_0}$
	$$ \begin{pmatrix} & 1 \\ 1 & \end{pmatrix} \gp{N}(\F) \cap \gp{Z}\gp{K}_0[\vP_{\F}^{n_0}] = \emptyset, \quad \begin{pmatrix} 1 & \\ u & 1 \end{pmatrix} \gp{N}(\F) \cap \gp{Z}\gp{K}_0[\vP_{\F}^{n_0}] = \emptyset. $$
	Hence we deduce $\mathrm{supp}(I_0) \subset \gp{Z} \gp{K}_0[\vP_{\F}^{n_0}] \gp{N}(\F)$ and conclude by $I_0(\id) = \Vol(\vO_{\F}) = 1$.
	
\noindent (2) Let $\phi^{\iota}(g) := \overline{\phi(g^{-1})}$ for all $\phi \in \Cont_c^{\infty}(\PGL_2(\F))$. We have $\phi_0^{\iota} = \phi_0$ and $\phi_0^{\iota} * \phi_0 = \Vol(\gp{Z} \gp{K}_0[\vP_{\F}^{n_0}]) \phi_0$ since $\phi_0$ is a unitary character upon restriction to $\gp{Z}\gp{K}_0[\vP_{\F}^{n_0}]$. Here the convolution is taken over the group $G := \PGL_2(\F)$. We can rewrite the relevant Bessel orbital integral as
\begin{multline*}
	\Vol(\gp{Z} \gp{K}_0[\vP_{\F}^{n_0}])^{-1} \int_{\F^2} \int_G \overline{\phi_0(g)} \phi_0 \left( g \begin{pmatrix} 1 & x_1 \\ & 1 \end{pmatrix} \begin{pmatrix} & -y \\ 1 & \end{pmatrix} \begin{pmatrix} 1 & x_2 \\ & 1 \end{pmatrix} \right) \psi(-x_1-x_2) \ud g \ud x_1 \ud x_2 \\
	= \Vol(\gp{Z} \gp{K}_0[\vP_{\F}^{n_0}])^{-1} \int_G \\
	\left\{ \int_{\F} \overline{\phi_0\left( g \begin{pmatrix} 1 & -x_1 \\ & 1 \end{pmatrix} \right)} \psi(-x_1) \ud x_1 \right\} \cdot \left\{ \int_{\F} \phi_0 \left( g \begin{pmatrix} & -y \\ 1 & \end{pmatrix} \begin{pmatrix} 1 & x_2 \\ & 1 \end{pmatrix} \right) \psi(-x_2) \ud x_2 \right\} \ud g \\
	= \Vol(\gp{Z} \gp{K}_0[\vP_{\F}^{n_0}])^{-1} \int_G \overline{I_0(g)} I_0 \left( g \begin{pmatrix} & -y \\ 1 & \end{pmatrix} \right) \ud g = \int_{\F} \psi(-u) I_0 \left( \begin{pmatrix} 1 & u \\ & 1 \end{pmatrix} \begin{pmatrix} & -y \\ 1 & \end{pmatrix} \right) \ud u.
\end{multline*}
	Considering the support of $I_0$, the above integral is non-vanishing only if for some $u' \in \F$
	$$ \begin{pmatrix} 1 & u \\ & 1 \end{pmatrix} \begin{pmatrix} & -y \\ 1 & \end{pmatrix} \begin{pmatrix} 1 & u' \\ & 1 \end{pmatrix} = \begin{pmatrix} u & uu'-y \\ 1 & u' \end{pmatrix} \in \gp{Z} \gp{K}_0[\vP_{\F}^{n_0}], $$
	implying $y \in \varpi_{\F}^{-2n} \vO_{\F}^{\times}$ for some $n \geq n_0$, and $u=\varpi^{-n}u_1$, $u' = \varpi^{-n}u_2$ for some $u_1,u_2 \in \vO_{\F}^{\times}$ satisfying $u_1u_2-\varpi^{2n}y \in \vP_{\F}^n$. Writing $y = \varpi_{\F}^{-2n}y_0$ for some $y_0 \in \vO_{\F}^{\times}$ we may take $u_2 = u_1^{-1} y_0$. We recognize the Bessel orbital integral as $H(y)$ by the following equation, which is easy to verify
	$$ q^n \int_{\vO_{\F}^{\times}} \chi_0(u_1^{-2} y_0) \psi \left( - \tfrac{u_1 + u_1^{-1}y_0}{\varpi_{\F}^n} \right) \ud u_1 = H(y). $$
	
\noindent (3) Up to notation this is due to Nelson \cite[Theorem 3.1]{Ne20}. Our former work \cite[Lemma 4.1]{BFW21+} contains more details in a similar situation.
\end{proof}

\begin{lemma} \label{lem: MCSuperCuspBessel}
	Let $\pi$ be a supercuspidal representation of $\GL_2(\F)$. Let $C(g)$ be a matrix coefficient for smooth vectors in $\pi$. Take a non-trivial additive character $\psi$ of $\F$. Then the (relative) orbital integral
	$$ I(g) := \int_{\F^2} C(n(u_1)gn(u_2)) \psi(-u_1-u_2) du_1 du_2 $$
	is equal to the $\psi$-Bessel function of $\pi$ up to a constant.
\end{lemma}
\begin{proof}
	Assume $C(g) = \Pairing{\pi(g)v_1}{v_2}$ for smooth vectors $v_1$ and $v_2$. One checks that the functional
	$$ V_{\pi}^{\infty} \to \C, \quad v \mapsto \int_{\F} \Pairing{\pi(n(u_1)).v}{v_2} \psi(-u_1) du_1 $$
	is a $\psi$-Whittaker functional. Hence the following function
	$$ W_{v_1}(g) := \int_{\F} C(n(u_1)g) \psi(-u_1) du_1 $$
	is the/a Whittaker function of $v_1$. We then apply \cite[Lemma 4.1]{So84} to get
	$$ I(g) = \BesselF_{\pi,\psi}(g) \cdot W_{v_1}(\id), \quad \forall g \in \gp{B}(\F)w_2\gp{B}(\F). $$
	The integral converges absolutely because $C$ is smooth of compact support modulo the center.
\end{proof}

\begin{lemma} \label{lem: WtFSC}
	In the dihedral case $\pi_0=\pi_{\beta}$ the weight function $h(\pi) \neq 0$ is non-vanishing only if $\pi \simeq \pi_{\beta}$, in which case it is positive and satisfies a lower bound $h(\pi_{\beta}) \gg q^{-e^{-1}n_0+1-e}$.
\end{lemma}
\begin{proof}
	Up to some constant approximately equal to $1$ we recognize $H(y)$ as $\BesselF_{\pi_{\beta},\psi} \begin{pmatrix} & -y \\ 1 & \end{pmatrix}$ by \cite[Theorem 1.1]{BS11}. By Lemma \ref{lem: MCSuperCuspBessel} there exists some $\phi_0 \in \Cont_c^{\infty}(\GL_2(\F))$, whose integral along the center $\phi_1 \in \Cont_c^{\infty}(\PGL_2(\F))$ is a matrix coefficient of $\pi_{\beta}$, so that $H$ is the Bessel orbital integral of $\phi_0$ as given in Lemma \ref{lem: WtFNonSC} (1). Now that $\pi(\phi_1) \neq 0$ only if $\pi \simeq \pi_{\beta}$ by Schur's lemma, it remains to show that $h(\pi_{\beta})$ is positive with the stated lower bound. By \eqref{eq: WtFviaBOI} the value $h(\pi_{\beta})$ is the square of the $\intL^2$-norm of $H(y)$ up to a constant essentially equal to $1$. Taking Lemma \ref{lem: TestFAsymp} (2) into account we have
\begin{align*}
	\int_{\F^{\times}} \extnorm{H(y)}^2 \tfrac{\ud^{\times} y}{\norm[y]_{\F}} &= \frac{1}{2} \sum_{\tau} \int_{\F^{\times}} \extnorm{H(\tau y^2)}^2 \tfrac{\ud^{\times} y}{\norm[\tau y^2]_{\F}} \\
	&= \frac{1}{2} \sum_{\tau} \int_{\F^{\times}} \left( \int_{\bL^1 \times \bL^1} \beta(x_{\tau}\alpha_1) \overline{\beta(x_{\tau} \alpha_2)} \psi_{\bL}(x_{\tau}(\alpha_1-\alpha_2)y) \ud \alpha_1 \ud \alpha_2 \right) \ud^{\times} y \\
	&= \frac{1}{2} \sum_{n \geq (n_0-2)e^{-1}+2} \int_{\bL^1} \beta(\alpha) \left( \sum_{\tau} \int_{\bL^1} \int_{\varpi_{\F}^{-n} \vO_{\F}^{\times}} \psi_{\bL}(x_{\tau}\alpha_2 y(\alpha-1)) \ud^{\times} y \ud \alpha_2 \right) \ud \alpha,
\end{align*}
	where we applied the change of variables $\alpha_1 \mapsto \alpha \alpha_2$ in the last line. By measure relation \eqref{eq: MeasComp} we get
	$$ \frac{1}{2} \sum_{\tau} \int_{\bL^1} \int_{\varpi_{\F}^{-n} \vO_{\F}^{\times}} \psi_{\bL}(x_{\tau}\alpha_2 y(\alpha-1)) \ud^{\times} y \ud \alpha_2 = \zeta_{\F}(1) \int_{-en \leq v_{\bL}(z) < e-en} \psi_{\bL}(z(\alpha-1)) \frac{\ud z}{\norm[z]_{\bL}}. $$
	Therefore we continue the calculation as
\begin{align*}
	\int_{\F^{\times}} \extnorm{H(y)}^2 \tfrac{\ud^{\times} y}{\norm[y]_{\F}} &= \zeta_{\F}(1) \sum_{n \geq n_0-1+e} \int_{\bL^1} \beta(\alpha) \left( \int_{\varpi_{\bL}^{-n}\vO_{\bL}^{\times}} \psi_{\bL}(z(\alpha-1)) \tfrac{\ud z}{\norm[z]_{\bL}} \right) \ud \alpha \\
	&= \tfrac{\zeta_{\F}(1)}{\zeta_{\bL}(1)} q_{\bL}^{-\frac{e-1}{2}} \sum_{n \geq n_0-1+e} \int_{\bL^1} \beta(\alpha) \left( \id_{\vP_{\bL}^{n+1-e}}(\alpha-1) - q_{\bL}^{-1} \id_{\vP_{\bL}^{n-e}}(\alpha-1) \right) \ud \alpha \\
	&= \tfrac{\zeta_{\F}(1)}{\zeta_{\bL}(1)^2} q_{\bL}^{-\frac{e-1}{2}} \sum_{n \geq n_0} \Vol(\bL^1 \cap (1+\vP_{\bL}^n)),
\end{align*}
	and conclude the lower bound by Proposition \ref{prop: MeasNorm} (3) and the fact $e \mid n_0$ (see Claim in Lemma \ref{lem: CondReInterp}).
\end{proof}
\begin{remark}
	By \eqref{eq: CondRelDih}, the bounds in Lemma \ref{lem: WtFNonSC} (1) \& \ref{lem: WtFSC} can be rewritten as $h(\pi_0) \gg q^{-\lceil \frac{\cond(\pi_0)}{2} \rceil}$.
\end{remark}

\section{Dual Weight Functions: Common Reductions}

	\subsection{Infinite Part}
	
	Recall the constant $a(\Pi)$ for the stability range defined in Proposition \ref{prop: StabRang}. Take the ``infinite part'' of the test function as $H_{\infty} = E_{\geq n_1}$ for some
\begin{equation} \label{eq: StabPar}
	n_1 \geq \max(2n_0/e+e-1, a(\Pi)) (\geq 2).
\end{equation}

\noindent Recall the formula for the dual weight function
\begin{equation} \label{eq: DWtFviaBOI}
	\widetilde{h}(\chi) = \int_{\F^{\times}} \widetilde{\VorH}_{\Pi,\psi} \circ \Mult_{-1}(H)(t) \cdot \psi(-t) \chi^{-1}(t) \norm[t]^{-\frac{1}{2}} \ud^{\times} t.
\end{equation}

\begin{lemma} \label{lem: DWtInfty}
	The partial dual weight function
	$$ \widetilde{h}_{\infty}(\chi) := \int_{\F^{\times}} \widetilde{\VorH}_{\Pi,\psi} \circ \Mult_{-1}(H_{\infty})(t) \cdot \psi(-t) \chi^{-1}(t) \norm[t]^{-\frac{1}{2}} \ud^{\times} t $$
is non-vanishing only if $\chi$ is unramified, in which case we have
	$$ \widetilde{h}_{\infty}(\chi) = \frac{\chi(\varpi_{\F})^{n_1} q^{-\frac{n_1}{2}}}{1-\chi(\varpi_{\F})q^{-\frac{1}{2}}}. $$
\end{lemma}
\begin{proof}
	By Proposition \ref{prop: StabRang} we have
	$$ \widetilde{\VorH}_{\Pi,\psi} \circ \Mult_{-1}(E_{\geq n_1})(t) = \psi(t) \cdot \id_{\vP_{\F}^{n_1}}(t^{-1}). $$
	The desired formula follows readily from the above one.
\end{proof}

	Taking into account Lemma \ref{lem: TestFAsymp} (2) we introduce the ``finite part'' of the test function as
\begin{equation} \label{eq: TestFComp}
	H_c := H-H_{\infty} = \sideset{}{_{\substack{n=\frac{2}{e}(n_0-2)+3 \\ 2 \mid en}}^{2n_1-1}} \sum H_n, \qquad H_n := H \cdot \id_{\varpi_{\F}^{-n} \vO_{\F}^{\times}}.
\end{equation}
	Since each function $H_n \in \Cont_c^{\infty}(\F^{\times})$, we may replace the extended Voronoi--Hankel transform with its original version in the dual weight formula and turn to the study of
\begin{equation} \label{eq: DWtFLeveln}
	\widetilde{h}_n(\chi) := \int_{\F^{\times}} \VorH_{\Pi,\psi} \circ \Mult_{-1}(H_n)(t) \cdot \psi(-t) \chi^{-1}(t) \norm[t]^{-\frac{1}{2}} \ud^{\times} t.
\end{equation}
	We divide the domain of integration into two parts: $\vO_{\F}-\{ 0 \}$ and $\F-\vO_{\F}$, giving the decomposition
\begin{equation} \label{eq: DWtFLevelnDecomp}
	\widetilde{h}_n(\chi) = \widetilde{h}_n^+(\chi) + \widetilde{h}_n^-(\chi). 
\end{equation}
	Accordingly the dual weight functions (see \eqref{eq: NWt}) have the following decomposition
\begin{equation} \label{eq: DWtFDecomp}
	\widetilde{h}(\chi) = \widetilde{h}_{\infty}(\chi) + \widetilde{h}_{c}(\chi), \quad \widetilde{h}_{c}(\chi) = \widetilde{h}_{c}^+(\chi) + \widetilde{h}_{c}^-(\chi), \quad \widetilde{h}_{c}^{\pm}(\chi) = \sideset{}{_{\substack{n=\frac{2}{e}(n_0-2)+3 \\ 2 \mid en}}^{2n_1-1}} \sum \widetilde{h}_n^{\pm}(\chi),
\end{equation}
\begin{equation} \label{eq: NDWtFDecomp}
	\widetilde{H}(\chi) = \widetilde{H}_{\infty}(\chi) + \widetilde{H}_{c}(\chi), \quad \widetilde{H}_{c}(\chi) = \widetilde{H}_{c}^+(\chi) + \widetilde{H}_{c}^-(\chi), \quad \widetilde{H}_{c}^{\pm}(\chi) = \sideset{}{_{\substack{n=\frac{2}{e}(n_0-2)+3 \\ 2 \mid en}}^{2n_1-1}} \sum \widetilde{H}_n^{\pm}(\chi).
\end{equation}
	We call those (partial) dual weight functions with ``$+$'' (resp. ``$-$'') the \emph{postive} (resp. \emph{negative}) part.

	\subsection{Positive Part}
	
	We first study $\widetilde{h}_n^+(\chi)$. Since $\psi$ is trivial on $\vO_{\F}$, we have $\widetilde{h}_n^+(\chi) = \widetilde{h}_n(1/2, \chi)$ for
\begin{equation} \label{eq: DWtFLevelnVar}
	\widetilde{h}_n(s, \chi) := \int_{\vO_{\F}-\{ 0 \}} \VorH_{\Pi,\psi} \circ \Mult_{-1}(H_n)(t) \cdot \chi^{-1}(t) \norm[t]^{-s} \ud^{\times} t,
\end{equation}
	which is simply the partial sum of non-negative powers of $X = q^s$ in the Laurent series expansion of
\begin{equation} \label{eq: DWtFLnLaurent}
	f_n(q^s; \chi,H) = \int_{\F^{\times}} \VorH_{\Pi,\psi} \circ \Mult_{-1}(H_n)(t) \cdot \chi^{-1}(t) \norm[t]^{-s} \ud^{\times} t. 
\end{equation}
	By the local functional equation we have (see Lemma \ref{lem: TestFMellin})
\begin{equation} \label{eq: DWtFLnLF}
	f_n(q^s; \chi,H) = \varepsilon_n(\chi, H) q^{n(s-1)} \cdot \varepsilon \left( s, \Pi \otimes \chi, \psi \right) \frac{L(1-s, \widetilde{\Pi} \otimes \chi^{-1})}{L(s, \Pi \otimes \chi)}.
\end{equation}
	To relate $f_n(q^s; \chi,H)$ with $\widetilde{h}_n(s, \chi)$ we need the following crucial lemma.

\begin{lemma} \label{lem: LaurentTruncF}
	Let $f(X) = \sideset{}{_{n > -\infty}} \sum a_n X^n$ be a Laurent series converging in $0 < \norm[X] < \rho$ for some $\rho > 1$. Let $f_+(X) = \sideset{}{_{n > -\infty}} \sum a_n X^n$. Let $D = X \tfrac{\ud}{\ud X}$. Assume $f(X)$ (hence $f_+(X)$) has a meromorphic continuation to $X \in \C$. 
\begin{itemize}
	\item[(0)] For any $X$ and any $\epsilon < \min(1, \rho \norm[X]^{-1})$ we have the relation
	$$ f_+(X) = f(X) - \int_{\norm[z]=\epsilon} \frac{f(Xz)}{1-z} \frac{\ud z}{2\pi i}. $$
	\item[(1)] For $0 < \norm[X] < \rho$ and any $1<r<\rho \norm[X]^{-1}$ we have the relation
	$$ f_+(X) = \int_{\norm[z]=r} \frac{f(Xz)}{z-1} \frac{\ud z}{2\pi i} = f(X) + \int_{\norm[z]=1} \frac{f(Xz)-f(X)}{z-1} \frac{\ud z}{2\pi i}. $$
	Consequently, for $0 < \norm[X] < \rho$ and $n \in \Z_{\geq 0}$ we have the bound
	$$ \extnorm{(D^n f_+)(X)} \ll \extnorm{(D^n f)(X)} + \sideset{}{_{\norm[z]=1}} \sup \extnorm{(D^{n+1} f)(Xz)}. $$
	\item[(2)] Suppose $f(X) = Q(X) P(X)^{-1}$ for some $Q \in \C[X,X^{-1}], P \in \C[X]$. Let
	$$ P(X) = \sideset{}{_{j=1}^r} \prod (1-b_j X)^{m_j} $$
	with $m_j \in \Z_{\geq 1}$ and distinct $b_j \neq 0$. Introduce
	$$ P_j(X) := \sideset{}{_{i \neq j}} \prod (1-b_i X)^{m_i}, \quad C_{j,k} := \frac{(-1)^k}{k! \cdot b_j^k} \left( \frac{Q}{P_j} \right)^{(k)}(b_j^{-1}). $$
	Suppose the highest power $X^m$ in $Q$ satisfies $m < \deg P$. Then we have
	$$ P(X) f_+(X) = \sideset{}{_{j=1}^r} \sum \sideset{}{_{k=0}^{m_j-1}} \sum C_{j,k} \cdot (1-b_jX)^k P_j(X). $$
\end{itemize}
\end{lemma}
\begin{proof}
	(0) The stated formula follows from the residue theorem via
	$$ \int_{\norm[z]=\epsilon} \frac{f(Xz)}{1-z} \frac{\ud z}{2\pi i} = \int_{\norm[z]=\epsilon} \left( \sum_{n > -\infty} a_n X^n z^n \right) \left( \sum_{k=0}^{\infty} z^k \right) \frac{\ud z}{2\pi i} = \sum_{n < 0} a_n X^n, $$
	since only the terms for $n+k=-1$ give non-zero contribution.

\noindent (1) Clearly the stated formula follows from the one in (0) by the residue theorem. To see the bound, we introduce $g(z) = f(Xz)$ and rewrite the integral as
	$$ \int_{\norm[z]=1} \frac{f(Xz)-f(X)}{z-1} \frac{\ud z}{2\pi i} = \int_{-\pi}^{\pi} \frac{g(e^{i\theta})-g(1)}{e^{i\theta}-1} \frac{\ud \theta}{2\pi} = \int_{-\pi}^{\pi} \frac{i}{e^{i\theta}-1} \left( \int_0^{\theta} g'(e^{it}) e^{it} \ud t \right) \frac{\ud \theta}{2\pi} . $$
	Note that $z g'(z) = Df(Xz)$. Therefore we obtain the formula
	$$ f_+(X) = f(X) + \int_{-\pi}^{\pi} \frac{i}{e^{i\theta}-1} \left( \int_0^{\theta} (Df)(Xe^{it}) \ud t \right) \frac{\ud \theta}{2\pi}. $$
	Applying $D^n$ on both sides we get a formula and conclude the stated bound as
\begin{multline*} 
	(D^n f_+)(X) = (D^n f)(X) + \int_{-\pi}^{\pi} \frac{i}{e^{i\theta}-1} \left( \int_0^{\theta} (D^{n+1}f)(Xe^{it}) \ud t \right) \frac{\ud \theta}{2\pi} \\
	\ll \extnorm{(D^n f)(X)} + \sideset{}{_{\norm[z]=1}} \sup \extnorm{(D^{n+1} f)(Xz)} \cdot \int_{-\pi}^{\pi} \extnorm{\frac{\theta}{e^{i\theta}-1}} \frac{\ud \theta}{2\pi}.
\end{multline*}
	
\noindent (2) We may assume $0 < \norm[X] < \rho$ and depart from the first equation in (1), then move the contour to $\norm[z]=r$ for $r \gg 1$, picking up the residues. The contour integral tends to $0$ as $r \to +\infty$ due to the assumption $m < \deg P$. Multiplying the resulted formula by $P(X)$, we get the stated formula.
\end{proof}

\begin{definition} \label{def: ExpChar}
	For any generic irreducible representation $\Pi$ of $\GL_r(\F)$, let $d(\Pi)$ be the degree of the $L$-function $L(s,\Pi)$, namely the degree of the polynomial $L(s,\Pi)^{-1}$ in $X:=q^{-s}$. Write
	$$ \condL(\Pi) := \cond(\Pi)+d(\Pi). $$
	Introduce the set of \emph{exponents} of $\Pi$ by
	$$ \mathrm{E}(\Pi) := \left\{ \xi \in \widehat{\vO_{\F}^{\times}} \ \middle| \ d(\Pi \otimes \xi) > 0 \right\}. $$
\end{definition}
\begin{lemma} \label{lem: ExpCharBds}
	Let $\Pi$ be a generic irreducible representation of $\GL_3(\F)$. We have $\condL(\Pi) \geq 3$, $\extnorm{\mathrm{E}(\Pi)} \leq 3$. For any $\xi \in \mathrm{E}(\Pi)$ we have the bounds 
	$$ \cond(\xi) \leq \cond(\Pi), \quad \cond(\Pi \otimes \xi) \leq 2 \cond(\Pi), \quad \condL(\Pi \otimes \xi) \leq 2 \cond(\Pi)+3. $$
\end{lemma}
\begin{proof}
	The upper bound for $\condL(\Pi \otimes \xi)$ obviously follows from the one for $\cond(\Pi \otimes \xi)$ since $d(\Pi \otimes \xi) \leq 3$ in any case. For the other bounds we distinguish cases for $\Pi$.

\noindent (1) If $\cond(\Pi)=0$, then $\Pi$ is spherical, $\mathrm{E}(\Pi)=\{ \id \}$ and $d(\Pi)=3$. The stated bounds clearly hold.

\noindent (2) If $\cond(\Pi) > 0$ and if $\Pi = \chi_1 \boxplus \chi_2 \boxplus \chi_3$ is induced from the Borel subgroup, then we have
	$$ \condL(\Pi) = \sideset{}{_{i=1}^3} \sum \condL(\chi_i) = \sideset{}{_{i=1}^3} \sum \max(\cond(\chi_i),1) \geq 3. $$
 	We also have $\mathrm{E}(\Pi) \subset \{ \xi_1^{-1}, \xi_2^{-1}, \xi_3^{-1} \}$ with $\chi_j \mid_{\vO_{\F}^{\times}} = \xi_j$, from which we deduce the bound for $\extnorm{\mathrm{E}(\Pi)}$. Obviously we have $\cond(\xi_j) \leq \cond(\Pi)$. Taking $\xi = \xi_1^{-1}$ for example, we have the second stated bound as
	$$ \cond(\Pi \otimes \xi) = \cond(\chi_2\xi_1^{-1}) + \cond(\chi_3\xi_1^{-1}) \leq 2 \max \{ \cond(\chi_j): 1 \leq j \leq 3 \} \leq 2 \cond(\Pi). $$
	
\noindent (3) If $\cond(\Pi) > 0$ and if $\Pi = \pi \boxplus \chi$ for some supercuspidal $\pi$ of $\GL_2(\F)$, then the central character of $\pi$ is $\chi^{-1}$. We have $\condL(\Pi) = \cond(\pi)+\max(\cond(\chi),1) \geq 2+1= 3$. We also have $\mathrm{E}(\Pi) = \{ \chi^{-1} \mid_{\vO_{\F}^{\times}} \}$. For $\xi = \chi^{-1} \mid_{\vO_{\F}^{\times}}$ we have $\cond(\xi) = \cond(\chi) \leq \cond(\Pi)$ and get the second stated bound via $\pi \otimes \chi^{-1} \simeq \widetilde{\pi}$ as
	$$ \cond(\Pi \otimes \xi) = \cond(\pi \otimes \chi^{-1}) = \cond(\widetilde{\pi}) = \cond(\pi) \leq \cond(\Pi). $$
	
\noindent (4) If $\cond(\Pi) > 0$ and if $\Pi = \mathrm{St}_{\eta} \boxplus \chi$, then $\chi \eta^2 = \id$. By \cite[\S 8 Proposition \& \S 10 Proposition]{Ro94} we have
	$$ \condL(\Pi) = 
	\begin{cases}
		2 \cond(\eta) & \text{if } \cond(\eta) \geq 1 \\
		1 & \text{if } \cond(\eta)=0
	\end{cases}
	+ d(\eta) + \max(\cond(\chi),1) \geq 3. $$
 	We also have $\mathrm{E}(\Pi) = \{ \chi^{-1} \mid_{\vO_{\F}^{\times}}, \eta^{-1} \mid_{\vO_{\F}^{\times}} \}$. For $\xi = \chi^{-1} \mid_{\vO_{\F}^{\times}}$ we argue the same as in the proof of (3). For $\xi = \eta^{-1} \mid_{\vO_{\F}^{\times}}$ we have
	$$ \cond(\xi) = \cond(\eta) < 1+\cond(\eta) = \cond(\Pi \otimes \xi) \leq \max(2\cond(\eta),1) + \cond(\chi) = \cond(\Pi). $$
	
\noindent (5) If $\cond(\Pi) > 0$ and if $\Pi = \mathrm{St}_{\eta}$ is a twist of the Steinberg representation of $\PGL_3(\F)$, then by \cite[\S 8 Proposition \& \S 10 Proposition]{Ro94} we have
	$$ \condL(\Pi) = \begin{cases}
		3 \cond(\eta) & \text{if } \cond(\eta) \geq 1 \\
		2 & \text{if } \cond(\eta)=0
	\end{cases} + d(\eta) \geq 3. $$
	We also have $\mathrm{E}(\Pi) = \{ \eta^{-1} \mid_{\vO_{\F}^{\times}} \}$. For $\xi = \eta^{-1} \mid_{\vO_{\F}^{\times}}$ we have
	$$ \cond(\xi) = \cond(\eta) \leq \max(3\cond(\eta),2) = \cond(\Pi), \quad \cond(\Pi \otimes \xi) = 2 \leq \cond(\Pi). $$
	
\noindent (6) If $\Pi$ is supercuspidal, then $\mathrm{E}(\Pi) = \emptyset$. We have $\condL(\Pi) = \cond(\Pi) \geq 3$.
\end{proof}

\noindent The following subset of $\widehat{\vO_{\F}^{\times}}$ will turn out to be important for the bound of $\widetilde{h}_n^+(\chi)$:
\begin{equation}
	\mathcal{A}_n = \mathcal{A}_n(\beta, \Pi) := \left\{ \xi \in \widehat{\vO_{\F}^{\times}} \ \middle| \ \condL(\Pi \otimes \xi) \leq n, \ \varepsilon_n(\xi,H) \neq 0 \right\}.
\end{equation}

\begin{lemma} \label{lem: AnAnal}
	For any $\xi \in \mathcal{A}_n$ we have $\cond(\xi) \leq \max \left( \tfrac{n_0}{e}, 2\cond(\Pi) \right)$.
\end{lemma}
\begin{proof}
	Suppose $\xi \in \mathcal{A}_n$ and $\cond(\xi) > \max \left( \tfrac{n_0}{e}, 2\cond(\Pi) \right)$. In particular, we have $\cond(\xi) \geq a(\Pi)$, the stability barrier of $\Pi$, by Proposition \ref{prop: StabRang}. Therefore we get $\cond(\Pi \otimes \xi) = 3\cond(\xi)$, $d(\Pi \otimes \xi) = 0$ and deduce
\begin{equation} \label{eq: CharAnCondBd1} 
	3\cond(\xi) = \condL(\Pi \otimes \xi) \leq n. 
\end{equation}
	On the other hand, we claim that the condition $\cond(\xi) > \tfrac{n_0}{e}$ and $\varepsilon_n(\xi,H) \neq 0$ imply
\begin{equation} \label{eq: CharAnCondBd2} 
	\cond(\xi) = \tfrac{n}{2}, 
\end{equation}
	and conclude the proof by comparing \eqref{eq: CharAnCondBd1} and \eqref{eq: CharAnCondBd2} taking into account $n \geq \tfrac{2}{e}(n_0-2)+3 >0$. In fact, the proof of \eqref{eq: CharAnCondBd2} is case-by-case check with Lemma \ref{lem: TestFMellin}: (1) If $\bL/\F$ is split, then $e=1$ and $\cond(\xi) > n_0 = \cond(\chi_0^{\pm 1})$. Only the first case in Lemma \ref{lem: TestFMellin} is possible, yielding \eqref{eq: CharAnCondBd2}. (2) If $\bL/\F$ is unramified, then $e=1$ and by \eqref{eq: CondCompNormF} we have $\cond(\xi \circ \Nr) = \cond(\xi) > n_0 = \cond(\beta)$. The last case in Lemma \ref{lem: TestFMellin} yields \eqref{eq: CharAnCondBd2}. (3) If $\bL/\F$ is ramified, then $e=2$ and by \eqref{eq: CondCompNormF} we have $\cond(\xi \circ \Nr) \geq 2\cond(\xi)-1 \geq 2(n_0/2+1)-1 = n_0+1 > \cond(\beta)$. The last case in Lemma \ref{lem: TestFMellin} yields \eqref{eq: CharAnCondBd2}.
\end{proof}

\begin{lemma} \label{lem: DWt+Bd}
	Write $\xi = \chi \mid_{\vO_{\F}^{\times}}$. For any $\epsilon > 0$ sufficiently small we have the bound
\begin{multline*}
	\extnorm{\widetilde{h}_c^+(\chi)} \leq \sideset{}{_{\frac{2}{e}(n_0-2)+3 \leq n < \condL(\Pi \otimes \xi)}} \sum \extnorm{\widetilde{h}_n^+(\chi)} + \sideset{}{_{n \geq \condL(\Pi \otimes \chi)}} \sum \extnorm{\widetilde{h}_n^+(\chi)} \\
	\ll_{\epsilon} \Cond(\Pi)^{\epsilon} q^{-\frac{n_0}{e}+\frac{1-e}{2}} \cdot \id_{\leq \max \left( \frac{n_0}{e},2\cond(\Pi) \right)}(\cond(\xi)).
\end{multline*}
\end{lemma}
\begin{proof}
	Write the relevant $L$-functions of $\Pi \otimes \chi$ as
	$$ L(s, \Pi \otimes \chi) = \sideset{}{_{j=1}^d} \prod (1-a_j q^{-s})^{-1} $$
where $d=d(\Pi \otimes \chi) \in \{ 0,1,2,3 \}$ is the degree, and $a_j=a_j(\Pi \otimes \chi)$ are the \emph{generalized Satake parameters} of $\Pi \otimes \chi$ satisfying $\norm[a_j] \in \{ 1, q^{-\frac{1}{2}}, q^{-1} \}$ by temperedness. We can rewrite \eqref{eq: DWtFLnLF} as
\begin{equation} \label{eq: VHTestFLevelnMellin}
	f_n(X; \chi,H) = \varepsilon_n(\chi, H) \varepsilon \left( \tfrac{1}{2}, \Pi \otimes \chi, \psi \right) q^{\frac{\cond(\Pi \otimes \chi)}{2}-n} X^{n-\condL(\Pi \otimes \chi)} \cdot \sideset{}{_{j=1}^d} \prod \tfrac{X-a_j}{1-\overline{a_j}q^{-1}X}. 
\end{equation}
	
\noindent If $n \geq \condL := \condL(\Pi \otimes \chi)$, then $f_n(X; \chi,H) \neq 0$ implies $\varepsilon_n(\chi,H) \neq 0$, hence $\chi \mid_{\vO_{\F}^{\times}} \in \mathcal{A}_n$ by definition. The Laurent expansion of $f_n(X; \chi,H)$ contains only non-negative powers of $X$. Therefore the right hand side is equal to $\widetilde{h}_n(s, \chi)$ by definition \eqref{eq: DWtFLevelnVar}. Consequently we get the formula and deduce the bound
	$$ \widetilde{h}_n(s, \chi) = \varepsilon_n(\chi, H) \varepsilon \left( \tfrac{1}{2}, \Pi \otimes \chi, \psi \right) q^{\frac{\cond(\Pi \otimes \chi)}{2}-n} q^{s(n-\cond(\Pi \otimes \chi))} \cdot \sideset{}{_{j=1}^d} \prod \tfrac{1-a_j q^{-s}}{1-\overline{a_j}q^{s-1}}, $$
\begin{equation} \label{eq: RedBd1}
	\extnorm{\widetilde{h}_n^+(\chi)} = \extnorm{\widetilde{h}_n(1/2, \chi)} \ll_{\RamCst_3} \id_{\mathcal{A}_n}(\chi \mid_{\vO_{\F}^{\times}}) \cdot q^{-\frac{n}{2}}. 
\end{equation}
	
\noindent If $n < \condL$, then the Laurent expansion of $f_n(X; \chi,H)$ contains non-negative powers of $X$ only if $d=d(\Pi \otimes \chi) > 0$, i.e. $\chi \mid_{\vO_{\F}^{\times}} \in \mathrm{E}(\Pi)$ where $\mathrm{E}(\Pi)$ is the set of exponents of $\Pi$ introduced in Definition \ref{def: ExpChar}. By Lemma \ref{lem: ExpCharBds} and the summation range of $n$ we have
\begin{equation} \label{eq: ExpNumBds} 
	\tfrac{2}{e}(n_0-2)+3 \leq n \leq 2 \cond(\Pi)+2. 
\end{equation}
	 Note that the Laurent expansion of $f_n(X; \chi,H)$ is absolutely convergent for $\norm[X] < q$. Note also that $Xf_n'(X; \chi, H)$ and $f_n(X; \chi,H)$ have the same type of bound for $\norm[X] = q^{\frac{1}{2}}$. By Lemma \ref{lem: LaurentTruncF} (1) and \eqref{eq: DWtFLevelnVar} we deduce
\begin{multline} \label{eq: RedBd2}
	\extnorm{\widetilde{h}_n^+(\chi)} = \extnorm{\widetilde{h}_n(1/2, \chi)} \leq \extnorm{f_n(q^{\frac{1}{2}}; \chi, H)} + \sideset{}{_{\norm[X] = q^{\frac{1}{2}}}} \sup \extnorm{Xf_n'(X; \chi,H)} \\
	\ll \id_{\frac{2}{e}(n_0-2)+3 \leq n \leq \condL(\Pi \otimes \chi)-1} \id_{\mathrm{E}(\Pi)}(\chi \mid_{\vO_{\F}^{\times}}) \cdot q^{-\frac{n}{2}} \cdot \left( 1 + \condL(\Pi \otimes \chi) - n \right) \\
	\ll \id_{\frac{2}{e}(n_0-2)+3 \leq n \leq 2 \cond(\Pi)+2} \id_{\mathrm{E}(\Pi)}(\chi \mid_{\vO_{\F}^{\times}}) \cdot q^{-\frac{n}{2}} \cdot (2\cond(\Pi)+3).
\end{multline}
	by Lemma \ref{lem: TestFMellin} \& \ref{lem: ExpCharBds}. We conclude by summing over $n$ (for $2 \mid en$) the bounds \eqref{eq: RedBd1} and \eqref{eq: RedBd2}, taking into account the bound for $\id_{\mathcal{A}_n}$ given by Lemma \ref{lem: AnAnal}.
\end{proof}

	\subsection{Negative Part}

	We turn to the study of a ``trivial'' bound of $\widetilde{h}_n^-(\chi)$. We introduce
\begin{equation}
	\mathcal{B}_n = \mathcal{B}_n(\beta, \Pi) := \left\{ \xi \in \widehat{\vO_{\F}^{\times}} \ \middle| \ \condL(\Pi \otimes \xi) > n, \ \varepsilon_n(\xi,H) \neq 0 \right\}.
\end{equation}

\begin{lemma} \label{lem: BnAnal}
	Recall $e=e(\bL/\F)$. We have $\mathcal{B}_n \subset \left\{ \xi \in \widehat{\vO_{\F}^{\times}} \ \middle| \ \cond(\xi) \leq n/2 \right\}$ and the bound $\extnorm{\mathcal{B}_n} \ll q^{\frac{n}{2}}$.
\end{lemma}
\begin{proof}
	For $\xi \in \mathcal{B}_n$, the condition $\varepsilon_n(\xi,H) \neq 0$ implies $\cond(\xi) \leq \tfrac{n}{2}$ by checking Lemma \ref{lem: TestFMellin} case-by-case, taking into account \eqref{eq: CondCompNormF}. For example in the case $\bL/\F$ is ramified, we have $n \geq n_0+1$ and
	$$ 2\cond(\xi)-1 \leq \cond(\xi \circ \Nr) \leq \max \left( \cond(\beta), \cond(\beta \cdot (\xi \circ \Nr)) \right) = \max(n_0, n-1) = n-1. $$
	Therefore we get $\extnorm{\mathcal{B}_n} \leq \extnorm{\vO_{\F}^{\times}/(1+\vP_{\F}^{n/2})} \ll q^{\frac{n}{2}}$.
\end{proof}

\begin{lemma} \label{lem: DWt-TrivBd}
	We have the bound 
	$$ \widetilde{h}_n^-(\chi) \ll q^{-\frac{1}{2}} \cdot \id_{\cond(\chi) \leq \cond(\Pi)+\frac{n}{2}} + \id_{n \leq 2\cond(\Pi)+2} \cdot q^{-\frac{n}{2}} \sideset{}{_{m=1}^{2\cond(\Pi)+3-n}} \sum q^{-\frac{m}{2}} \id_{\cond(\chi) \leq \max(\cond(\Pi),m)}. $$
	Consequently if $\cond(\Pi) > 0$ and $n_0 \leq A \cond(\Pi)$ for some constant $A \geq 1$ then we have
	$$ \widetilde{h}_c^-(\chi) \ll \sideset{}{_{\substack{n=\frac{2}{e}(n_0-2)+3 \\ 2 \mid en}}^{2n_1-1}} \sum \extnorm{\widetilde{h}_n^-(\chi)} \ll \cond(\Pi) \id_{\cond(\chi) \leq 3A \cond(\Pi)}. $$
\end{lemma}
\begin{proof}
	Write $\chi_0 = \chi \mid_{\vO_{\F}^{\times}}$. By definition and the Plancherel formula on $\vO_{\F}^{\times}$ we can write
	$$ \widetilde{h}_n^-(\chi) = \Vol(\vO_{\F}^{\times}, \ud^{\times}t)^{-1} q^{-n} \sideset{}{_{m=1}^{\infty}} \sum \chi(\varpi_{\F})^m q^{-\frac{m}{2}} \cdot C_m(n, \chi_0^{-1}),$$
	$$ C_m(n, \chi_0^{-1}) := \sideset{}{_{\xi \in \widehat{\vO_{\F}^{\times}}}} \sum a_m(\xi;n) \cdot b_m(\xi \chi_0^{-1}), $$
	where we have put
	$$ a_m(\xi;n) := \int_{\varpi_{\F}^{-m} \vO_{\F}^{\times}} \VorH_{\Pi,\psi}(H_n)(t) \xi^{-1}(t) \ud^{\times} t, \quad b_m(\xi) := \int_{\vO_{\F}^{\times}} \psi(-\varpi_{\F}^{-m}t) \xi(t) \ud^{\times} t. $$
	Since $m \geq 1$ the integral $b_m(\xi)$ is essentially a Gauss sum, which we can easily bound as
\begin{equation} \label{eq: bmBd}
	b_m(\xi) \ll q^{-\frac{m}{2}} \cdot \id_{\cond(\xi)=m} + \id_{m=1} \cdot q^{-1} \cdot \id_{\cond(\xi)=0} \ll q^{-\frac{m}{2}} \cdot \id_{\cond(\xi) \leq m}.
\end{equation}
	The defining formula for $a_m(\xi;n)$ makes sense for all $m \in \Z$, and we have (writing $X:=q^s$)
\begin{multline*}
	\sideset{}{_{m \in \Z}} \sum a_m(\xi;n) X^{-m} = \int_{\F^{\times}} \VorH_{\Pi,\psi}(H_n)(t) \xi^{-1}(t) \norm[t]^{-s} \ud^{\times} t = \gamma \left( s, \Pi \otimes \xi, \psi \right) \cdot \varepsilon_n(\xi, H) X^n \\
	= \varepsilon_n(\xi, H) \varepsilon \left( \tfrac{1}{2}, \Pi \otimes \xi, \psi \right) q^{\frac{\cond(\Pi \otimes \xi)}{2}} X^{n-\condL} \cdot \sideset{}{_{j=1}^d} \prod \tfrac{X-a_j}{1-\overline{a_j}q^{-1}X},
\end{multline*}
	which is equivalent to \eqref{eq: VHTestFLevelnMellin}. If $\xi \notin E(\Pi)$ then $d:=d(\Pi \otimes \xi)=0$, and we get for $m \geq 1$
\begin{equation} \label{eq: amBd1} 
	a_m(\xi;n) \ll \id_{\mathcal{B}_n}(\xi) \cdot \id_{m+n}(\cond(\Pi \otimes \xi)) \cdot q^{\frac{m+n}{2}}. 
\end{equation}
	If $\xi \in E(\Pi)$ then we apply the residue theorem (as $\norm[a_j] \in \{ 1, q^{-\frac{1}{2}}, q^{-1} \}$) to get
\begin{align}
	a_m(\xi;n) &= \varepsilon_n(\xi, H) \varepsilon \left( \tfrac{1}{2}, \Pi \otimes \xi, \psi \right) q^{\frac{\cond(\Pi \otimes \xi)}{2}} \int_{\norm[X]=q^{\frac{1}{2}}} X^{n-\cond(\Pi \otimes \xi)+m-1} \sideset{}{_{j=1}^d} \prod \tfrac{1-a_j X^{-1}}{1-\overline{a_j}q^{-1}X} \frac{\ud X}{2\pi i} \nonumber \\
	&\ll \id_{\mathcal{B}_n}(\xi) \cdot \id_{\geq m+n}(\condL(\Pi \otimes \xi)) \cdot q^{\frac{n+m}{2}}. \label{eq: amBd2}
\end{align}
	Combining the bounds \eqref{eq: bmBd}-\eqref{eq: amBd2}, Lemma \ref{lem: ExpCharBds} \& \ref{lem: BnAnal} and $\cond(\Pi \otimes \xi) \leq \cond(\Pi)+3\cond(\xi)$ (see \cite{BuH97}) we get
\begin{align*}
	C_m(n,\chi_0^{-1}) &\ll q^{\frac{n}{2}} \sideset{}{_{\substack{\xi \in \mathcal{B}_n - E(\Pi) \\ \cond(\xi \chi_0^{-1}) \leq m}}} \sum \id_{\cond(\Pi \otimes \xi) = m+n} + q^{\frac{n}{2}} \sideset{}{_{\substack{\xi \in \mathcal{B}_n \cap E(\Pi) \\ \cond(\xi \chi_0^{-1}) \leq m}}} \sum \id_{\geq m+n}(\condL(\Pi \otimes \xi)) \\
	&\ll q^{\frac{n}{2}} \sideset{}{_{\substack{\cond(\xi) \leq n/2 \\ \cond(\xi \chi_0^{-1}) \leq m}}} \sum \id_{m \leq \cond(\Pi)+\frac{n}{2}} + q^{\frac{n}{2}} \id_{\leq 2\cond(\Pi)+3}(m+n) \id_{\leq \max(\cond(\Pi),m)}(\cond(\chi_0)) \\
	&\leq q^{\frac{n}{2}+\min(\frac{n}{2},m)} \id_{m \leq \cond(\Pi)+\frac{n}{2}} \id_{\leq \cond(\Pi)+\frac{n}{2}}(\cond(\chi_0)) + q^{\frac{n}{2}} \id_{\leq 2\cond(\Pi)+3}(m+n) \id_{\leq \max(\cond(\Pi),m)}(\cond(\chi_0)),
\end{align*}
	and conclude the first bound. To derive the second bound, one simply notice $n_1$ can be taken as $2A\cond(\Pi)$, since under the condition $\cond(\Pi) > 0$ we have $a(\Pi) \leq \cond(\Pi)$ by the ``moreover'' part of Proposition \ref{prop: StabRang}.
\end{proof}

\begin{remark}
	The bound established in Lemma \ref{lem: DWt-TrivBd} is far from being optimal. For example in the case $e=1$ for $\xi \in \mathcal{B}_n$ the typical size of $\cond(\Pi \otimes \xi)$ should be $3n/2$, hence the term $q^{-\frac{\cond(\Pi \otimes \xi)}{2}}$ could be bounded as $q^{-\frac{3n}{4}}$. But even with this improvement the individual bound of $\widetilde{h}_n^-(\chi)$ is too weak to apply in the case $n_0 \gg \cond(\Pi)$. It would be interesting to recover the cancellation in the sum of $a_m(\xi;n) b_m(\xi)$ over $\xi$ by refining the above method.
\end{remark}

	\subsection{For Unramified Characters}

	We turn to the unramified case $\chi = \norm_{\F}^{s}$. According to the decomposition \eqref{eq: NDWtFDecomp} of $\widetilde{H}$, the notation $\widetilde{H}_{\infty}(k; s_0)$, $\widetilde{H}_c(k; s_0)$ and $\widetilde{H}_n^{\pm}(k; s_0)$ have obvious meaning and
	$$ \widetilde{H}(k; s_0) = \widetilde{H}_{\infty}(k; s_0) + \sideset{}{_{\pm}} \sum \widetilde{H}_c^{\pm}(k; s_0), \quad \widetilde{H}_c^{\pm}(k; s_0) = \sideset{}{_{\substack{n=\frac{2}{e}(n_0-2)+3 \\ 2 \mid en}}^{2n_1-1}} \sum \widetilde{H}_n^{\pm}(k; s_0). $$
\begin{lemma} \label{lem: DNDWt+Bd}
	For any $k \in \Z_{\geq 0}$ we have the following bounds.
\begin{itemize}
	\item[(1)] $\widetilde{H}_{\infty}(k; \tfrac{1}{2}) \ll_{k,\epsilon} q^{-n_1(1-\epsilon)}$, $\widetilde{H}_{\infty}(k; -\tfrac{1}{2}) \ll_{k,\epsilon} q^{n_1 \epsilon}$.
	\item[(2)] $\widetilde{H}_n^{+}(k; \pm \tfrac{1}{2}) = 0$ unless $n=\tfrac{2n_0}{e}+e-1$. We have
	$$ \widetilde{H}_c^{+}(k; \tfrac{1}{2}) \ll_{k,\epsilon} \Cond(\Pi)^{\frac{1}{2}} \cdot q^{\left( \frac{2n_0}{e}+e-1 \right) \epsilon}, \quad \widetilde{H}_c^{+}(k; -\tfrac{1}{2}) \ll_{k,\epsilon} \Cond(\Pi)^{\frac{1}{2}} \cdot q^{-\left( \frac{2n_0}{e}+e-1 \right)(1-\epsilon)}. $$
\end{itemize}
\end{lemma}
\begin{proof}
	(1) By Lemma \ref{lem: DWtInfty} we have $\widetilde{H}_{\infty}(s) = q^{-\left( \frac{1}{2}+s \right) n_1} L\left( \tfrac{1}{2}-s, \widetilde{\Pi} \right)^{-1}$. The stated bounds follow readily.
	
\noindent (2) By Lemma \ref{lem: TestFMellin} and \eqref{eq: DWtFLnLF} we have
\begin{multline*}
	f_n(q^{\frac{1}{2}}; \norm_{\F}^s, H) = f_n(q^{s+\frac{1}{2}}; \id, H) = \\
	q^{n \left( s - \frac{1}{2} \right)} \varepsilon \left( s+\tfrac{1}{2}, \Pi, \psi \right) \frac{L(\tfrac{1}{2}-s, \widetilde{\Pi})}{L(\tfrac{1}{2}+s, \Pi)} \cdot \begin{cases}
		\id_{n=2n_0} \cdot \zeta_{\F}(1) \chi_0(-1) & \text{if } \bL/\F \text{ split} \\
		\id_{n=\frac{2n_0}{e}+e-1} \cdot \zeta_{\F}(1) \varepsilon(1/2, \pi_{\beta}, \psi) & \text{if } \bL/\F \text{ non-split}
	\end{cases}.
\end{multline*}
	We introduce $a_j$ ($\norm[a_j] \in \{ 1, q^{-\frac{1}{2}}, q^{-1} \}$) as the parameters of $L(s,\Pi)$ so that $L(\tfrac{1}{2}-s, \widetilde{\Pi})=P(X)^{-1}$ with
	$$ P(X) := \sideset{}{_{j=1}^{d(\Pi)}} \prod \left( 1 - \overline{a_j} q^{-\frac{1}{2}} X \right), \quad f(X) := f_n(q^{\frac{1}{2}} X; \id, H) = X^{n-\condL(\Pi)} \sideset{}{_{j=1}^{d(\Pi)}} \prod \tfrac{X - a_j q^{-\frac{1}{2}}}{1 - \overline{a_j} q^{-\frac{1}{2}} X} \cdot q^{-\frac{n}{2}} \delta $$
	for some $\norm[\delta]=1$. We then rewrite, with $X = q^s$
	$$ \widetilde{H}_n^+(\norm_{\F}^s) = \tfrac{f_+(q^s)}{L(\tfrac{1}{2}-s, \widetilde{\Pi}) \zeta_{\F}(\tfrac{1}{2}+s)} = (1-q^{-\frac{1}{2}}X^{-1}) P(X)f_+(X). $$
	The bounds at $s=-1/2$ then follows readily from Lemma \ref{lem: LaurentTruncF} (1) via the bounds
	$$ \sideset{}{_{\norm[z]=1}} \sup \extnorm{D^k f(q^{-\frac{1}{2}}z)} \ll_{k,\epsilon} \Cond(\Pi)^{\frac{1}{2}} q^{-n(1-\epsilon)}, \quad D^k \left( (1-q^{-\frac{1}{2}}X^{-1}) P(X) \right) \mid_{X=q^{-\frac{1}{2}}} \ll_k 1. $$
	We now consider the bounds at $s=1/2$. If $n=\tfrac{2n_0}{e}+e-1 \geq \condL(\Pi)$, then $f_+(X) = f(X)$, implying
	$$ \widetilde{H}_n^+(\norm_{\F}^s) = q^{-\frac{n}{2}} \delta \cdot X^{n-\cond(\Pi)} (1-q^{-\frac{1}{2}}X^{-1}) \sideset{}{_{j=1}^{d(\Pi)}} \prod \left( 1 - a_j q^{-\frac{1}{2}} X^{-1} \right), $$
	$$ D^k \widetilde{H}_n^+(\norm_{\F}^s) \mid_{s=1/2} \ll_{k,\epsilon} \Cond(\Pi)^{-\frac{1}{2}} q^{n\epsilon}. $$
	If $n=\tfrac{2n_0}{e}+e-1 < \condL(\Pi) =: \condL$, then $Q(X) := P(X) f(X)$ has highest term $X^m$ with $m < d(\Pi) = \deg P$. We also note that $f_+(X) = 0$ unless $d:=d(\Pi) \geq 1$. We apply Lemma \ref{lem: LaurentTruncF} (2) distinguishing cases:
	
\noindent (\rmnum{1}) $d=1$. Necessarily we have $c:=c(\Pi) \geq 2$, $n-c \leq 0$. We get and conclude the stated bounds by
	$$ P(X) f_+(X) = \overline{a_1}^{c-n} \delta \cdot q^{-\frac{c}{2}} (1-\norm[a_1]^2 q^{-1}) \ll \Cond(\Pi)^{-\frac{1}{2}}. $$
	
\noindent (\rmnum{2}) $d=2$. Necessarily $\Pi = \mathrm{St}_{\chi} \boxplus \chi^{-2}$ for some unramified unitary character $\chi$. We have $c = 1$ and $n-c \leq 1$, and may assume $\norm[a_1] = q^{-\frac{1}{2}}$ and $\norm[a_2]=1$. We get the formula
\begin{align*} 
	P(X) f_+(X) &= \overline{a_1}^{c-n} \delta \cdot q^{-\frac{c}{2}} \frac{(1-\norm[a_1]^2q^{-1})(1-\overline{a_1}a_2q^{-1})}{1-\overline{a_2} \overline{a_1}^{-1}} (1-\overline{a_2}q^{-\frac{1}{2}}X) + \\
	&\quad \overline{a_2}^{c-n} \delta \cdot q^{-\frac{c}{2}} \frac{(1-\norm[a_2]^2q^{-1})(1-\overline{a_2}a_1q^{-1})}{1-\overline{a_1} \overline{a_2}^{-1}} (1-\overline{a_1}q^{-\frac{1}{2}}X)
\end{align*}
	and conclude the stated bounds by
\begin{equation} \label{eq: DerBd} 
	D^k \left( P(X) f_+(X) \right) \mid_{X=q^{\frac{1}{2}}} \ll_k 1. 
\end{equation}

\noindent (\rmnum{3}) $d=3$. Necessarily $\Pi = \mu_1 \boxplus \mu_2 \boxplus \mu_3$ for some unramified unitary characters $\mu_i$. We have $c=0$ and $n \leq 2$, and may assume $\norm[a_j]=1$ with $a_1a_2a_3=1$. It is easy to see $C_{j,k} \ll 1$. Therefore \eqref{eq: DerBd} still holds and we conclude the stated bounds in the same way.
\end{proof}

\begin{lemma} \label{lem: DNDWt-TrivBd}
	For any $k \in \Z_{\geq 0}$ we have the following bounds 
	$$ \widetilde{H}_n^{-}(k; \tfrac{1}{2}) \ll_{k,\epsilon} \Cond(\Pi)^{\epsilon} , \quad \widetilde{H}_n^{-}(k; -\tfrac{1}{2}) \ll_{k,\epsilon} \left( \Cond(\Pi)^2 q^{\frac{n}{2}} \right)^{1+\epsilon}. $$
	Consequently if $\cond(\Pi) > 0$ and $n_0 \leq A \cond(\Pi)$ for some constant $A \geq 1$ then we have
	$$ \widetilde{H}_c^-(k; \delta \cdot \tfrac{1}{2}) \leq \sideset{}{_{\substack{n=\frac{2}{e}(n_0-2)+3 \\ 2 \mid en}}^{2n_1-1}} \sum \extnorm{\widetilde{H}_n^{-}(k; \delta \cdot \tfrac{1}{2})} \ll_{\epsilon}
	\begin{cases} 
		\Cond(\Pi)^{\epsilon} & \text{if } \delta = 1 \\
		\Cond(\Pi)^{\left( 2+2Ae^{-1} \right)(1+\epsilon)} & \text{if } \delta = -1
	\end{cases}. $$
\end{lemma}
\begin{proof}
	With similar argument as in the proof of Lemma \ref{lem: DWt-TrivBd} we get
	$$ \widetilde{h}_n^-(\norm_{\F}^s) = q^{-n} \sideset{}{_{m=1}^{\infty}} \sum q^{m \left( \frac{1}{2}-s \right)} C_m(n), \quad C_m(n) \ll q^{\frac{n}{2}+\min(m,\frac{n}{2})} \id_{m \leq \cond(\Pi)+\frac{n}{2}} + q^{\frac{n}{2}} \id_{m+n \leq 2\cond(\Pi)+3}. $$
	The stated bounds follow readily.
\end{proof}

\section{Dual Weight Functions: Split and Unramified Cases}

	\subsection{First Quadratic Elementary Functions}
	
\begin{definition} \label{def: QEleF1}
	For any $n \in \Z_{\geq 0}$ we define the \emph{first quadratic elementary functions} $F_n \in \Cont_c^{\infty}(\F^{\times})$ by 
	$$ F_n(y^2) := \id_{v(y)=-n} \cdot \sideset{}{_{\pm}} \sum \psi(\pm y), $$
	and are supported in the subset of \emph{square} elements of $\F^{\times}$.
\end{definition}

\begin{definition} \label{def: QGSum}
	Let $\eta_0$ be the character of $\F^{\times}$ associated with the (ramified) quadratic extension $\F[\sqrt{-\varpi_{\F}}]/\F$. Explicitly it is given by the following formula for all $n \in \Z$
	$$ \eta_0(\varpi_{\F}^n u) = \begin{cases}
		1 & \text{if } u \in (\vO_{\F}^{\times})^2 \\
		-1 & \text{if } u \in \vO_{\F}^{\times} - (\vO_{\F}^{\times})^2
	\end{cases}. $$
	Denote by $\tau_0 = \tau(\eta_0, \psi; \varpi_{\F})$ the \emph{quadratic Gauss sum} given by
	$$ \tau_0 := \int_{\vO_{\F}^{\times}} \psi \left( \tfrac{u}{\varpi_{\F}} \right) \eta_0(u) \ud u = \int_{\vO_{\F}} \psi \left( \tfrac{u^2}{\varpi_{\F}} \right) \ud u. $$
\end{definition}

\begin{lemma} \label{lem: MellinQEleF1}
	Let $\chi$ be a quasi-character of $\F^{\times}$. We have
	$$ \int_{\F^{\times}} F_n(y) \chi(y) \ud^{\times} y = \begin{cases}
		\id_{\cond(\chi^2)=n} \cdot \zeta_{\F}(1) \gamma(1,\chi^{-2},\psi) & \text{if } n \geq 2 \\
		\id_{\cond(\chi^2)=1} \cdot \zeta_{\F}(1) \gamma(1,\chi^{-2},\psi) - \id_{\cond(\chi^2)=0} \cdot \zeta_{\F}(1) q^{-1} \chi(\varpi_{\F})^{-2} & \text{if } n = 1 \\
		\id_{\cond(\chi^2)=0} & \text{if } n = 0
	\end{cases}. $$
	Note that we can replace the condition $\cond(\chi^2)=n$ with $\cond(\chi)=n$ if $n \geq 1$.
\end{lemma}
\begin{proof}
	The case for $n=0$ is simple and omitted. For $n \geq 1$ with the change of variables $y \to y^2$
	$$ \int_{\F^{\times}} F_n(y) \chi(y) \ud^{\times} y = \int_{\varpi_{\F}^{-n} \vO_{\F}^{\times}} \psi(y) \chi^2(y) \ud^{\times} y, $$
	the desired formula follows readily from \cite[Exercise 23.5]{BuH06}.
\end{proof}

\begin{corollary} \label{cor: VHQEleF1}
	Let $n \geq a(\Pi) (\geq 2)$, the stability barrier of $\Pi$ (see Proposition \ref{prop: StabRang}). We have
	$$ \VorH_{\Pi,\psi}(F_n)(y) = \id_{\varpi_{\F}^{-n} \vO_{\F}^{\times}}(y) \cdot 
	\begin{cases}
		q^{\frac{3n}{2}} \psi(4y) & \text{if } 2 \mid n \\
		\tau_0 q^{\lceil \frac{3n}{2} \rceil} \psi(4y) \eta_0(4y) & \text{if } 2 \nmid n
	\end{cases}. $$
\end{corollary}
\begin{proof}
	By the local functional equation, Lemma \ref{lem: MellinQEleF1} and Proposition \ref{prop: StabRang} we have for any $\chi \in \widehat{\vO_{\F}^{\times}}$
\begin{multline*} 
	\int_{\F^{\times}} \VorH_{\Pi,\psi}(F_n)(y) \chi(y)^{-1} \norm[y]^{-s} \ud^{\times} y = \zeta_{\F}(1) \id_{\cond(\chi) = n} \cdot \gamma(s,\chi,\psi)^3 \gamma(1-2s,\chi^{-2},\psi) \\
	= \zeta_{\F}(1) \id_{\cond(\chi) = n} \gamma(1,\chi,\psi)^3 \gamma(1,\chi^{-2},\psi) \cdot q^{3n-ns}.
\end{multline*}
	Therefore the support of $\VorH_{\Pi,\psi}(F_n)$ is contained in $\varpi_{\F}^{-n} \vO_{\F}^{\times}$. Applying \cite[Exercise 23.5]{BuH06} we find
	$$ \zeta_{\F}(1) \id_{\cond(\chi) = n} \gamma(1,\chi,\psi)^3 \gamma(1,\chi^{-2},\psi) \cdot q^{3n} = \id_{\cond(\chi) = n} \cdot q^{3n} \int_{(\vO_{\F}^{\times})^4} \psi \left( \tfrac{t_1+t_2+t_3+t_4}{\varpi_{\F}^n} \right) \chi^{-1} \left( \tfrac{t_1t_2t_3}{\varpi_{\F}^n t_4^2} \right) \ud \vec{t}, $$
	where we have written $\ud \vec{t} := \ud t_1 \ud t_2 \ud t_3 \ud t_4$ for simplicity. The Fourier inversion on $\vO_{\F}^{\times}$ yields for $y \in \vO_{\F}^{\times}$
\begin{multline*}
	\VorH_{\Pi,\psi}(F_n) \left( \tfrac{y}{\varpi_{\F}^n} \right) = \zeta_{\F}(1) q^{3n} \sideset{}{_{\substack{\chi \in \widehat{\vO_{\F}^{\times}} \\ \cond(\chi)=n}}} \sum \int_{(\vO_{\F}^{\times})^4} \psi \left( \tfrac{t_1+t_2+t_3+t_4}{\varpi_{\F}^n} \right) \chi \left( \tfrac{t_4^2 y}{t_1t_2t_3} \right) \ud \vec{t} \\
	= \zeta_{\F}(1) q^{3n} \int_{(\vO_{\F}^{\times})^4} \psi \left( \tfrac{t_1+t_2+t_3+t_4}{\varpi_{\F}^n} \right) \left( \extnorm{\left( \vO_{\F}/\vP_{\F}^n \right)^{\times}} \id_{1+\vP_{\F}^n} - \extnorm{\left( \vO_{\F}/\vP_{\F}^{n-1} \right)^{\times}} \id_{1+\vP_{\F}^{n-1}} \right) \left( \tfrac{t_4^2 y}{t_1t_2t_3} \right) \ud \vec{t} \\
	= q^{3n} \int_{(\vO_{\F}^{\times})^3} \psi \left( \tfrac{t_2+t_3+t_4 + t_2^{-1}t_3^{-1}t_4^2y}{\varpi_{\F}^n} \right) \ud t_2 \ud t_3 \ud t_4 - \\
	q^{4n-1} \int_{(\vO_{\F}^{\times})^3 \times \vP_{\F}^{n-1}} \psi \left( \tfrac{t_2+t_3+t_4 + t_2^{-1}t_3^{-1}t_4^2y}{\varpi_{\F}^n} \right) \psi \left( \tfrac{t_2^{-1}t_3^{-1}t_4^2y u}{\varpi_{\F}^n} \right) \ud t_2 \ud t_3 \ud t_4 \ud u.
\end{multline*}
	The last integral is vanishing because $\cond(\psi)=0$. Re-numbering the variables we get
	$$ \VorH_{\Pi,\psi}(F_n) \left( \tfrac{y}{\varpi_{\F}^n} \right) = q^{3n} \int_{(\vO_{\F}^{\times})^3} \psi \left( \tfrac{t_1+t_2+t_3 + t_1^{-1}t_2^{-1}t_3^2 y}{\varpi_{\F}^n} \right) \ud t_1 \ud t_2 \ud t_3. $$
	Performing the level $\lceil \tfrac{n}{2} \rceil$ regularization to $\ud t_j$ we see that the non-vanishing of the above integral implies
	$$ t_1 - \tfrac{t_3^2 y}{t_1 t_2} \in \vP_{\F}^{\lfloor \frac{n}{2} \rfloor}, \quad t_2 - \tfrac{t_3^2 y}{t_1 t_2} \in \vP_{\F}^{\lfloor \frac{n}{2} \rfloor}, \quad t_3 + \tfrac{2 t_3^2 y}{t_1 t_2} \in \vP_{\F}^{\lfloor \frac{n}{2} \rfloor}. $$
	Writing $a = \tfrac{t_3^2 y}{t_1 t_2} \in \vO_{\F}^{\times}$, we see that the above conditions imply $a \equiv 4 y \pmod{\vP_{\F}^{\lfloor \frac{n}{2} \rfloor}}$, hence
	$$ t_1 \in 4y \left( 1 + \vP_{\F}^{\lfloor \frac{n}{2} \rfloor} \right), \quad t_2 \in 4y \left( 1 + \vP_{\F}^{\lfloor \frac{n}{2} \rfloor} \right), \quad t_3 \in -8y \left( 1 + \vP_{\F}^{\lfloor \frac{n}{2} \rfloor} \right). $$
	The change of variables $t_1 = 4y(1+u_1)$, $t_2 = 4y(1+u_2)$ and $t_3 = -8y(1+u_3)$ with $u_j \in \vP_{\F}^{\lfloor \frac{n}{2} \rfloor}$ gives
\begin{multline*}
	\VorH_{\Pi,\psi}(F_n) \left( \tfrac{y}{\varpi_{\F}^n} \right) = q^{3n} \int_{u_j \in \vP_{\F}^{\lfloor \frac{n}{2} \rfloor}} \psi \left( \frac{4y(u_1+u_2-2u_3) + \frac{4(1+u_3)^2y}{(1+u_1)(1+u_2)} }{\varpi_{\F}^n} \right) \ud u_1 \ud u_2 \ud u_3 \\
	= q^{3n} \psi \left( \tfrac{4y}{\varpi_{\F}^n} \right) \int_{u_j \in \vP_{\F}^{\lfloor \frac{n}{2} \rfloor}} \psi \left( \tfrac{4y}{\varpi_{\F}^n} \left( u_1^2 + u_2^2 + u_3^2 + u_1u_2 - 2u_1u_3 - 2u_2u_3 \right) \right) \ud u_1 \ud u_2 \ud u_3 \\
	= q^{3n} \psi \left( \tfrac{4y}{\varpi_{\F}^n} \right) \int_{u_j \in \vP_{\F}^{\lfloor \frac{n}{2} \rfloor}} \psi \left( \tfrac{4y}{\varpi_{\F}^n} \left( u_1^2 + u_2^2 + u_1u_2 - u_1u_3 \right) \right) \ud u_1 \ud u_2 \ud u_3 \\
	= q^{3n-\lfloor \frac{n}{2} \rfloor} \psi \left( \tfrac{4y}{\varpi_{\F}^n} \right) \int_{u_j \in \vP_{\F}^{\lfloor \frac{n}{2} \rfloor}} \psi \left( \tfrac{4y}{\varpi_{\F}^n} \left( u_1^2 + u_2^2 + u_1u_2 \right) \right) \id_{\vP_{\F}^{\lceil \frac{n}{2} \rceil}}(u_1) \ud u_1 \ud u_2 \\
	= q^{2n} \psi \left( \tfrac{4y}{\varpi_{\F}^n} \right) \int_{\vP_{\F}^{\lfloor \frac{n}{2} \rfloor}} \psi \left( \tfrac{4y}{\varpi_{\F}^n} u_2^2 \right) \ud u_2,
\end{multline*}
	where we made the change of variables $u_2 \mapsto u_2 + u_3$ in the third line. The desired formula follows.
\end{proof}

\begin{remark}
	If we compute the Mellin transform of $\VorH_{\Pi,\psi}(F_n)(y)$ with the given formula in Corollary \ref{cor: VHQEleF1}, we get in the case $2 \mid n \geq 2$ and for $\cond(\chi)=n$ an interesting equation
	$$ \chi(4) \gamma \left( \tfrac{1}{2}, \chi^{-1}, \psi \right) = \gamma \left( \tfrac{1}{2}, \chi, \psi \right)^3 \gamma \left( \tfrac{1}{2}, \chi^{-2}, \psi \right). $$
\end{remark}

\begin{corollary} \label{cor: VHQEleF1Bis}
	Suppose $\Pi = \mu_1 \boxplus \mu_2 \boxplus \mu_3$ with $\cond(\mu_j)=0$ and $\mu_1\mu_2\mu_3 = \id$. Let $E_i(y) := \mu_i^{-1}(y) \norm[y] \id_{\vO_{\F}}(y)$ and $f_i := (1 - \mu_i(\varpi_{\F}) \Trans(\varpi_{\F})).E_i$. For $f,g \in \intL^1(\F^{\times})$ define $f*g(y) := \int_{\F^{\times}} f(yt^{-1})g(t) \ud^{\times} t$.
	We have
\begin{align*} 
	\VorH_{\Pi,\psi}(F_0)(y) &= (f_1*f_2*f_3)(y) + \tau_0 q^2 \cdot \eta_0(-y) \id_{\varpi_{\F}^{-3} \vO_{\F}^{\times}}(y), \\
	\VorH_{\Pi,\psi}(F_1)(y) &= -q^{-1} \cdot (f_1*f_2*f_3)(\varpi_{\F}^{-2}y) - \zeta_{\F}(1) \cdot \id_{\varpi_{\F}^{-1}\vO_{\F}^{\times}}(y) \\
	&\quad + \tau_0 q \id_{\varpi_{\F}^{-1} \vO_{\F}^{\times}}(y) \int_{(\varpi_{\F}^{-1}\vO_{\F}^{\times})^2} \psi \left( t_1+t_2 - \tfrac{t_1t_2}{4y} \right) \eta_0 \left( \tfrac{4y}{t_1t_2} \right) \ud t_1 \ud t_2.
\end{align*}
\end{corollary}
\begin{proof}
	By the local functional equation and Lemma \ref{lem: MellinQEleF1} we have for any $\chi \in \widehat{\F^{\times}}$
	$$ \int_{\F^{\times}} \VorH_{\Pi,\psi}(F_0)(y) \chi(y)^{-1} \norm[y]^{-s} \ud^{\times} y = \begin{cases}
		\sideset{}{_{i=1}^3} \prod \tfrac{L(1-s, \mu_i^{-1}\chi^{-1})}{L(s, \mu_i \chi)} & \text{if } \cond(\chi) = 0 \\
		q^{3(\frac{1}{2}-s)} \gamma \left( \tfrac{1}{2}, \chi, \psi \right)^3 & \text{if } \cond(\chi \eta_0) = 0 \\
		0 & \text{otherwise}
	\end{cases}. $$
	Therefore we can write $\VorH_{\Pi,\psi}(F_0) = \VorH_{\Pi,\psi}(F_0)_0 + \VorH_{\Pi,\psi}(F_0)_1$ with the properties
	$$ \VorH_{\Pi,\psi}(F_0)_0(y \delta) = \VorH_{\Pi,\psi}(F_0)_0(y), \quad \VorH_{\Pi,\psi}(F_0)_1(y\delta) = \VorH_{\Pi,\psi}(F_0)_1(y) \eta_0(\delta), \quad \forall y \in \F^{\times}, \delta \in \vO_{\F}^{\times}; $$
	$$ \int_{\F^{\times}} \VorH_{\Pi,\psi}(F_0)_0(y) \norm[y]^{-s} \ud^{\times} y = \sideset{}{_{i=1}^3} \prod \tfrac{L(1-s, \mu_i^{-1})}{L(s, \mu_i)} = \sideset{}{_{i=1}^3} \prod \int_{\F^{\times}} f_i(y) \norm[y]^{-s} \ud^{\times} y, $$
	$$ \int_{\F^{\times}} \VorH_{\Pi,\psi}(F_0)_0(y) \eta_0(y) \norm[y]^{-s} \ud^{\times} y = q^{3(\frac{1}{2}-s)} \gamma \left( \tfrac{1}{2}, \eta_0, \psi \right)^3 = \eta_0(-1) \tau_0 q^2 \cdot \int_{\varpi_{\F}^{-3} \vO_{\F}^{\times}} \norm[y]^{-s} \ud^{\times} y. $$
	Hence we identify $\VorH_{\Pi,\psi}(F_0)_0 = f_1*f_2*f_3$ and $\VorH_{\Pi,\psi}(F_0)_1(y) = \tau_0 q^2 \cdot \eta_0(-y) \id_{\varpi_{\F}^{-3} \vO_{\F}^{\times}}(y)$. Similarly we can write $\VorH_{\Pi,\psi}(F_1) = \VorH_{\Pi,\psi}(F_1)_0 + \VorH_{\Pi,\psi}(F_1)_1$ with the properties
	$$ \VorH_{\Pi,\psi}(F_1)_0(y \delta) = \VorH_{\Pi,\psi}(F_1)_0(y), \quad \forall \delta \in \vO_{\F}^{\times}; $$
	$$ \int_{\F^{\times}} \VorH_{\Pi,\psi}(F_1)_0(y) \norm[y]^{-s} \ud^{\times} y = - q^{2s-1} \sideset{}{_{i=1}^3} \prod \int_{\F^{\times}} f_i(y) \norm[y]^{-s} \ud^{\times} y, $$
	$$ \int_{\F^{\times}} \VorH_{\Pi,\psi}(F_1)_1(y) \chi^{-1}(y) \norm[y]^{-s} \ud^{\times} y = \id_{\cond(\chi)=1} \cdot q^{3-s} \int_{\varpi_{\F}^{-1} \vO_{\F}^{\times}} \psi(y) \chi^2(y) \ud^{\times}y \cdot \left( \int_{\varpi_{\F}^{-1} \vO_{\F}^{\times}} \psi(y) \chi^{-1}(y) \ud^{\times}y \right)^3. $$
	We easily identify $\VorH_{\Pi,\psi}(F_1)_0(y) = -q^{-1} \cdot (f_1*f_2*f_3)(\varpi_{\F}^{-2}y)$. The argument in the proof of Corollary \ref{cor: VHQEleF1} shows that $\mathrm{supp} \left( \VorH_{\Pi,\psi}(F_1)_1 \right) \subset \varpi_{\F}^{-1} \vO_{\F}^{\times}$, and for $y \in \vO_{\F}^{\times}$ that
\begin{multline*}
	\VorH_{\Pi,\psi}(F_1)_1 \left( \tfrac{y}{\varpi_{\F}} \right) = q^{3} \int_{(\vO_{\F}^{\times})^3} \psi \left( \tfrac{t_2+t_3+t_4 + t_2^{-1}t_3^{-1}t_4^2y}{\varpi_{\F}} \right) \ud t_2 \ud t_3 \ud t_4 - \zeta_{\F}(1) q^{3} \int_{(\vO_{\F}^{\times})^4} \psi \left( \tfrac{t_1+t_2+t_3+t_4}{\varpi_{\F}} \right) \ud \vec{t} \\
	= \tau_0 q^3 \int_{(\vO_{\F}^{\times})^2} \psi \left( \tfrac{t_2+t_3-\frac{t_2t_3}{4y}}{\varpi_{\F}} \right) \eta_0 \left( \tfrac{y}{t_2t_3} \right) \ud t_2 \ud t_3 - q^2 \int_{(\vO_{\F}^{\times})^2} \psi \left( \tfrac{t_2+t_3}{\varpi_{\F}} \right) \ud t_2 \ud t_3 - \zeta_{\F}(1) q^{-1} \\
	= \tau_0 q^3 \int_{(\vO_{\F}^{\times})^2} \psi \left( \tfrac{t_2+t_3-\frac{t_2t_3}{4y}}{\varpi_{\F}} \right) \eta_0 \left( \tfrac{y}{t_2t_3} \right) \ud t_2 \ud t_3 - \zeta_{\F}(1).
\end{multline*}
	Re-numbering the variables and inserting the formula of $\VorH_{\Pi,\psi}(F_0)$ we get the formula of $\VorH_{\Pi,\psi}(F_1)$.
\end{proof}

	\subsection{Further Reductions}

	The functions $F_n$ are ``building blocks'' of our test functions $H_c$ in \eqref{eq: TestFComp} when $\bL/\F$ is not ramified. In fact writing $\varepsilon_{\bL}:=\eta_{\bL/\F}(\varpi_{\F}) \in \{ \pm 1 \}$ and applying Lemma \ref{lem: TestFAsymp} we can rewrite the summands of $H_c$ as
\begin{align}
	&H_{2n+1} = 0, \label{eq: FPTestFNR0} \\
	&H_{2n} = \begin{cases}
		E_n & \text{if } 2n_0 \leq n \leq n_1-1 \\
		(\varepsilon_{\bL}q)^n \int_{\substack{\bL^1 \cap \vO_{\bL} \\ \Tr(\alpha) \in 2(1+\vP_{\F}^{2(n_0-1)})}} \beta(\alpha) \cdot \left( \Trans \left( \Tr(\alpha)^2 \right).F_n \right) \ud \alpha & \text{if } n = 2n_0-1 \\
		(\varepsilon_{\bL}q)^n \int_{\substack{\bL^1 \cap \vO_{\bL} \\ \Tr(\alpha) \in 2(1+\varpi_{\F}^{2(n-n_0)} \vO_{\F}^{\times})}} \beta(\alpha) \cdot \left( \Trans \left( \Tr(\alpha)^2 \right).F_n \right) \ud \alpha & \text{if } n_0+1 \leq n < 2n_0-1
	\end{cases}. \label{eq: FPTestFNR1}
\end{align}
	The decomposition of $H_{2n_0}$ is subtler and really goes in the direction of expression in terms of the \emph{quadratic elementary functions}. We shall write (the second numeric parameter $m$ in subscript always indicates the parameter of the relevant quadratic elementary function)
\begin{align} 
	H_{2n_0} &= H_{2n_0}^a + H_{2n_0}^b, \nonumber \\
	H_{2n_0}^a &= \tfrac{(\varepsilon_{\bL}q)^{n_0}}{2} \id_{\eta_0(-1)=\varepsilon_{\bL}} \cdot \sideset{}{_{m=0}^{n_0-1}} \sum H_{2n_0, m}^{a,1} + \tfrac{(\varepsilon_{\bL}q)^{n_0}}{2} \id_{\eta_0(-1)=-\varepsilon_{\bL}} \cdot \sideset{}{_{m=0}^{n_0-1}} \sum H_{2n_0, m}^{a,\varepsilon}, \label{eq: FPTestFNR2} \\
	H_{2n_0}^b &= \tfrac{(\varepsilon_{\bL}q)^{n_0}}{2} \cdot \left\{ H_{2n_0,n_0}^{b,\varepsilon} + H_{2n_0,n_0}^{b,1} \right\}, \label{eq: FPTestFNR2Bis}
\end{align}
	where the summands are given by (below we assume $m \geq 1$ and $\tau \in \{ 1, \varepsilon \}$)
	$$ H_{2n_0, 0}^{a,\tau} = \int_{\substack{\bL^1 \cap \vO_{\bL} \\ \Tr(x_{\tau}\alpha) \in \vP_{\F}^{n_0}}} \beta(x_{\tau} \alpha) \ud \alpha \cdot \left( \Trans \left( \tau^{-1}\varpi_{\F}^{2n_0} \right).F_0 \right), $$
	$$ H_{2n_0, m}^{a,\tau} = \int_{\substack{\bL^1 \cap \vO_{\bL} \\ \Tr(x_{\tau} \alpha) \in \varpi_{\F}^{n_0-m} \vO_{\F}^{\times}}} \beta(x_{\tau} \alpha) \cdot \left( \Trans \left( \tau^{-1} \Tr(x_{\tau} \alpha)^2 \right).F_{m} \right) \ud \alpha; $$
	$$ H_{2n_0,n_0}^{b,\varepsilon} = \int_{\substack{\bL^1 \cap \vO_{\bL} \\ \Tr(x_{\varepsilon} \alpha) \in \vO_{\F}^{\times}}} \beta(x_{\varepsilon} \alpha) \cdot \left( \Trans \left( \varepsilon^{-1} \Tr(x_{\varepsilon} \alpha)^2 \right).F_{n_0} \right) \ud \alpha, $$
	$$ H_{2,1}^{b,1} = \int_{\substack{\bL^1 \cap \vO_{\bL} \\ \Tr(\alpha) \in \vO_{\F}^{\times}}} \beta(\alpha) \cdot \left( \Trans \left( \Tr(\alpha)^2 \right).F_{1} \right) \ud \alpha, $$
	$$ H_{2n_0,n_0}^{b,1} = \int_{\substack{\bL^1 \cap \vO_{\bL} \\ \Tr(\alpha)^2 \in \vO_{\F}^{\times}-4(1+\vP_{\F})}} \beta(\alpha) \cdot \left( \Trans \left( \Tr(\alpha)^2 \right).F_{n_0} \right) \ud \alpha, \ n_0 \geq 2. $$

\begin{lemma} \label{lem: JacTypeSRC}
	Let $\chi$ be a (unitary) character of $\F^{\times}$ with $\cond(\chi)=n$. Recall the additive parameter $c_{\beta}$ (resp. $c_{\chi}$) in Lemma \ref{lem: AddParMultChar} (resp. Remark \ref{rmk: AddParMultChar}).
		
\noindent (1) Assume $0 \leq n < n_0$. For $\bL/\F$ \emph{split} resp. \emph{unramified} we have
	$$ \int_{\vO_{\F}^{\times}} \chi_0 \left( \tfrac{1+\varpi_{\F}^{n_0-n} t}{1-\varpi_{\F}^{n_0-n} t} \right) \chi(t) \ud t \quad {\rm resp.} \quad \int_{\vO_{\F}^{\times}} \beta(1+\varpi_{\F}^{n_0-n} t \sqrt{\varepsilon}) \chi(t) \ud t \ll q^{-\frac{n}{2}}. $$
	
\noindent (2) Assume $n=n_0 \geq 2$. For $\bL/\F$ \emph{split} resp. \emph{unramified} we have for any $k \in \{ 0,1 \}$
\begin{multline*} 
	\int_{\vO_{\F}^{\times} - \sideset{}{_{\pm}} \bigcup (\pm 1+ \vP_{\F})} \chi_0 \left( \tfrac{1+t}{1-t} \right) \chi(t) \eta_0^k(1-t^2) \ud t \quad {\rm resp.} \quad \int_{\vO_{\F}^{\times}} \beta(1+t \sqrt{\varepsilon}) \chi(t) \eta_0^k(1-t^2 \varepsilon) \ud t \\
	 \ll \begin{cases} q^{-\frac{n_0}{2}} & \text{if } \chi \notin \Ecp(\beta) \\ q^{-\frac{n_0}{2}+\frac{1}{2}} & \text{if } \chi \in \Ecp(\beta) \end{cases}, 
\end{multline*}
	where $\Ecp(\beta) \neq \emptyset$ only if $2 \nmid n_0$ and $\eta_0(-1)=\varepsilon_{\bL}$, under which condition it is given by
	$$ \Ecp(\beta) = \begin{cases}
	\left\{ \chi \ \middle| \ (c_{\chi} c_{\beta}^{-1})^2 \equiv -1 \pmod{\vP_{\F}} \right\} & \text{if } \bL/\F \text{ split} \\
	\left\{ \chi \ \middle| \  (c_{\chi} c_{\beta}^{-1})^2 \equiv - \varepsilon \pmod{\vP_{\F}} \right\} & \text{if } \bL/\F \text{ unramified}
	\end{cases}. $$
\end{lemma}
\begin{proof}
	We omit the proof of the split case, which is similar and simpler. The case of $n=0<n_0$ is easy and omitted, since the integrand is constant.

\noindent (1) If $n=1$, then $t \mapsto \beta \left( 1+ \varpi_{\F}^{n_0-n} t \sqrt{\varepsilon} \right)$ is a non-trivial additive character of $\vO_{\F}$ and the bound follows from the one for Gauss sums. Assume $n \geq 2$. We perform a level $\lceil \tfrac{n}{2} \rceil$ regularization to $\ud t$ and get
\begin{multline*}
	\int_{\vO_{\F}^{\times}} \beta \left( 1+ \varpi_{\F}^{n_0-n} t \sqrt{\varepsilon} \right) \chi(t) \ud t \\
	= \int_{\vO_{\F}^{\times}} \beta \left( 1+ \varpi_{\F}^{n_0-n} t \sqrt{\varepsilon} \right) \chi(t) \cdot \left( \oint_{\vO_{\F}} \beta \left( 1 + \tfrac{\varpi_{\F}^{n_0-\lfloor \frac{n}{2} \rfloor}}{1+\varpi_{\F}^{n_0-n}t \sqrt{\varepsilon}} t u \sqrt{\varepsilon} \right) \chi(1+\varpi_{\F}^{\lceil \frac{n}{2} \rceil}u) \ud u \right) dt.
\end{multline*}
	The integrands of the inner integral, denoted by $I(t; n)$, are additive characters of $\vO_{\F}$. We have
	$$ I(t; n) = \oint_{\vO_{\F}} \psi \left( \tfrac{1}{\varpi_{\F}^{\lfloor \frac{n}{2} \rfloor}} \cdot \tfrac{2 c_{\beta} t u \varepsilon}{1-\varpi_{\F}^{2(n_0-n)} t^2 \varepsilon} \right) \psi \left( \tfrac{c_{\chi}u}{\varpi_{\F}^{\lfloor \frac{n}{2} \rfloor}} \right) \ud u. $$
	The non-vanishing of $I(t;n)$ implies the congruence condition
	$$ \tfrac{2 c_{\beta} t \varepsilon}{1-\varpi_{\F}^{2(n_0-n)} t^2 \varepsilon} + c_{\chi} \in \vP_{\F}^{\lfloor \frac{n}{2} \rfloor} \quad \Leftrightarrow \quad 2 c_{\beta} t \varepsilon + c_{\chi}(1-\varpi_{\F}^{2(n_0-n)} t^2 \varepsilon) \in \vP_{\F}^{\lfloor \frac{n}{2} \rfloor}, $$
	which has a unique solution $t \in t_0 + \vP_{\F}^{\lfloor \frac{n}{2} \rfloor}$ with $t_0 \in \vO_{\F}^{\times}$ by Hensel's lemma. Consequently we get
\begin{equation} \label{eq: IntermCan}
	\int_{\vO_{\F}^{\times}} \beta \left( 1+ \varpi_{\F}^{n_0-n} t \sqrt{\varepsilon} \right) \chi(t) \ud t = \int_{t_0 + \vP_{\F}^{\lfloor \frac{n}{2} \rfloor}} \beta \left( 1+ \varpi_{\F}^{n_0-n} t \sqrt{\varepsilon} \right) \chi(t) \ud t \ll q^{-\lfloor \frac{n}{2} \rfloor}.
\end{equation}
	If $2 \mid n$ then we are done (for both (1) and (2)). Otherwise let $n=2m+1$. We may assume
	$$ 2 c_{\beta} t_0 \varepsilon + c_{\chi}(1-\varpi_{\F}^{2(n_0-n)} t_0^2 \varepsilon) = 0, $$
	make the change of variables $t = t_0(1+\varpi_{\F}^m u)$, and continue (\ref{eq: IntermCan}) to conclude by
	$$ \int_{\vO_{\F}^{\times}} \beta \left( 1+ \varpi_{\F}^{n_0-n} t \sqrt{\varepsilon} \right) \chi(t) \ud t = q^{-m} \beta(1+ \varpi_{\F}^{n_0-n} t_0 \sqrt{\varepsilon}) \chi(t_0) \int_{\vO_{\F}} \psi \left( -\tfrac{c_{\chi}u^2}{2 \varpi_{\F}} \right) du \ll q^{-\frac{n}{2}}. $$
	
\noindent (2) Let $n=2m+1$ with $m \geq 1$. With the change of variables $t=t_0(1+\varpi_{\F}^m u)$ for $t_0 \in \vO_{\F}^{\times}$ satisfying
\begin{equation} \label{eq: EcpQuadCond}
	2 c_{\beta} t_0 \varepsilon + c_{\chi}(1- t_0^2 \varepsilon) = 0,
\end{equation}
	which is solvable only if $\eta_0(-1) = \varepsilon_{\bL}$, we get the equation
\begin{multline*}
	\int_{\vO_{\F}^{\times}} \beta(1+t \sqrt{\varepsilon}) \chi(t) \eta_0^k(1-t^2 \varepsilon) \ud t = \sideset{}{_{t_0}} \sum \beta(1+t_0 \sqrt{\varepsilon}) \chi(t_0) \eta_0^k(1-t_0^2 \varepsilon) \cdot \int_{\vO_{\F}} \psi \left( \tfrac{u^2}{\varpi} \left( \tfrac{2 c_{\beta} t_0^3 \varepsilon}{(1-t_0^2 \varepsilon)^2} - \tfrac{c_{\chi}}{2} \right) \right) \ud u \\
	= \sideset{}{_{t_0}} \sum \beta(1+t_0 \sqrt{\varepsilon}) \chi(t_0) \eta_0^k(1-t_0^2 \varepsilon) \cdot \int_{\vO_{\F}} \psi \left( \tfrac{c_{\chi} u^2}{2 \varpi} \left( \tfrac{c_{\chi}}{c_{\beta}} t_0 - 1 \right) \right) \ud u.
\end{multline*}
	The above integral is $\ll q^{-\frac{1}{2}}$ unless $t_0 \equiv c_{\chi}^{-1}c_{\beta} \pmod{\vP_{\F}}$, in which case \eqref{eq: EcpQuadCond} implies $\chi \in \Ecp(\beta)$.
\end{proof}

\begin{lemma} \label{lem: e1DWt-FineBd1}
	Suppose $n_0 + 1 \leq n \leq 2n_0 - 1$ (hence $n_0 \geq 2$) and $n \geq a(\Pi)$. Then we have
	$$ \widetilde{h}_{2n}^-(\chi) \ll q^{-n_0} \id_{2n_0-n}(\cond(\chi \eta_0^n)) + q^{-n_0} \id_{n=2n_0-1} \id_{\cond(\chi \eta_0)=0}. $$
\end{lemma}
\begin{proof}
	(1) First consider $n < 2n_0-1$. By Corollary \ref{cor: VHQEleF1} we have for any $\delta \in 1+\varpi_{\F}^{2(n-n_0)} \vO_{\F}^{\times}$
\begin{multline} \label{eq: DWtQEleF1Trans1}
	\int_{\F-\vO_{\F}} \VorH_{\Pi,\psi} \circ \Mult_{-1} \left( \Trans \left( 4\delta \right).F_n \right)(t) \cdot \psi(-t) \chi^{-1}(t) \norm[t]^{-\frac{1}{2}} \ud^{\times} t \\
	= q^{-n} \zeta_{\F}(1) \cdot \begin{cases}
		\chi^{-1} \left( \tfrac{\delta}{1-\delta} \right) \cdot \gamma(1,\chi,\psi) \id_{2n_0-n}(\cond(\chi)) & \text{if } 2 \mid n \\
		\chi^{-1} \eta_0 \left( \tfrac{\delta}{1-\delta} \right) \cdot \tau_0 q^{\frac{1}{2}} \gamma(1,\chi \eta_0,\psi) \id_{2n_0-n}(\cond(\chi \eta_0)) & \text{if } 2 \nmid n
	\end{cases},
\end{multline}
	where we used $\eta_0(\delta)=1$. Inserting \eqref{eq: DWtQEleF1Trans1} with $\delta = 4^{-1} \Tr(\alpha)^2$ into \eqref{eq: FPTestFNR1} we get
\begin{multline} \label{eq: h2n-BdRedNR}
	\widetilde{h}_{2n}^-(\chi) = \int_{\F-\vO_{\F}} \VorH_{\Pi,\psi} \circ \Mult_{-1} (H_{2n})(t) \cdot \psi(-t) \chi^{-1}(t) \norm[t]^{-\frac{1}{2}} \ud^{\times} t \ll \\
	\extnorm{\int_{\substack{\bL^1 \cap \vO_{\bL} \\ \Tr(\alpha) \in 2(1+\varpi_{\F}^{2(n-n_0)} \vO_{\F}^{\times})}} \beta(\alpha) \cdot \chi^{-1} \eta_0^n \left( \tfrac{\Tr(\alpha)^2}{4-\Tr(\alpha)^2} \right) \ud \alpha} \cdot q^{-\frac{2n_0-n}{2}} \id_{2n_0-n}(\cond(\chi \eta_0^n)).
\end{multline}
	
\noindent In the case of split, resp. unramified $\bL/\F$ we apply the change of variables $t = \tfrac{\alpha - \alpha^{-1}}{\alpha + \alpha^{-1}}$, resp. $t = \tfrac{\alpha - \alpha^{-1}}{\alpha + \alpha^{-1}} \tfrac{1}{\sqrt{\varepsilon}}$. The inner integral in \eqref{eq: h2n-BdRedNR} becomes
	$$ 2 \chi\eta_0^n(-1) \int_{\varpi_{\F}^{n-n_0} \vO_{\F}^{\times}} \chi_0 \left( \tfrac{1+t}{1-t} \right) \chi^2(t) \ud t \quad {\rm resp.} \quad \chi\eta_0^n(-\varepsilon) \int_{\varpi_{\F}^{n-n_0} \vO_{\F}^{\times}} \beta(1+t \sqrt{\varepsilon}) \chi^2(t) \ud t. $$
	We apply Lemma \ref{lem: JacTypeSRC} (1) to bound the above integrals and conclude.
	
\noindent (2) Consider $n=2n_0-1$. We have for any $\delta \in 1+\vP_{\F}^{2(n_0-1)}$
\begin{multline} \label{eq: DWtQEleF1Trans1Bis}
	\int_{\F-\vO_{\F}} \VorH_{\Pi,\psi} \circ \Mult_{-1} \left( \Trans \left( 4\delta \right).F_n \right)(t) \cdot \psi(-t) \chi^{-1}(t) \norm[t]^{-\frac{1}{2}} \ud^{\times} t = q^{-(2n_0-1)} \cdot \tau_0 q^{\frac{1}{2}} \cdot \\
	\begin{cases}
		\chi\eta(\varpi_{\F})^{2n_0-1} \cdot \id_0(\cond(\chi \eta_0))  & \text{if } \delta \in 1+\vP_{\F}^{2n_0-1} \\
		\chi^{-1} \eta_0 \left( \tfrac{\delta}{1-\delta} \right) \cdot \zeta_{\F}(1) \left\{ \gamma(1,\chi \eta_0,\psi) \id_{1}(\cond(\chi \eta_0)) - q^{-1} \id_0(\cond(\chi \eta_0)) \right\} & \text{if } \delta \in 1+\varpi_{\F}^{2(n_0-1)} \vO_{\F}^{\times}
	\end{cases}.
\end{multline}
	Inserting \eqref{eq: DWtQEleF1Trans1Bis} with $\delta = 4^{-1} \Tr(\alpha)^2$ into \eqref{eq: FPTestFNR1} we get
\begin{multline} \label{eq: h2n-BdRedNRBis}
	\widetilde{h}_{2n}^-(\chi) = \int_{\F-\vO_{\F}} \VorH_{\Pi,\psi} \circ \Mult_{-1} (H_{2n})(t) \cdot \psi(-t) \chi^{-1}(t) \norm[t]^{-\frac{1}{2}} \ud^{\times} t \ll \\
	\extnorm{\int_{\substack{\bL^1 \cap \vO_{\bL} \\ \Tr(\alpha) \in 2(1+\varpi_{\F}^{2(n_0-1)} \vO_{\F}^{\times})}} \beta(\alpha) \cdot \chi^{-1} \eta_0 \left( \tfrac{\Tr(\alpha)^2}{4-\Tr(\alpha)^2} \right) \ud \alpha} \cdot q^{-\frac{1}{2}} \id_{1}(\cond(\chi \eta_0)) + \\
	\extnorm{ \int_{\substack{\bL^1 \cap \vO_{\bL} \\ \Tr(\alpha) \in 2(1+\vP_{\F}^{2n_0-1})}} \beta(\alpha) \ud \alpha - \tfrac{\chi\eta_0(\varpi_{\F})^{-1}}{q-1} \int_{\substack{\bL^1 \cap \vO_{\bL} \\ \Tr(\alpha) \in 2(1+\varpi_{\F}^{2(n_0-1)} \vO_{\F}^{\times})}} \beta(\alpha) \ud \alpha } \cdot \id_0(\cond(\chi \eta_0)).
\end{multline}
	The first summand is bounded the same way as before. Note that $\Tr(\alpha) \in 2(1+\vP_{\F}^{2n_0-1})$ is equivalent to $\alpha \in 1+\vP_{\bL}^{n_0}$. With the same change of variables the inner integrals in the second summand become
	$$ 2q^{-n_0} - \tfrac{2\chi\eta_0(\varpi_{\F})^{-1}}{q-1} \int_{\varpi_{\F}^{n_0-1} \vO_{\F}^{\times}} \chi_0 \left( \tfrac{1+t}{1-t} \right) \ud t \quad {\rm resp.} \quad q^{-n_0} - \tfrac{\chi\eta_0(\varpi_{\F})^{-1}}{q-1} \int_{\varpi_{\F}^{n_0-1} \vO_{\F}^{\times}} \beta(1+t \sqrt{\varepsilon}) \ud t. $$
	They are of size $O(q^{-n_0})$ since the integrands are additive characters of conductor exponent $n_0$.
\end{proof}

\begin{lemma} \label{lem: e1DWt-FineBd2}
	Suppose $(2 \leq) a(\Pi) \leq m < n_0$. Then we have for unitary $\chi$ and $\tau \in \{ 1, \varepsilon \}$
	$$ \widetilde{h}_{2n_0,m}^{a,\tau}(\chi) := \int_{\F-\vO_{\F}} \VorH_{\Pi,\psi} \circ \Mult_{-1} \left( H_{2n_0,m}^{a,\tau} \right)(t) \psi(-t)\chi^{-1}(t)\norm[t]^{-\frac{1}{2}} \ud^{\times} t \ll q^{-2n_0} \id_{m > \frac{2n_0}{3}} \id_m(\cond(\chi)). $$
\end{lemma}
\begin{proof}
	We first use Corollary \ref{cor: VHQEleF1} to obtain for any $\delta \in \varpi_{\F}^{2(n_0-m)} \vO_{\F}^{\times}$ (note that $\eta_0(4-\delta)=1$)
\begin{multline} \label{eq: DWtQEleF1Trans2}
	\int_{\F-\vO_{\F}} \VorH_{\Pi,\psi} \circ \Mult_{-1} \left( \Trans \left( \delta \right).F_m \right)(t) \cdot \psi(-t) \chi^{-1}(t) \norm[t]^{-\frac{1}{2}} \ud^{\times} t \\
	= q^{-n_0} \zeta_{\F}(1) \id_{m > \frac{2n_0}{3}} \cdot \begin{cases}
		\chi \left( \tfrac{4-\delta}{\delta} \right) \cdot \gamma(1,\chi,\psi) \id_{m}(\cond(\chi)) & \text{if } 2 \mid m \\
		\chi \left( \tfrac{4-\delta}{\delta} \right) \cdot \tau_0 q^{\frac{1}{2}} \gamma(1,\chi\eta_0,\psi) \id_{m}(\cond(\chi)) & \text{if } 2 \nmid m
	\end{cases}.
\end{multline}
	
\noindent Inserting \eqref{eq: DWtQEleF1Trans2} with $\delta = \tau^{-1} \Tr(x_{\tau} \alpha)^2$ into the integral representation of $H_{2n_0,m}^{a,\tau}$ we get
\begin{equation} \label{eq: h2n0-BdRed1}
	\widetilde{h}_{2n_0,m}^{a,\tau}(\chi) \ll \id_{m > \frac{2n_0}{3}} \cdot \extnorm{ \int_{\substack{\bL^1 \cap \vO_{\bL} \\ \Tr(x_{\tau} \alpha) \in \varpi_{\F}^{n_0-m} \vO_{\F}^{\times}}} \beta(x_{\tau} \alpha) \cdot \chi \left( \tfrac{4 \tau -\Tr(x_{\tau} \alpha)^2}{\Tr(x_{\tau} \alpha)^2} \right) \ud \alpha } \cdot q^{-n_0-\frac{m}{2}} \id_{m}(\cond(\chi)).
\end{equation}
	In the case of split, resp. unramified $\bL/\F$ we apply the change of variables $t = \tfrac{\alpha + \alpha^{-1}}{\alpha - \alpha^{-1}}$, resp. $t = \tfrac{\alpha + \alpha^{-1}}{\alpha - \alpha^{-1}} \tfrac{1}{\sqrt{\varepsilon}}$ for $\widetilde{h}_{2n_0,m}^{a,1}(\chi)$; $t = \tfrac{\alpha + \varepsilon \alpha^{-1}}{\alpha - \varepsilon \alpha^{-1}}$, resp. $t = \tfrac{x_{\varepsilon}\alpha + \overline{x_{\varepsilon}\alpha}}{x_{\varepsilon}\alpha - \overline{x_{\varepsilon}\alpha}} \tfrac{1}{\sqrt{\varepsilon}}$ for $\widetilde{h}_{2n_0,m}^{a,\varepsilon}(\chi)$. The inner integrals in \eqref{eq: h2n0-BdRed1} become
	$$ 2 \chi_0\chi(-1) \int_{\varpi_{\F}^{n_0-m} \vO_{\F}^{\times}} \chi_0 \left( \tfrac{1+t}{1-t} \right) \chi^{-2}(t) \ud t \quad {\rm resp.} \quad 2\beta(\sqrt{\varepsilon})\chi^{-1}(-\varepsilon) \int_{\varpi_{\F}^{n_0-m} \vO_{\F}^{\times}} \beta(1+t \sqrt{\varepsilon}) \chi^{-2}(t) \ud t. $$
	We apply Lemma \ref{lem: JacTypeSRC} (1) to bound the above integrals and conclude the desired inequalities.
\end{proof}

\begin{lemma} \label{lem: e1DWt-FineBd3}
	Suppose $n_0 \geq a(\Pi) (\geq 2)$. Let $\tau \in \{ 1, \varepsilon \}$. With $\Ecp(\beta)$ defined in Lemma \ref{lem: JacTypeSRC} (2) we have
\begin{multline*} 
	\widetilde{h}_{2n_0,n_0}^{b,\tau}(\chi) := \int_{\F-\vO_{\F}} \VorH_{\Pi,\psi} \circ \Mult_{-1} \left( H_{2n_0,n_0}^{b,\tau} \right)(t) \psi(-t)\chi^{-1}(t)\norm[t]^{-\frac{1}{2}} \ud^{\times} t \\
	\ll q^{-2n_0} \id_{n_0}(\cond(\chi)) \left( \id_{\Ecp(\beta)^c}(\chi^2) + q^{\frac{1}{2}} \id_{\Ecp(\beta)}(\chi^2) \right).
\end{multline*}
\end{lemma}
\begin{proof}
	We first use Corollary \ref{cor: VHQEleF1} to obtain for any $\delta \in \vO_{\F}^{\times}-4(1+\vP_{\F})$
\begin{multline} \label{eq: DWtQEleF1Trans3}
	\int_{\F-\vO_{\F}} \VorH_{\Pi,\psi} \circ \Mult_{-1} \left( \Trans \left( \delta \right).F_m \right)(t) \cdot \psi(-t) \chi^{-1}(t) \norm[t]^{-\frac{1}{2}} \ud^{\times} t \\
	= q^{-n_0} \zeta_{\F}(1) \cdot \begin{cases}
		\chi \left( \tfrac{4-\delta}{\delta} \right) \cdot \gamma(1,\chi,\psi) \id_{n_0}(\cond(\chi)) & \text{if } 2 \mid n_0 \\
		\eta_0(4-\delta)\chi \left( \tfrac{4-\delta}{\delta} \right) \cdot \tau_0 q^{\frac{1}{2}} \gamma(1,\chi\eta_0,\psi) \id_{n_0}(\cond(\chi)) & \text{if } 2 \nmid n_0
	\end{cases}.
\end{multline}

\noindent Inserting \eqref{eq: DWtQEleF1Trans3} with $\delta = \tau^{-1} \Tr(x_{\tau} \alpha)^2$ into the integral representation of $H_{2n_0,n_0}^{b,\tau}$ we get
\begin{equation} \label{eq: h2n0-BdRed3}
	\widetilde{h}_{2n_0,n_0}^{b,1}(\chi) \ll \extnorm{ \int_{\substack{\bL^1 \cap \vO_{\bL} \\ \Tr(\alpha)^2 \in \vO_{\F}^{\times} - 4(1+\vP_{\F})}} \beta(\alpha) \cdot \chi \left( \tfrac{4-\Tr(\alpha)^2}{\Tr(\alpha)^2} \right) \eta_0^{n_0} \left( 4 - \Tr(\alpha)^2 \right) \ud \alpha } \cdot q^{-\frac{3n_0}{2}} \id_{n_0}(\cond(\chi)).
\end{equation}
\begin{equation} \label{eq: h2n0-BdRed4}
	\widetilde{h}_{2n_0,n_0}^{b,\varepsilon}(\chi) \ll \extnorm{ \int_{\substack{\bL^1 \cap \vO_{\bL} \\ \Tr(x_{\varepsilon} \alpha) \in \vO_{\F}^{\times}}} \beta(x_{\varepsilon} \alpha) \cdot \chi \left( \tfrac{4 \tau -\Tr(x_{\varepsilon} \alpha)^2}{\Tr(x_{\varepsilon} \alpha)^2} \right) \eta_0^{n_0} \left( 4 - \tau^{-1} \Tr(x_{\varepsilon} \alpha)^2 \right) \ud \alpha } \cdot q^{-\frac{3n_0}{2}} \id_{n_0}(\cond(\chi)).
\end{equation}
	In the case of split, resp. unramified $\bL/\F$ we apply the change of variables $t = \tfrac{\alpha + \alpha^{-1}}{\alpha - \alpha^{-1}}$, resp. $t = \tfrac{\alpha + \alpha^{-1}}{\alpha - \alpha^{-1}} \tfrac{1}{\sqrt{\varepsilon}}$ for $\widetilde{h}_{2n_0,n_0}^{b,1}(\chi)$; $t = \tfrac{\alpha + \varepsilon \alpha^{-1}}{\alpha - \varepsilon \alpha^{-1}}$, resp. $t = \tfrac{x_{\varepsilon}\alpha + \overline{x_{\varepsilon}\alpha}}{x_{\varepsilon}\alpha - \overline{x_{\varepsilon}\alpha}} \tfrac{1}{\sqrt{\varepsilon}}$ for $\widetilde{h}_{2n_0,n_0}^{b,\varepsilon}(\chi)$. The inner integrals in \eqref{eq: h2n0-BdRed3} and \eqref{eq: h2n0-BdRed4}, denoted by $I_1$ and $I_{\varepsilon}$ respectively, become
\begin{multline*} 
	I_1 = I_{\varepsilon} = 2 \chi_0\chi(-1) \int_{\vO_{\F}^{\times} - \sideset{}{_{\pm}} \bigcup (\pm 1+\vP_{\F})} \chi_0 \left( \tfrac{1+t}{1-t} \right) \chi^{-2}(t) \eta_0^{n_0}(1-t^2) \ud t \quad {\rm resp.} \\
	\begin{cases} 
		I_1 = 2\beta(\sqrt{\varepsilon})\chi^{-1}(-\varepsilon) \eta_0^{n_0}(-\varepsilon) \int_{\substack{\vO_{\F}^{\times} \\ 1-\varepsilon t^2 \in -\varepsilon (\vO_{\F}^{\times})^2}} \beta(1+t \sqrt{\varepsilon}) \chi^{-2}(t) \ud t \\
		I_{\varepsilon} = 2\beta(\sqrt{\varepsilon})\chi^{-1}(-\varepsilon) \eta_0^{n_0}(-1) \int_{\substack{\vO_{\F}^{\times} \\ 1-\varepsilon t^2 \in - (\vO_{\F}^{\times})^2}} \beta(1+t \sqrt{\varepsilon}) \chi^{-2}(t) \ud t
	\end{cases}.
\end{multline*}
	We apply Lemma \ref{lem: JacTypeSRC} (2) to bound the above integrals and conclude the desired inequalities.
\end{proof}

\begin{lemma} \label{lem: e1DWt-FineBd4}
	Let $n_0 \geq \max(2\cond(\Pi), \cond(\Pi)+2)$ and $m \leq \cond(\Pi)$. Then for any $\delta \in \varpi_{\F}^{2(n_0-m)}\vO_{\F}^{\times}$ we have
	$$ \int_{\F-\vO_{\F}} \VorH_{\Pi,\psi} \circ \Mult_{-1} \left( \Trans(\delta).F_m \right)(t) \cdot \psi(-t)\chi^{-1}(t)\norm[t]^{-\frac{1}{2}} \ud^{\times} t = 0. $$
	Consequently, we get for $\tau \in \{ 1,\varepsilon \}$ and any unitary $\chi$ the vanishing of $\widetilde{h}_{2n_0,m}^{a,\tau}(\chi) = 0$.
\end{lemma}
\begin{proof}
	For simplicity we only consider the case $m \geq 2$. By the local functional equation and Lemma \ref{lem: MellinQEleF1} we have (recall the parameters of $L(s,\Pi \otimes \xi)$ introduced in \eqref{eq: VHTestFLevelnMellin})
\begin{multline*} 
	\int_{\F^{\times}} \VorH_{\Pi,\psi}(F_m)(t) \xi^{-1}(t) \norm[t]^{-s} \ud^{\times} t = \id_{\cond(\xi)=m} \cdot \zeta_{\F}(1) \gamma \left( 1, \xi^{-2}, \psi \right) \varepsilon \left( \tfrac{1}{2}, \Pi \otimes \xi, \psi \right) q^{\frac{\cond(\Pi \otimes \xi)}{2}} \\
	\cdot q^{(2m-\condL(\Pi \otimes \xi))s} \sideset{}{_{j=1}^{d(\Pi \otimes \xi)}} \prod \frac{q^s-a_j}{1-\overline{a_j}q^{-1+s}}.
\end{multline*}
	As a Laurent series in $X:=q^s$, the right hand side contains $X^k$ only for $k \geq 2m-\condL(\Pi \otimes \xi)$. We have
	$$ \condL(\Pi \otimes \xi) \begin{cases}
		= \cond(\Pi \otimes \xi) \leq \cond(\Pi)+3\cond(\xi) \leq 4\cond(\Pi) & \text{if } \xi \notin E(\Pi) \\
		\leq 2\cond(\Pi)+3 & \text{if } \xi \in E(\Pi)
	\end{cases} $$
by \cite{BuH97} and Lemma \ref{lem: ExpCharBds} respectively. Therefore the support of $\VorH_{\Pi,\psi} \circ \Mult_{-1} \left( \Trans(\delta).F_m \right)$ is contained in
	$$ \delta \vP_{\F}^{2m-\max(4\cond(\Pi), 2\cond(\Pi)+3)} = \vP_{\F}^{2n_0-\max(4\cond(\Pi), 2\cond(\Pi)+3)} \subset \vO_{\F}, $$
proving the desired vanishing.
\end{proof}

	\subsection{The Bounds of Dual Weight}
	
\begin{lemma} \label{lem: h2-Bdn0=1}
	Suppose $\cond(\Pi)=0$ and $n_0=1$. Then we have $\widetilde{h}_2^-(\chi) \leq q^{-\frac{3}{2}} \id_{\cond(\chi)=0} + q^{-1} \id_{\cond(\chi)=1}$.
\end{lemma}
\begin{proof}
	Necessarily we have $\Pi = \mu_1 \boxplus \mu_2 \boxplus \mu_3$ for unramified $\mu_j$. Note that $H_2^a$ (resp. $H_2^b$) is related to $F_0$ (resp. $F_1$) by the formulae:
	$$ H_2^a = \varepsilon_{\bL} \id_{\eta_0(-1)=\varepsilon_{\bL}} \beta(\sqrt{-1}) \cdot \left( \Trans(\varpi_{\F}^2).F_0 \right) + \varepsilon_{\bL} \id_{\eta_0(-1)=-\varepsilon_{\bL}} \beta(\sqrt{-\varepsilon}) \cdot \left( \Trans(\varepsilon^{-1}\varpi_{\F}^2).F_0 \right), $$
	$$ H_2^b = \tfrac{\varepsilon_{\bL}q}{2} \int_{\substack{\bL^1 \cap \vO_{\bL} \\ \Tr(\alpha) \in \vO_{\F}^{\times}}} \beta(\alpha) \cdot \left( \Trans(\Tr(\alpha)^2).F_1 \right) \ud \alpha + \tfrac{\varepsilon_{\bL}q}{2} \int_{\substack{\bL^1 \cap \vO_{\bL} \\ \Tr(x_{\varepsilon} \alpha) \in \vO_{\F}^{\times}}} \beta(x_{\varepsilon} \alpha) \cdot \left( \Trans(\varepsilon^{-1} \Tr(x_{\varepsilon} \alpha)^2).F_1 \right) \ud \alpha. $$
	For any $\delta_0 \in \varpi_{\F}^2 \vO_{\F}^{\times}$ and $\delta_1 \in \vO_{\F}^{\times}$ we have
	$$ \int_{\F-\vO_{\F}} \VorH_{\Pi,\psi} \circ \Mult_{-1} \left( \Trans(\delta_0).F_0 \right)(t) \psi(-t)\chi(t)\norm[t]^{-\frac{1}{2}} \ud^{\times}t = q^{-\frac{5}{2}} \int_{\varpi_{\F}^{-1}\vO_{\F}^{\times}} \VorH_{\Pi,\psi}(F_0)(\delta_0^{-1}t) \psi(-t)\chi(t) \ud^{\times}t, $$
	$$ \int_{\F-\vO_{\F}} \VorH_{\Pi,\psi} \circ \Mult_{-1} \left( \Trans(\delta_1).F_1 \right)(t) \psi(-t)\chi(t)\norm[t]^{-\frac{1}{2}} \ud^{\times}t = q^{-\frac{5}{2}} \int_{\varpi_{\F}^{-1}\vO_{\F}^{\times}} \VorH_{\Pi,\psi}(F_1)(\delta_1^{-1}t) \psi(-t)\chi(t) \ud^{\times}t, $$
by inspecting the supports of $\VorH_{\Pi,\psi}(F_0)$ and $\VorH_{\Pi,\psi}(F_1)$ given in Corollary \ref{cor: VHQEleF1Bis}. Define
	$$ K(y) := \tau_0 q \id_{\varpi_{\F}^{-1} \vO_{\F}^{\times}}(y) \int_{(\varpi_{\F}^{-1}\vO_{\F}^{\times})^2} \psi \left( t_1+t_2 - \tfrac{t_1t_2}{4y} \right) \eta_0 \left( \tfrac{4y}{t_1t_2} \right) \ud t_1 \ud t_2,  $$
	$$ \widetilde{K}(\delta, \chi) := q^{-\frac{5}{2}} \int_{\varpi_{\F}^{-1}\vO_{\F}^{\times}} K(\delta^{-1}t) \psi(-t)\chi(t) \ud^{\times}t. $$
	$$ I_1(\chi) := \int_{\substack{\bL^1 \cap \vO_{\bL} \\ \Tr(\alpha) \in \vO_{\F}^{\times}}} \beta(\alpha) \widetilde{K}(\Tr(\alpha)^2,\chi) \ud \alpha, \quad I_2(\chi) := \int_{\substack{\bL^1 \cap \vO_{\bL} \\ \Tr(x_{\varepsilon} \alpha) \in \vO_{\F}^{\times}}} \beta(x_{\varepsilon} \alpha) \widetilde{K}(\varepsilon^{-1} \Tr(x_{\varepsilon} \alpha)^2,\chi) \ud \alpha. $$
	The formulae of $\VorH_{\Pi,\psi}(F_0)$ and $\VorH_{\Pi,\psi}(F_1)$ in Corollary \ref{cor: VHQEleF1Bis} easily imply the bound
\begin{equation} \label{eq: h2-Bd}
	\widetilde{h}_2^-(\chi) \ll q^{-\frac{5}{2}} \id_{\cond(\chi) \leq 1} + q \extnorm{I_1(\chi) + I_2(\chi)}.
\end{equation}
	An easy change of variables gives
\begin{multline} \label{eq: HypGK}
	\widetilde{K}(\delta, \chi) = \tau_0 q^{\frac{1}{2}} \zeta_{\F}(1) \int_{(\vO_{\F}^{\times})^3} \psi \left( \tfrac{t_1+t_2 - \frac{t_1t_2 \delta}{4t} - t}{\varpi_{\F}} \right) \eta_0 \left( \tfrac{4t}{\delta t_1t_2} \right) \chi \left( \tfrac{t}{\varpi_{\F}} \right) \ud t_1 \ud t_2 \ud t  \\
	= \tau_0 q^{\frac{3}{2}} \zeta_{\F}(1) \int_{\substack{(\vO_{\F}^{\times})^4 \\ t_1t_2 \equiv 4 \delta^{-1} t_3t_4 \pmod{\vP_{\F}}}} \psi \left( \tfrac{t_1+t_2-t_3-t_4}{\varpi_{\F}} \right) \chi \left( \tfrac{t_3}{\varpi_{\F}} \right) \eta_0(t_4) \ud t_1 \ud t_2 \ud t_3 \ud t_4.
\end{multline}
	Note that the function $t \mapsto K(\delta^{-1}t)\psi(-t)$ is invariant by $1+\vP_{\F}$, hence $\widetilde{K}(\delta,\chi)$ is non-zero only if $\cond(\chi) \leq 1$. If $\cond(\chi)=0$ then one simplifies \eqref{eq: HypGK} by performing the integrals over $t_1, t_3, t_4, t_2$ in order
\begin{equation} \label{eq: HypGKCond0}
	\widetilde{K}(\delta, \chi) = \chi(\varpi_{\F})^{-1} \zeta_{\F}(1) \left( q^{-\frac{5}{2}} + q^{-\frac{3}{2}} \cdot \id_{\vO_{\F}^{\times}}(\delta - 4) \eta_0 \left( \tfrac{\delta}{\delta - 4} \right) \right). 
\end{equation}
	Inserting \eqref{eq: HypGKCond0} into \eqref{eq: h2-Bd} we see that the integrands $\widetilde{K}(\cdot)$ are constant, equal to $\chi(\varpi_{\F})^{-1} \zeta_{\F}(1) q^{-\frac{3}{2}} (q^{-1} + \varepsilon_{\bL})$, in both integrals. We obtain the bound $\widetilde{h}_2^-(\chi) \ll q^{-\frac{3}{2}}$ by
	$$ \int_{\substack{\bL^1 \cap \vO_{\bL} \\ \Tr(\alpha) \in \vO_{\F}^{\times}}} \beta(\alpha) \ud \alpha = - q^{-1} \id_{\eta_0(-1)=\varepsilon_{\bL}} \beta(\sqrt{-1}), \quad \int_{\substack{\bL^1 \cap \vO_{\bL} \\ \Tr(x_{\varepsilon} \alpha) \in \vO_{\F}^{\times}}} \beta(x_{\varepsilon} \alpha) \ud \alpha = - q^{-1 }\id_{\eta_0(-1)=-\varepsilon_{\bL}} \beta(\sqrt{-\varepsilon}). $$
	If $\cond(\chi)=1$, then regarding $\chi$ and $\eta_0$ as characters of $\fF_q^{\times}$ we rewrite $\widetilde{K}$ as
	$$ \widetilde{K}(\delta, \chi) = - \tau_0 q^{-1} \zeta_{\F}(1) \cdot H(4\delta^{-1},q; (\id, \id), (\chi^{-1},\eta_0)), $$
where $H(\cdot)$ is precisely the hypergeometric sum of Katz appeared in \cite[\S 5.4]{WX23+}. In the case of split, resp. unramified $\bL/\F$ we apply the change of variables $t = \tfrac{\alpha - \alpha^{-1}}{\alpha + \alpha^{-1}}$, resp. $t = \tfrac{\alpha - \alpha^{-1}}{\alpha + \alpha^{-1}} \sqrt{\varepsilon}$ for $I_1(\chi)$; $t = \tfrac{\alpha - \varepsilon \alpha^{-1}}{\alpha + \varepsilon \alpha^{-1}}$, resp. $t = \tfrac{x_{\varepsilon}\alpha - \overline{x_{\varepsilon}\alpha}}{x_{\varepsilon}\alpha + \overline{x_{\varepsilon}\alpha}} \sqrt{\varepsilon}$ for $I_2(\chi)$. These integrals become
\begin{multline*} 
	I_1(\chi) = I_2(\chi) = 2 \int_{\vO_{\F} - \sideset{}{_{\pm}} \bigcup (\pm 1+\vP_{\F})} \chi_0 \left( \tfrac{1+t}{1-t} \right) \widetilde{K} \left( \tfrac{4}{1-t^2}, \chi \right) \ud t \quad {\rm resp.} \\
	\begin{cases} 
		I_1(\chi) = 2\beta(\sqrt{\varepsilon})^{-1} \int_{\vO_{\F}} \beta(t + \sqrt{\varepsilon}) \widetilde{K} \left( \tfrac{4}{1-\varepsilon^{-1}t^2}, \chi \right) \id_{(\vO_{\F}^{\times})^2}(1-\varepsilon^{-1}t^2) \ud t \\
		I_2(\chi) = 2\beta(\sqrt{\varepsilon})^{-1} \int_{\vO_{\F}} \beta(t + \sqrt{\varepsilon}) \widetilde{K} \left( \tfrac{4}{1-\varepsilon^{-1}t^2}, \chi \right) \id_{\varepsilon (\vO_{\F}^{\times})^2}(1-\varepsilon^{-1}t^2) \ud t
	\end{cases}.
\end{multline*}
	In the case of unramified $\bL/\F$ we recognize and get the bound (with $H(\cdot)$ defined in \cite[\S 5.3.2]{WX23+})
	$$ q \extnorm{I_1(\chi) + I_2(\chi)} = 2 \zeta_{\F}(1) q^{-\frac{3}{2}} \extnorm{ \sideset{}{_{\alpha \in \fF_q}} \sum \beta(\alpha + \sqrt{\varepsilon}) H(1-\alpha^2 \varepsilon^{-1},q; (\id, \id), (\chi^{-1},\eta_0))} \ll q^{-1} $$
	by \cite[Lemma 5.10]{WX23+}. In the case of split $\bL/\F$ we recognize and \emph{claim} a bound of
\begin{equation} \label{eq: PYExpSVar}
	q \extnorm{I_1(\chi) + I_2(\chi)} = 4 \zeta_{\F}(1) q^{-\frac{3}{2}} \extnorm{ \sideset{}{_{\alpha \in \fF_q - \{ \pm 1\}}} \sum \chi_0 \left( \tfrac{\alpha +1}{\alpha-1} \right) H(1-\alpha^2,q; (\id, \id), (\chi^{-1},\eta_0))} \ll q^{-1}. 
\end{equation}
	The method of the proof of \cite[Lemma 5.10]{WX23+} works through to bound the above sum. In fact the essence of that proof is to show:
\begin{itemize}
	\item[(1)] the $\ell$-adic sheaf associated with the function $\alpha \mapsto H(1-\alpha^2 \varepsilon^{-1},q; (\id, \id), (\chi^{-1},\eta_0))$ has rank $2$, 
	\item[(2)] while the sheaf associated with the function $\alpha \mapsto \beta(\alpha + \sqrt{\varepsilon})$ has rank $1$.
\end{itemize}
	The analogues for the above split case are
\begin{itemize}
	\item[(1')] the $\ell$-adic sheaf associated with the function $\alpha \mapsto H(1-\alpha^2,q; (\id, \id), (\chi^{-1},\eta_0))$ has rank $2$, 
	\item[(2')] while the sheaf associated with the function $\alpha \mapsto \chi_0 \left( \tfrac{\alpha+1}{\alpha-1} \right)$ has rank $1$.
\end{itemize}
	Note that (1') and (1) follow from the same argument that the sheaf associated with $t \mapsto H(t,q; \cdots)$ has geometric monodromy group $\SL_2$, which does not admit finite index subgroup; while (2') is trivially true because the Kummer sheaf has rank $1$. In fact, both (2) and (2') are known to Weil \cite[Appendix \Rmnum{5}]{We74}. We conclude the claimed bound \eqref{eq: PYExpSVar}.
\end{proof}
	
\begin{proposition} \label{prop: DWtBde=1}
	With the test function $H$ given by \eqref{eq: TestFROI} the dual weight function is bounded as
	$$ \widetilde{h}(\chi) \ll_{\epsilon} \Cond(\Pi)^{2+\epsilon} q^{-n_0+\epsilon} \left( \id_{\leq \max(n_0,6\cond(\Pi))}(\cond(\chi)) + q^{\frac{1}{2}} \id_{\eta_0(-1) = \varepsilon_{\bL}} \id_{2 \nmid n_0 \geq 3} \id_{\cond(\chi) = n_0} \id_{\Ecp(\beta)}(\chi^2) \right), $$
	where we recall $\varepsilon_{\bL} = \eta_{\bL/\F}(\varpi_{\F})$, and $\Ecp(\beta)$ is given in Lemma \ref{lem: JacTypeSRC} (2).
\end{proposition}
\begin{proof}
	We distinguish two cases $\cond(\Pi) > 0$ and $\cond(\Pi) = 0$.
	
\noindent (1) Suppose $\cond(\Pi) > 0$. We have $a(\Pi) \leq \cond(\Pi)$ by Proposition \ref{prop: StabRang}. If $n_0 \leq 2\cond(\Pi)$, then we may take $n_1=4\cond(\Pi)$. Applying Lemma \ref{lem: DWtInfty}, \ref{lem: DWt+Bd} and \ref{lem: DWt-TrivBd} we deduce
	$$ \widetilde{h}(\chi) \ll \extnorm{\widetilde{h}_{\infty}(\chi)} + \extnorm{\widetilde{h}_c^+(\chi)} + \extnorm{\widetilde{h}_c^-(\chi)} \ll \cond(\Pi) \id_{\leq 6\cond(\Pi)}(\cond(\chi)), $$
	the ``worst'' bound being offered by Lemma \ref{lem: DWt-TrivBd}. If $n_0 \geq 2 \cond(\Pi)$, then we may take $n_1=2n_0$. Applying Lemma \ref{lem: DWtInfty}, \ref{lem: DWt+Bd} and Lemma \ref{lem: e1DWt-FineBd1}-\ref{lem: e1DWt-FineBd4} we deduce
	$$ \widetilde{h}(\chi) \ll q^{-n_0} \left( \id_{\leq n_0}(\cond(\chi)) + q^{\frac{1}{2}} \id_{\eta_0(-1) = \varepsilon_{\bL}} \id_{2 \nmid n_0 \geq 3} \id_{\cond(\chi) = n_0} \id_{\Ecp(\beta)}(\chi) \right). $$
	The stated bound is simply a common upper bound of the above two.
	
\noindent (2) Suppose $\cond(\Pi)=0$. Necessarily we have $\Pi = \mu_1 \boxplus \mu_2 \boxplus \mu_3$ for unramified $\mu_j$ and $a(\Pi)=2$. The argument of the second part in (1) works through also for $n_0 \geq 2$. It remains to consider the case $n_0=1$. Note that $n_1 = 2$, hence the contribution of $H_{\infty}$ is $\widetilde{h}_{\infty}(\chi) \ll q^{-1} \id_{\cond(\chi)=0}$ by Lemma \ref{lem: DWtInfty}. We have
	$$ \widetilde{h}(\chi) = \widetilde{h}_{\infty}(\chi) + \widetilde{h}_2^+(\chi) + \widetilde{h}_2^-(\chi) $$
by \eqref{eq: TestFComp} and the condition $2 \mid n$ following \eqref{eq: FPTestFNR0}. By Lemma \ref{lem: DWt+Bd} we have $\widetilde{h}_2^+(\chi) \ll_{\epsilon} q^{-1+\epsilon} \id_{\cond(\chi)=0}$. By Lemma \ref{lem: h2-Bdn0=1} we have $\widetilde{h}_2^-(\chi) \ll q^{-\frac{3}{2}} \id_{\cond(\chi)=0} + q^{-1} \id_{\cond(\chi)=1}$. The stated bound follows.
\end{proof}

	\subsection{The Bounds of Unramified Dual Weight}
	
\begin{lemma} \label{lem: e1DNDWt-FineBd1}
	Suppose $n_0 + 1 \leq n \leq 2n_0 - 1$ and $n \geq a(\Pi)$. Then the function $\widetilde{h}_{2n}^-(\norm_{\F}^s) \neq 0$ is non-vanishing only if $n=2n_0-1$. Moreover for any $k \in \Z_{\geq 0}$ we have
	$$ \widetilde{H}_{2(2n_0-1)}^-(k; \tfrac{1}{2}) \ll_{k,\epsilon} q^{-\frac{1}{2}+ \left( 2n_0-1 \right) \epsilon}, \quad \widetilde{H}_{2(2n_0-1)}^-(k; -\tfrac{1}{2}) \ll_{k,\epsilon} q^{\frac{1}{2}-2n_0+\left( 2n_0-1 \right) \epsilon}. $$
\end{lemma}
\begin{proof}
	By Corollary \ref{cor: VHQEleF1} and \eqref{eq: FPTestFNR1} the support of $\VorH_{\Pi,\psi}(F_n)$, hence also $\VorH_{\Pi,\psi} \circ \Mult_{-1} (H_{2n})$ are contained in $\varpi_{\F}^{-n} \vO_{\F}^{\times}$. Therefore we get by \eqref{eq: h2n-BdRedNR} the relation $\widetilde{h}_{2n}^-(\norm_{\F}^s) = q^{ns} \widetilde{h}_{2n}^-(\id)$. The stated results follow readily from Lemma \ref{lem: e1DWt-FineBd1}.
\end{proof}

\begin{lemma} \label{lem: h2-DNBdn0=1}
	Suppose $\cond(\Pi)=0$ and $n_0=1$. Then we have for any $k \in \Z_{\geq 0}$
	$$ \widetilde{H}_{2}^-(k; \tfrac{1}{2}) \ll_{k,\epsilon} q^{-(1-\epsilon)}, \quad \widetilde{H}_{2}^-(k; -\tfrac{1}{2}) \ll_{k,\epsilon} q^{-2(1-\epsilon)}. $$
\end{lemma}
\begin{proof}
	Similarly as the previous lemma, we have $\widetilde{h}_2^-(\norm_{\F}^s) = q^s \cdot \widetilde{h}_2^-(\id)$ by the inspection of the supports of $\VorH_{\Pi,\psi} \circ \Mult_{-1}(H_2)$ via Corollary \ref{cor: VHQEleF1Bis} in the proof of Lemma \ref{lem: h2-Bdn0=1}, from which the stated bounds follow.
\end{proof}

\begin{proposition} \label{prop: DNDWtBde=1}
	With the test function $H$ given by \eqref{eq: TestFROI} we have for any $k \in \Z_{\geq 0}$
	$$ \widetilde{H}(k; \tfrac{1}{2}) \ll_{k,\epsilon} \Cond(\Pi)^{\frac{1}{2}} \cdot q^{2n_0 \epsilon}, \quad \widetilde{H}(k; -\tfrac{1}{2}) \ll_{k,\epsilon} \Cond(\Pi)^{4+\epsilon} \cdot q^{2n_0 \epsilon}. $$
\end{proposition}
\begin{proof}
	The argument is quite similar to Proposition \ref{prop: DWtBde=1}. We distinguish two cases $\cond(\Pi) > 0$ and $\cond(\Pi) = 0$.
	
\noindent (1) Suppose $\cond(\Pi) > 0$. We take $n_1 = 4 \cond(\Pi)$ if $n_0 \leq 2 \cond(\Pi)$ and apply Lemma \ref{lem: DNDWt+Bd} \& \ref{lem: DNDWt-TrivBd}. If $n_0 \geq 2 \cond(\Pi)$ we take $n_1 = 2n_0$. The proofs of Lemma \ref{lem: e1DWt-FineBd1}-\ref{lem: e1DWt-FineBd4} imply $\widetilde{h}_{2n_0}^-(\norm_{\F}^s) = 0$. Together with Lemma \ref{lem: DNDWt+Bd} \& \ref{lem: e1DNDWt-FineBd1} we conclude the stated bounds in this case.

\noindent (2) Suppose $\cond(\Pi) = 0$. Necessarily we have $\Pi = \mu_1 \boxplus \mu_2 \boxplus \mu_3$ for unramified $\mu_j$ and $a(\Pi)=2$. The argument of the second part in (1) works through also for $n_0 \geq 2$. It remains to consider the case $n_0=1$. Note that $n_1 = 2$. We conclude the stated bounds by Lemma \ref{lem: DNDWt+Bd} \& \ref{lem: h2-DNBdn0=1}.
\end{proof}

\section{Dual Weight Functions: Ramified Cases}

	\subsection{Second Quadratic Elementary Functions}
	
\begin{definition} \label{def: QEleF2}
	Let $\bL=\F[\sqrt{\varpi_{\F}}]$ and recall the associated quadratic character $\eta_{\bL/\F}$. For any $n \in \Z_{\geq 0}$ we define the \emph{second quadratic elementary functions} $G_n \in \Cont_c^{\infty}(\F^{\times})$ by 
	$$ G_n(y^2) := \id_{v(y)=-n} \cdot \sideset{}{_{\pm}} \sum \eta_{\bL/\F}(\pm y) \psi(\pm y), $$
	and are supported in the subset of \emph{square} elements of $\F^{\times}$.
\end{definition}

\begin{lemma} \label{lem: MellinQEleF2}
	Let $\chi$ be a quasi-character of $\F^{\times}$. We have
\begin{multline*} 
	\int_{\F^{\times}} G_n(y) \chi(y) \ud^{\times} y = \\
	\begin{cases}
		\id_{\cond(\chi^2)=n} \cdot \zeta_{\F}(1) \gamma(1,\eta_{\bL/\F}\chi^{-2},\psi) & \text{if } n \geq 2 \\
		\id_{\cond(\eta_{\bL/\F}\chi^2)=1} \cdot \zeta_{\F}(1) \gamma(1,\eta_{\bL/\F}\chi^{-2},\psi) - \id_{\cond(\eta_{\bL/\F}\chi^2)=0} \cdot \zeta_{\F}(1) q^{-1} (\eta_{\bL/\F}\chi^{-2})(\varpi_{\F}) & \text{if } n = 1 \\
		\id_{\cond(\eta_{\bL/\F}\chi^2)=0} & \text{if } n = 0
	\end{cases}. 
\end{multline*}
\end{lemma}
\begin{proof}
	The proof is quite similar to Lemma \ref{lem: MellinQEleF1} and is omitted.
\end{proof}

\begin{corollary} \label{cor: VHQEleF2}
	Let $n \geq a(\Pi) (\geq 2)$, the stability barrier of $\Pi$ (see Proposition \ref{prop: StabRang}). We have
	$$ \VorH_{\Pi,\psi}(G_n)(y) = \eta_0(-1)^n \eta_0(-2) \id_{\varpi_{\F}^{-n} \vO_{\F}^{\times}}(y) \cdot 
	\begin{cases}
		q^{\frac{3n}{2}} \psi(4y) \eta_0(4y) & \text{if } 2 \mid n \\
		\tau_0 q^{\lceil \frac{3n}{2} \rceil} \psi(4y) & \text{if } 2 \nmid n
	\end{cases}. $$
\end{corollary}
\begin{proof}
	The proof is quite similar to Corollary \ref{cor: VHQEleF1}. We only emphasize the difference. We get
	$$ \VorH_{\Pi,\psi}(G_n) \left( \tfrac{y}{\varpi_{\F}^n} \right) = q^{3n} \int_{(\vO_{\F}^{\times})^3} \eta_{\bL/\F} \left( \tfrac{t_3}{\varpi_{\F}^n} \right) \psi \left( \tfrac{t_1+t_2+t_3 + t_1^{-1}t_2^{-1}t_3^2 y}{\varpi_{\F}^n} \right) \ud t_1 \ud t_2 \ud t_3. $$
	Performing the level $\lceil \tfrac{n}{2} \rceil$ regularization to $\ud t_j$ we see that the non-vanishing of the above integral implies $t_1 = 4y(1+u_1)$, $t_2 = 4y(1+u_2)$ and $t_3 = -8y(1+u_3)$ with $u_j \in \vP_{\F}^{\lfloor \frac{n}{2} \rfloor}$ and obtain
	$$ \VorH_{\Pi,\psi}(G_n) \left( \tfrac{y}{\varpi_{\F}^n} \right) = q^{2n} \eta_{\bL/\F} \left( \tfrac{-8y}{\varpi_{\F}^n} \right) \psi \left( \tfrac{4y}{\varpi_{\F}^n} \right) \int_{\vP_{\F}^{\lfloor \frac{n}{2} \rfloor}} \psi \left( \tfrac{4y}{\varpi_{\F}^n} u_2^2 \right) \ud u_2. $$
	We conclude by noting $\eta_{\bL/\F} \left( \varpi_{\F}^{-n} y \right) = \eta_0(-1)^n \eta_0(y)$ for $y \in \vO_{\F}^{\times}$.
\end{proof}

\begin{corollary} \label{cor: VHQEleF2Bis}
	Suppose $\Pi = \mu_1 \boxplus \mu_2 \boxplus \mu_3$ with $\cond(\mu_j)=0$ and $\mu_1\mu_2\mu_3 = \id$. Let $E_i(y) := \mu_i^{-1}(y) \norm[y] \id_{\vO_{\F}}(y)$ and $f_i := (1 - \mu_i(\varpi_{\F}) \Trans(\varpi_{\F})).E_i$. For $f,g \in \intL^1(\F^{\times})$ define $f*g(y) := \int_{\F^{\times}} f(yt^{-1})g(t) \ud^{\times} t$.
\begin{itemize}
	\item[(1)] We have $G_0=0$ unless $4 \mid q-1$, in which case $\left\{ \chi \in \widehat{\vO_{\F}^{\times}} \ \middle| \ \eta_0 = \chi^2 \right\} = \{ \eta_1, \eta_1^{-1} \}$. Extending $\eta_1$ to $\F^{\times}$ by $\eta_1(\varpi_{\F})=1$ we get
	$$ \VorH_{\Pi,\psi}(G_0)(y) = \id_{\varpi_{\F}^{-3} \vO_{\F}^{\times}}(y) \cdot q^{\frac{3}{2}} \cdot \left\{ \gamma(\tfrac{1}{2}, \eta_1, \psi)^3 \eta_1(y) + \gamma(\tfrac{1}{2}, \eta_1^{-1}, \psi)^3 \eta_1^{-1}(y) \right\}. $$
	\item[(2)] We have the formula
\begin{multline*}
	\VorH_{\Pi,\psi}(G_1)(y) = \zeta_{\F}(1) \gamma(1,\eta_{\bL/\F},\psi) \cdot \left( f_1*f_2*f_3 \right)(\varpi_{\F}^{-2}y) + \\
	\eta_{\bL/\F}(-1) \id_{\varpi_{\F}^{-1}\vO_{\F}^{\times}}(y) \cdot \left\{ \int_{(\varpi_{\F}^{-1}\vO_{\F}^{\times})^3} \eta_0(t_3) \psi \left( t_1+t_2+t_3 + \tfrac{t_3^2y}{t_1t_2} \right) \ud \vec{t} + \zeta_{\F}(1) \tau_0 \right\}.
\end{multline*}
\end{itemize}
\end{corollary}
\begin{proof}
	(1) From Lemma \ref{lem: MellinQEleF2} and the local functional equation we deduce for any $\chi \in \widehat{\vO_{\F}^{\times}}$
	$$ \int_{\F^{\times}} \VorH_{\Pi,\psi}(G_0)(y) \chi^{-1}(y) \norm[y]^{-s} \ud^{\times} y = \id_{\chi = \eta_1^{\pm 1}} \cdot q^{3 \left( \frac{1}{2}-s \right)} \gamma \left( \tfrac{1}{2},\chi,\psi \right)^3. $$
	We readily deduce the desired formula for $\VorH_{\Pi,\psi}(G_0)$.
	
\noindent (2) We can write $\VorH_{\Pi,\psi}(G_1) = \VorH_{\Pi,\psi}(G_1)_0 + \VorH_{\Pi,\psi}(G_1)_1$ with the properties
	$$ \VorH_{\Pi,\psi}(G_1)_0(y\delta) = \VorH_{\Pi,\psi}(G_1)_0(y), \quad \forall y \in \F^{\times}, \delta \in \vO_{\F}^{\times}; $$
	$$ \int_{\F^{\times}} \VorH_{\Pi,\psi}(G_1)_0(y) \norm[y]^{-s} \ud^{\times} y = \zeta_{\F}(1) \gamma(1,\eta_{\bL/\F},\psi) \cdot q^{2s} \cdot \sideset{}{_{i=1}^3} \prod \int_{\F^{\times}} f_i(y) \norm[y]^{-s} \ud^{\times} y, $$
\begin{multline*} 
	\int_{\F^{\times}} \VorH_{\Pi,\psi}(G_1)_1(y) \chi^{-1}(y) \norm[y]^{-s} \ud^{\times} y = \\
	\id_{\cond(\chi)=1} \cdot q^{3-s} \int_{\varpi_{\F}^{-1}\vO_{\F}^{\times}} \psi(y) \eta_{\bL/\F}\chi^2(y) \ud^{\times} y \cdot \left( \int_{\varpi_{\F}^{-1} \vO_{\F}^{\times}} \psi(y) \chi^{-1}(y) \ud^{\times}y \right)^3.
\end{multline*}
	We easily identify $\VorH_{\Pi,\psi}(G_1)_0(y) = \zeta_{\F}(1) \gamma(1,\eta_{\bL/\F},\psi) \cdot \left( f_1*f_2*f_3 \right)(\varpi_{\F}^{-2}y)$. We also deduce that $\mathrm{supp}(\VorH_{\Pi,\psi}(G_1)_1) \subset \varpi_{\F}^{-1} \vO_{\F}^{\times}$, and for $y \in \vO_{\F}^{\times}$
\begin{multline*}
	\VorH_{\Pi,\psi}(G_1)_1 \left( \tfrac{y}{\varpi_{\F}} \right) = q^{3} \int_{(\vO_{\F}^{\times})^3} \eta_{\bL/\F} \left( \tfrac{t_4}{\varpi_{\F}} \right) \psi \left( \tfrac{t_2+t_3+t_4 + t_2^{-1}t_3^{-1}t_4^2y}{\varpi_{\F}} \right) \ud t_2 \ud t_3 \ud t_4 \\
	- \zeta_{\F}(1) q^{3} \int_{(\vO_{\F}^{\times})^4} \eta_{\bL/\F} \left( \tfrac{t_4}{\varpi_{\F}} \right) \psi \left( \tfrac{t_1+t_2+t_3+t_4}{\varpi_{\F}} \right) \ud \vec{t}.
\end{multline*}
	The second summand is equal to $\zeta_{\F}(1) \eta_{\bL/\F}(-1) \tau_0$. We conclude by re-numbering the variables.
\end{proof}

	\subsection{Further Reductions}
	
	The functions $G_n$ are ``building blocks'' of our test functions $H_c$ in \eqref{eq: TestFComp} when $\bL/\F$ is ramified. In fact writing $\lambda_{\bL} := \lambda(\bL/\F, \psi) \eta_{\bL/\F}(2)$ and applying Lemma \ref{lem: TestFAsymp} we can rewrite the summands of $H_c$ as
\begin{align}
	&H_{2n+1} = 0 \quad \text{if } n \neq n_0/2, \label{eq: FPTestFR0} \\
	&H_{2n} = \begin{cases}
		E_n & \text{if } n_0+1 \leq n \leq n_1-1 \\
		\lambda_{\bL}q^n \int_{\substack{\bL^1 \\ \Tr(\alpha) \in 2(1+\vP_{\F}^{n_0-1})}} \beta(\alpha) \cdot \left( \Trans \left( \Tr(\alpha)^2 \right).G_n \right) \ud \alpha & \text{if } n = n_0 \\
		\lambda_{\bL}q^n \int_{\substack{\bL^1 \\ \Tr(\alpha) \in 2(1+\varpi_{\F}^{2n-n_0-1} \vO_{\F}^{\times})}} \beta(\alpha) \cdot \left( \Trans \left( \Tr(\alpha)^2 \right).G_n \right) \ud \alpha & \text{if } \tfrac{n_0}{2}+1 \leq n \leq n_0-1
	\end{cases}. \label{eq: FPTestFR1}
\end{align}
	The decomposition of $H_{n_0+1}$ is subtler and really goes in the direction of expression in terms of the \emph{quadratic elementary functions}. We shall write (the second numeric parameter $m$ in subscript always indicates the parameter of the relevant quadratic elementary function)
\begin{equation} \label{eq: FPTestFR2}
	H_{n_0+1} = \tfrac{\lambda(\bL/\F,\psi)}{2} q^{\frac{n_0+1}{2}} \cdot \sideset{}{_{m=0}^{\frac{n_0}{2}}} \sum H_{n_0+1,m},
\end{equation}
	where the summands are given by (below $m \geq 1$)
	$$ H_{n_0+1,m} = \int_{\substack{\bL^1 \\ \Tr(\varpi_{\bL} \alpha) \in \varpi_{\F}^{n_0/2+1-m} \vO_{\F}^{\times}}} \beta(\varpi_{\bL} \alpha) \cdot \eta_{\bL/\F} \left( \Tr(\varpi_{\bL}\alpha) \right) \cdot \Trans \left( -\varpi_{\F}^{-1} \Tr(\varpi_{\bL} \alpha)^2 \right).G_m \ud \alpha, $$
	$$ H_{n_0+1,0} = \int_{\substack{\bL^1 \\ \Tr(\varpi_{\bL} \alpha) \in \vP_{\F}^{n_0/2+1}}} \beta(\varpi_{\bL} \alpha) \ud \alpha \cdot \Trans \left( (-\varpi_{\F})^{n_0+1} \right).G_0. $$
	
\begin{lemma} \label{lem: JacTypeSRCBis}
	Let $\chi$ be a (unitary) character of $\F^{\times}$ with $\cond(\chi)=n \leq \tfrac{n_0}{2}$. Recall the additive parameter $c_{\beta}$ (resp. $c_{\chi}$) in Lemma \ref{lem: AddParMultChar} (resp. Remark \ref{rmk: AddParMultChar}). We have
	$$ \int_{\vO_{\F}^{\times}} \beta \left( 1+\varpi_{\F}^{\frac{n_0}{2}-n} t \varpi_{\bL} \right) \chi(t) \ud t \ll q^{-\frac{n}{2}}. $$
\end{lemma}
\begin{proof}
	Since the proof is quite similar to the one of Lemma \ref{lem: JacTypeSRC} (1), we skip some details. The case for $n=0$ is easy. If $n=1$, then $t \mapsto \beta \left( 1+ \varpi_{\F}^{n_0-n} t \sqrt{\varepsilon} \right)$ is a non-trivial additive character of $\vO_{\F}$ and the bound follows from the one for Gauss sums. Assume $n \geq 2$. We perform a level $\lceil \tfrac{n}{2} \rceil$ regularization to $\ud t$ and get the non-vanishing condition
	$$ \tfrac{2 c_{\beta} t}{1-\varpi_{\F}^{n_0+1-2-n} t^2} + c_{\chi} \in \vP_{\F}^{\lfloor \frac{n}{2} \rfloor} \quad \Leftrightarrow \quad 2 c_{\beta} t + c_{\chi}(1-\varpi_{\F}^{n_0+1-2n} t^2) \in \vP_{\F}^{\lfloor \frac{n}{2} \rfloor}, $$
	which has a unique solution $t \in t_0 + \vP_{\F}^{\lfloor \frac{n}{2} \rfloor}$ with $t_0 \in \vO_{\F}^{\times}$ by Hensel's lemma. Consequently we get
\begin{equation} \label{eq: IntermCanBis}
	\int_{\vO_{\F}^{\times}} \beta \left( 1+\varpi_{\F}^{\frac{n_0}{2}-n} t \varpi_{\bL} \right) \chi(t) \ud t = \int_{t_0 + \vP_{\F}^{\lfloor \frac{n}{2} \rfloor}} \beta \left( 1+\varpi_{\F}^{\frac{n_0}{2}-n} t \varpi_{\bL} \right) \chi(t) \ud t \ll q^{-\lfloor \frac{n}{2} \rfloor}.
\end{equation}
	If $2 \mid n$ then we are done (for both (1) and (2)). Otherwise let $n=2m+1$. We may assume
	$$ 2 c_{\beta} t_0 + c_{\chi}(1-\varpi_{\F}^{n_0+1-2n} t_0^2 \varepsilon) = 0, $$
	make the change of variables $t = t_0(1+\varpi_{\F}^m u)$, and continue (\ref{eq: IntermCanBis}) as
	$$ \int_{\vO_{\F}^{\times}} \beta \left( 1+\varpi_{\F}^{\frac{n_0}{2}-n} t \varpi_{\bL} \right) \chi(t) \ud t = q^{-m} \beta \left( 1+\varpi_{\F}^{\frac{n_0}{2}-n} t_0 \varpi_{\bL} \right) \chi(t_0) \int_{\vO_{\F}} \psi \left( -\tfrac{c_{\chi}u^2}{2 \varpi_{\F}} \right) du \ll q^{-\frac{n}{2}} $$
	and conclude the proof.
\end{proof}
	
\begin{lemma} \label{lem: e2DWt-FineBd1}
	Suppose $\tfrac{n_0}{2} + 1 \leq n \leq n_0$ and $n \geq a(\Pi)$. Then we have
	$$ \widetilde{h}_{2n}^-(\chi) \ll q^{-\frac{n_0+1}{2}} \id_{n_0+1-n}(\cond(\chi \eta_0^{n+1})) + q^{-\frac{n_0+1}{2}} \id_{n=n_0} \id_{\cond(\chi \eta_0)=0}. $$
\end{lemma}
\begin{proof}
	(1) First consider $n < n_0$. By Corollary \ref{cor: VHQEleF2} we have for any $\delta \in 1+\varpi_{\F}^{2n-n_0-1} \vO_{\F}^{\times}$
\begin{multline} \label{eq: DWtQEleF2Trans1}
	\int_{\F-\vO_{\F}} \VorH_{\Pi,\psi} \circ \Mult_{-1} \left( \Trans \left( 4\delta \right).G_n \right)(t) \cdot \psi(-t) \chi^{-1}(t) \norm[t]^{-\frac{1}{2}} \ud^{\times} t = \\
	q^{-n} \eta_0(2(-1)^{n+1}) \zeta_{\F}(1) \cdot \begin{cases}
		\chi^{-1} \eta_0 \left( \tfrac{\delta}{1-\delta} \right) \cdot \gamma(1,\chi \eta_0,\psi) \id_{n_0+1-n}(\cond(\chi \eta_0)) & \text{if } 2 \mid n \\
		\chi^{-1} \left( \tfrac{\delta}{1-\delta} \right) \cdot \tau_0 q^{\frac{1}{2}} \cdot \gamma(1,\chi,\psi) \id_{n_0+1-n}(\cond(\chi)) & \text{if } 2 \nmid n
	\end{cases},
\end{multline}
	where we used $\eta_0(\delta)=1$. Inserting \eqref{eq: DWtQEleF2Trans1} with $\delta = 4^{-1} \Tr(\alpha)^2$ into \eqref{eq: FPTestFR1} we get
\begin{multline} \label{eq: h2n-BdRedR}
	\widetilde{h}_{2n}^-(\chi) = \int_{\F-\vO_{\F}} \VorH_{\Pi,\psi} \circ \Mult_{-1} (H_{2n})(t) \cdot \psi(-t) \chi^{-1}(t) \norm[t]^{-\frac{1}{2}} \ud^{\times} t \ll \\
	\extnorm{\int_{\substack{\bL^1 \\ \Tr(\alpha) \in 2(1+\varpi_{\F}^{2n-n_0-1} \vO_{\F}^{\times})}} \beta(\alpha) \cdot \chi^{-1} \eta_0^{n+1} \left( \tfrac{\Tr(\alpha)^2}{4-\Tr(\alpha)^2} \right) \ud \alpha} \cdot q^{-\frac{n_0+1-n}{2}} \id_{n_0+1-n}(\cond(\chi \eta_0^{n+1})).
\end{multline}
	Applying the change of variables $t = \tfrac{\alpha - \alpha^{-1}}{\alpha + \alpha^{-1}} \tfrac{1}{\varpi_{\bL}}$ the inner integral in \eqref{eq: h2n-BdRedR} becomes, taking into account \emph{the measure normalization} in Proposition \ref{prop: MeasNorm} (3)
	$$ q^{-\frac{1}{2}} \cdot \chi\eta_0^{n+1}(-1) \int_{\varpi_{\F}^{n-\frac{n_0}{2}-1} \vO_{\F}^{\times}} \beta(1+t\varpi_{\bL}) \chi^2(t) \ud t. $$
	We apply Lemma \ref{lem: JacTypeSRCBis} (1) to bound the above integrals and conclude.
	
\noindent (2) Consider $n=n_0$. We have for any $\delta \in 1+\vP_{\F}^{n_0-1}$
\begin{multline} \label{eq: DWtQEleF2Trans1Bis}
	\int_{\F-\vO_{\F}} \VorH_{\Pi,\psi} \circ \Mult_{-1} \left( \Trans \left( 4\delta \right).G_n \right)(t) \cdot \psi(-t) \chi^{-1}(t) \norm[t]^{-\frac{1}{2}} \ud^{\times} t = q^{-n_0} \eta_0(-2) \cdot \tau_0 q^{\frac{1}{2}} \cdot \\
	\begin{cases}
		\chi^{-1} \eta_0 \left( \tfrac{\delta}{1-\delta} \right) \cdot \zeta_{\F}(1) \left\{ \gamma(1,\chi \eta_0,\psi) \id_{1}(\cond(\chi \eta_0)) - q^{-1} \id_{0}(\cond(\chi \eta_0)) \right\} & \text{if } \delta \in 1+\varpi_{\F}^{n_0-1} \vO_{\F}^{\times} \\
		\chi\eta_0(\varpi_{\F})^{n_0} \cdot \id_{0}(\cond(\chi \eta_0)) & \text{if } \delta \in 1+\vP_{\F}^{n_0}
	\end{cases}.
\end{multline}
	Inserting \eqref{eq: DWtQEleF2Trans1Bis} with $\delta = 4^{-1} \Tr(\alpha)^2$ into \eqref{eq: FPTestFR1} we get
\begin{multline} \label{eq: h2n-BdRedRBis}
	\widetilde{h}_{2n}^-(\chi) = \int_{\F-\vO_{\F}} \VorH_{\Pi,\psi} \circ \Mult_{-1} (H_{2n})(t) \cdot \psi(-t) \chi^{-1}(t) \norm[t]^{-\frac{1}{2}} \ud^{\times} t \ll \\
	\extnorm{\int_{\substack{\bL^1 \\ \Tr(\alpha) \in 2(1+\varpi_{\F}^{n_0-1} \vO_{\F}^{\times})}} \beta(\alpha) \cdot \chi^{-1} \eta_0 \left( \tfrac{\Tr(\alpha)^2}{4-\Tr(\alpha)^2} \right) \ud \alpha} \cdot q^{-\frac{1}{2}} \id_{1}(\cond(\chi \eta_0)) + \\
	\extnorm{ \int_{\substack{\bL^1 \\ \Tr(\alpha) \in 2(1+\vP_{\F}^{n_0})}} \beta(\alpha) \ud \alpha - \tfrac{\chi\eta_0(\varpi_{\F})^{-1}}{q-1} \int_{\substack{\bL^1 \\ \Tr(\alpha) \in 2(1+\varpi_{\F}^{n_0-1} \vO_{\F}^{\times})}} \beta(\alpha) \ud \alpha } \cdot \id_{0}(\cond(\chi \eta_0)).
\end{multline}
	The first summand is bounded the same way as before. With the same change of variables the inner integrals of the second summand become
	$$ q^{-\frac{1}{2}} \cdot \left( q^{-\frac{n_0}{2}} - \tfrac{\chi\eta_0(\varpi_{\F})^{-1}}{q-1} \int_{\varpi_{\F}^{\frac{n_0}{2}-1}\vO_{\F}^{\times}} \beta(1+t\varpi_{\bL}) \ud t \right). $$
	It is of size $O(q^{-\frac{n_0+1}{2}})$ since the integrand is an additive character of conductor exponent $n_0/2$.
\end{proof}

\begin{lemma} \label{lem: e2DWt-FineBd2}
	Suppose $(2 \leq) a(\Pi) \leq m \leq \tfrac{n_0}{2}$. Then we have for unitary $\chi$
	$$ \widetilde{h}_{n_0+1,m}(\chi) := \int_{\F-\vO_{\F}} \VorH_{\Pi,\psi} \circ \Mult_{-1} \left( H_{n_0+1,m} \right)(t) \psi(-t)\chi^{-1}(t)\norm[t]^{-\frac{1}{2}} \ud^{\times} t \ll q^{-n_0-1} \id_{m > \frac{n_0+1}{3}} \id_m(\cond(\chi)). $$
\end{lemma}
\begin{proof}
	We first use Corollary \ref{cor: VHQEleF2} to obtain for any $\delta \in \varpi_{\F}^{n_0+1-2m} \vO_{\F}^{\times}$ (note that $\eta_0(4-\delta)=1$)
\begin{multline} \label{eq: DWtQEleF2Trans2}
	\int_{\F-\vO_{\F}} \VorH_{\Pi,\psi} \circ \Mult_{-1} \left( \Trans \left( \delta \right).G_m \right)(t) \cdot \psi(-t) \chi^{-1}(t) \norm[t]^{-\frac{1}{2}} \ud^{\times} t \\
	= q^{-\frac{n_0+1}{2}} \eta_0(2(-1)^{m+1}) \zeta_{\F}(1) \id_{m > \frac{n_0+1}{3}} \cdot \begin{cases}
		\chi \left( \tfrac{4-\delta}{\delta} \right) \cdot \gamma(1,\chi\eta_0,\psi) \id_{m}(\cond(\chi)) & \text{if } 2 \mid m \\
		\chi \left( \tfrac{4-\delta}{\delta} \right) \cdot \tau_0 q^{\frac{1}{2}} \gamma(1,\chi,\psi) \id_{m}(\cond(\chi)) & \text{if } 2 \nmid m
	\end{cases}.
\end{multline}
	
\noindent Inserting \eqref{eq: DWtQEleF2Trans2} with $\delta = -\varpi_{\F}^{-1} \Tr(\varpi_{\bL} \alpha)^2$ into the integral representation of $H_{n_0+1,m}$ we get
\begin{multline} \label{eq: hn0+1-BdRed1}
	\widetilde{h}_{n_0+1,m}(\chi) \ll \id_{m > \frac{n_0+1}{3}} \cdot q^{-\frac{n_0+1+m}{2}} \id_{m}(\cond(\chi)) \cdot \\
	\extnorm{ \int_{\substack{\bL^1 \\ \Tr(\varpi_{\bL} \alpha) \in \varpi_{\F}^{\frac{n_0}{2}+1-m} \vO_{\F}^{\times}}} \beta(\varpi_{\bL} \alpha) \eta_{\bL/\F} \left( \Tr(\varpi_{\bL}\alpha) \right) \cdot \chi \left( \tfrac{4 \varpi_{\F} + \Tr(\varpi_{\bL} \alpha)^2}{- \Tr(\varpi_{\bL} \alpha)^2} \right) \ud \alpha }.
\end{multline}
	Applying the change of variables $t = \tfrac{\varpi_{\bL}\alpha + \overline{\varpi_{\bL}\alpha}}{\varpi_{\bL}\alpha - \overline{\varpi_{\bL}\alpha}} \tfrac{1}{\varpi_{\bL}}$ the inner integral in \eqref{eq: hn0+1-BdRed1} becomes, taking into account \emph{the measure normalization} in Proposition \ref{prop: MeasNorm} (3)
	$$ q^{-\frac{1}{2}} \cdot 2\beta(\varpi_{\bL})\eta_{\bL/\F}(-2)\chi^{-1}(-\varpi_{\F}) \int_{\varpi_{\F}^{\frac{n_0}{2}-m} \vO_{\F}^{\times}} \beta(1+t \varpi_{\bL}) \eta_{\bL/\F}\chi^{-2}(t) \ud t. $$
	We apply Lemma \ref{lem: JacTypeSRCBis} (1) to bound the above integrals and conclude the desired inequalities.
\end{proof}

\begin{lemma} \label{lem: e2DWt-FineBd3}
	Let $n_0 \geq \max(4\cond(\Pi), 2\cond(\Pi)+2)$ and $m \leq \cond(\Pi)$. Then for any $\delta \in \varpi_{\F}^{n_0+1-2m}\vO_{\F}^{\times}$ we have
	$$ \int_{\F-\vO_{\F}} \VorH_{\Pi,\psi} \circ \Mult_{-1} \left( \Trans(\delta).G_m \right)(t) \cdot \psi(-t)\chi^{-1}(t)\norm[t]^{-\frac{1}{2}} \ud^{\times} t = 0. $$
	Consequently, we get for any unitary $\chi$ the vanishing of $\widetilde{h}_{n_0+1,m}(\chi) = 0$.
\end{lemma}
\begin{proof}
	The proof is quite similar to Lemma \ref{lem: e1DWt-FineBd4}. One shows that the support of $\VorH_{\Pi,\psi} \circ \Mult_{-1} \left( \Trans(\delta).G_m \right)$ is contained in $\delta \vP_{\F}^{2m-\max(4\cond(\Pi), 2\cond(\Pi)+3)} \subset \vO_{\F}$ under the assumption.
\end{proof}

	\subsection{The Bounds of Dual Weight}
	
\begin{lemma} \label{lem: h2-Bdn0=2}
	Suppose $\cond(\Pi)=0$ and $n_0=2$. We have $\widetilde{h}_3^-(\chi) = 0$ for any unitary character $\chi$.
\end{lemma}
\begin{proof}
	Necessarily we have $\Pi = \mu_1 \boxplus \mu_2 \boxplus \mu_3$ for unramified $\mu_j$. Note that $H_{3,0}$ (resp. $H_{3,1}$) is related to $G_0$ (resp. $G_1$) by the formulae
	$$ H_{3,0} = \int_{\substack{\bL^1 \\ \Tr(\varpi_{\bL} \alpha) \in \vP_{\F}^{2}}} \beta(\varpi_{\bL} \alpha) \ud \alpha \cdot \Trans \left( (-\varpi_{\F})^{3} \right).G_0, $$
	$$ H_{3,1} = \int_{\substack{\bL^1 \\ \Tr(\varpi_{\bL} \alpha) \in \varpi_{\F} \vO_{\F}^{\times}}} \beta(\varpi_{\bL} \alpha) \cdot \eta_{\bL/\F} \left( \Tr(\varpi_{\bL}\alpha) \right) \cdot \Trans \left( -\varpi_{\F}^{-1} \Tr(\varpi_{\bL} \alpha)^2 \right).G_1 \ud \alpha. $$
	Inspecting the supports of $G_0$ and $G_1$ given in Corollary \ref{cor: VHQEleF2Bis} we see $\mathrm{supp}(H_{3,m}) \subset \vO_{\F}$ for $m \in \{ 0,1 \}$.
\end{proof}

\begin{proposition} \label{prop: DWtBde=2}
	With the test function $H$ given by \eqref{eq: TestFROI} the dual weight function is bounded as
	$$ \widetilde{h}(\chi) \ll_{\epsilon} \Cond(\Pi)^{2+\epsilon} q^{-\frac{n_0+1}{2}+\epsilon} \id_{\leq \max(\frac{n_0}{2},6\cond(\Pi))}(\cond(\chi)). $$
\end{proposition}
\begin{proof}
	The argument is quite similar to and simpler than Proposition \ref{prop: DWtBde=1}. In the case $\cond(\Pi) > 0$ we distinguish $n_0 \leq 4\cond(\Pi)$, resp. $n_0 \geq 4\cond(\Pi)$. We apply Lemma \ref{lem: DWtInfty}, \ref{lem: DWt+Bd} and \ref{lem: DWt-TrivBd} (with $A=4$), resp. Lemma \ref{lem: DWtInfty}, \ref{lem: DWt+Bd} and Lemma \ref{lem: e2DWt-FineBd1}-\ref{lem: e2DWt-FineBd3}. In the case $\cond(\Pi)=0$, we apply Lemma \ref{lem: DWtInfty}, \ref{lem: DWt+Bd} and \ref{lem: h2-Bdn0=2}. We leave the details to the reader.
\end{proof}

	\subsection{The Bounds of Unramified Dual Weight}
	
\begin{lemma} \label{lem: e2DNDWt-FineBd1}
	Suppose $\tfrac{n_0}{2} + 1 \leq n \leq n_0$ and $n \geq a(\Pi)$. Then the function $\widetilde{h}_{2n}^-(\norm_{\F}^s) \neq 0$ is non-vanishing only if $n=n_0$. Moreover for any $k \in \Z_{\geq 0}$ we have
	$$ \widetilde{H}_{2n_0}^-(k; \tfrac{1}{2}) \ll_{k,\epsilon} q^{-\frac{1}{2}+n_0 \epsilon}, \quad \widetilde{H}_{2n_0}^-(k; -\tfrac{1}{2}) \ll_{k,\epsilon} q^{-n_0-\frac{1}{2}+n_0\epsilon}. $$
\end{lemma}
\begin{proof}
	By Corollary \ref{cor: VHQEleF2} and \eqref{eq: FPTestFR1} the support of $\VorH_{\Pi,\psi}(G_n)$, hence also $\VorH_{\Pi,\psi} \circ \Mult_{-1} (H_{2n})$ are contained in $\varpi_{\F}^{-n} \vO_{\F}^{\times}$. Therefore we get by \eqref{eq: h2n-BdRedR} the relation $\widetilde{h}_{2n}^-(\norm_{\F}^s) = q^{ns} \widetilde{h}_{2n}^-(\id)$. The stated results follow readily from Lemma \ref{lem: e2DWt-FineBd1}.
\end{proof}

\begin{proposition} \label{prop: DNDWtBde=2}
	With the test function $H$ given by \eqref{eq: TestFROI} we have for any $k \in \Z_{\geq 0}$
	$$ \widetilde{H}(k; \tfrac{1}{2}) \ll_{k,\epsilon} \Cond(\Pi)^{\frac{1}{2}} \cdot q^{(n_0+1) \epsilon}, \quad \widetilde{H}(k; -\tfrac{1}{2}) \ll_{k,\epsilon} \Cond(\Pi)^{4+\epsilon} \cdot q^{(n_0+1) \epsilon}. $$
\end{proposition}
\begin{proof}
	The proof is quite similar to and simpler than the one of Proposition \ref{prop: DNDWtBde=1}. We simply note that we can take $n_1=n_0+1$, and leave the details to the reader.
\end{proof}

\appendix

\section{Relation with Petrow--Young's Exponential Sums}
	
	For simplicity of notation we shall write $\sideset{}{_t} \sum$ for $\sideset{}{_{t \in \fF_q}} \sum$, and use the convention $\rho(0)=0$ for any character $\rho$ of $\fF_q^{\times}$ (even if $\rho=\id$ is the trivial one). In \eqref{eq: PYExpSVar} we have encountered the following algebraic exponential sum
\begin{equation} \label{eq: ExpSe=1} 
	S = S(\chi_0,\chi) := \sideset{}{_{\alpha \in \fF_q - \{ \pm 1\}}} \sum \chi_0 \left( \tfrac{\alpha + 1}{\alpha - 1} \right) H(1-\alpha^2,q; (\id, \id), (\chi^{-1},\eta)), 
\end{equation}
where $\chi_0,\chi$ are non-trivial characters of $\fF_q^{\times}$ and $\eta$ is the unique non-trivial quadratic character of $\fF_q^{\times}$. Note that the setting of the relevant dual weight specializes to the Petrow--Young's \cite{PY19_CF} upon taking $\Pi = \id \boxplus \id \boxplus \id$. It is a natural question to relate $S$ with their algebraic exponential sum
\begin{equation} \label{eq: PYExpS}
	T = T(\chi_0,\chi) := \sideset{}{_{u,v}} \sum \chi \left( \tfrac{u(u+1)}{v(v+1)} \right) \chi_0(uv-1).
\end{equation}
	The purpose of this appendix is to give such an explicit relation. Our approach will take into account the recent discovery of Xi \cite{Xi23}, which relates $T$ to a special value of a hypergeometric sum of Katz.
	
	We need to rewrite a \emph{hyper-Kloosterman sum} via the \emph{duplication formula of Gauss sums}.
\begin{definition}
	For a character $\rho$ of $\fF_q^{\times}$ and a non-trivial character $\psi$ of $\fF_q$ we have the Gauss sum
	$$ \tau(\rho) = \tau(\rho, \psi) := \sideset{}{_t} \sum \rho(t) \psi(t). $$
\end{definition}
\begin{lemma} \label{lem: DupFGS}
	We have the relation $\tau(\rho^2) \tau(\eta) = \rho(4) \tau(\rho) \tau(\rho \eta)$ for any character $\rho$ of $\fF_q^{\times}$.
\end{lemma}
\begin{proof}
	If $\rho = \id$ or $\eta$, the stated relation trivially holds true. Assume $\rho \neq \id, \eta$. We have the classical relation of Gauss and Jacobi sums \cite[\S 8.3 Theorem 1]{IR90}
	$$ \tfrac{\tau(\rho)^2}{\tau(\rho^2)} = J(\rho,\rho) = \sideset{}{_t} \sum \rho(t(1-t)), \quad \tfrac{\tau(\rho)\tau(\eta)}{\tau(\rho \eta)} = J(\rho,\eta) = \sideset{}{_t} \sum \eta(t) \rho(1-t). $$
	Since $q$ is odd, we can rewrite by completing the square
\begin{multline*}
	J(\rho,\rho) = \sideset{}{_t} \sum \rho(t(1-t)) = \rho(4)^{-1} \sideset{}{_t} \sum \rho(1-t^2) \\
	= \rho(4)^{-1} \sideset{}{_t} \sum (1+\eta(t)) \rho(1-t) = \rho(4)^{-1} \sideset{}{_t} \sum \eta(t) \rho(1-t) = \rho(4)^{-1} J(\rho,\eta).
\end{multline*}
	The stated relation follows readily since $\tau(\rho) \neq 0$.
\end{proof}
\begin{corollary} \label{cor: HypKlSVar}
	For any $\delta \in \fF_q^{\times}$ we have the equality
	$$ \sideset{}{_{u,t \neq 0}} \sum \eta \left( 1-u \right) \psi \left( \tfrac{\delta}{t^2u} + 2t \right) = \sideset{}{_{x_1x_2x_3 = \delta}} \sum \psi(x_1+x_2+x_3). $$
\end{corollary}
\begin{proof}
	The right hand side is $\Kl_3(\delta)$, a hyper-Kloosterman sum, whose Mellin transform satisfies
	$$ \sideset{}{_{\delta \in \fF_q^{\times}}} \sum \Kl_3(\delta) \rho(\delta) = \tau(\rho)^3, \quad \forall \rho \in \widehat{\fF_q^{\times}}. $$
	It suffices to identify the Mellin transform of the left hand side as $\tau(\rho)^3$. We have
\begin{multline*}
	\sideset{}{_{\delta}} \sum \sideset{}{_{u,t \neq 0}} \sum \eta \left( 1-u \right) \psi \left( \tfrac{\delta}{t^2u} + 2t \right) \rho(\delta) = \tau(\rho) \sideset{}{_{u,t}} \sum \eta \left( 1-u \right) \psi \left( 2t \right) \rho(t^2u) \\
	= \rho(4)^{-1} \tau(\rho) \tau(\rho^2) \sideset{}{_u} \sum \eta \left( 1-u \right) \rho(u) = \tfrac{\tau(\rho)^2 \tau(\rho^2) \tau(\eta)}{\rho(4) \tau(\rho \eta)}.
\end{multline*}
	Lemma \ref{lem: DupFGS} identifies the above right hand side precisely as $\tau(\rho)^3$.
\end{proof}

	For $S=S(\chi_0,\chi)$, we make the change of variables $u=1+\alpha$ and $v = 1-\alpha$, and detect the condition $u+v = 2$ with the additive character $\psi$ to get
\begin{align*}
	S &= -q^{-\frac{3}{2}} \chi_0(-1)\sideset{}{_{u+v=2}} \sum \chi_0(u) \overline{\chi_0}(v) \sideset{}{_{\substack{x_i,y_i \\ x_1x_2 = uv y_1y_2}}} \sum \chi(y_1) \eta(y_2) \psi(x_1+x_2-y_1-y_2) \\
	&= -q^{-\frac{5}{2}} \chi_0(-1) \sideset{}{_t} \sum \sideset{}{_{\substack{x_i,y_i,u,v \\ x_1x_2 = uv y_1y_2}}} \sum \chi_0(u) \overline{\chi_0}(v) \chi(y_1) \eta(y_2) \psi(x_1+x_2-y_1-y_2+tu+tv-2t) \\
	&= -q^{-\frac{5}{2}}(S_1+S_2).
\end{align*}
	In the above the sum $S_1$ is defined and computed as
\begin{align*}
	S_1 &:= \chi_0(-1) \sideset{}{_{t \in \fF_q^{\times}}} \sum \sideset{}{_{\substack{x_i,y_i,u,v \\ x_1x_2 = uv y_1y_2}}} \sum \chi_0(u) \overline{\chi_0}(v) \chi(y_1) \eta(y_2) \psi(x_1+x_2-y_1-y_2-tu-tv+2t) \\
	&= \chi_0(-1) \sideset{}{_{t \in \fF_q^{\times}}} \sum \sideset{}{_{\substack{x_i,y_i,u,v \\ t^2 x_1x_2 = uv y_1y_2}}} \sum \chi_0(u) \overline{\chi_0}(v) \chi(y_1) \eta(y_2) \psi(x_1+x_2-y_1-y_2-u-v+2t) \\
	&= \chi_0(-1) \sideset{}{_{t \in \fF_q^{\times}}} \sum \sideset{}{_{\substack{x_i,y_i,u,v \\ t^2 x_1x_2 = uv y_1}}} \sum \eta(y_2) \psi(y_2(x_2-1)) \cdot \chi_0(u) \overline{\chi_0}(v) \chi(y_1) \psi(x_1+2t-u-v-y_1) \\
	&= \chi_0(-1) \tau(\eta) \sideset{}{_{t \in \fF_q^{\times}}} \sum \sideset{}{_{\substack{x_i,y_i \\ t^2 x_1x_2 = y_1y_2y_3}}} \sum \eta(x_2-1) \chi(y_1) \chi_0(y_2) \overline{\chi_0}(y_3) \psi(x_1+2t-y_1-y_2-y_3) \\
	&= \chi_0\eta(-1) \tau(\eta) \sideset{}{_{y_i}} \sum \chi(y_1) \chi_0(y_2) \overline{\chi_0}(y_3) \psi(-y_1-y_2-y_3) \sideset{}{_{x_2,t \neq 0}} \sum \eta(1-x_2) \psi \left( \tfrac{y_1y_2y_3}{t^2 x_2} + 2t \right).
\end{align*}
	Re-naming the variable $u=x_2$ we identify the inner sum as $\Kl_3(y_1y_2y_3)$ by Corollary \ref{cor: HypKlSVar} so that
\begin{align*} 
	S_1 &= \chi_0\eta(-1) \tau(\eta) \sideset{}{_{\substack{x_i,y_i \\ x_1x_2x_3=y_1y_2y_3}}} \sum \chi(y_1) \chi_0(y_2) \overline{\chi_0}(y_3) \psi(x_1+x_2+x_3-y_1-y_2-y_3) \\
	&= -q^{\frac{5}{2}} \chi_0\eta(-1) \tau(\eta) H(1,q; (\id,\id,\id), (\overline{\chi}, \chi_0, \overline{\chi_0})) = -q^2 \cdot T(\chi_0,\chi),
\end{align*}
	where we have applied \cite[Theorem 1.1]{Xi23} to get the last equality. The sum $S_2$ is defined as
	$$ S_2 := \chi_0(-1) \sideset{}{_{\substack{x_i,y_i,u,v \\ x_1x_2 = uv y_1y_2}}} \sum \chi_0(u) \overline{\chi_0}(v) \chi(y_1) \eta(y_2) \psi(x_1+x_2-y_1-y_2). $$
	Let $t = uv$, so $u = tv^{-1}$, and sum first over $v \in \fF_q^{\times}$. We see that $S_2 \neq 0$ only if $\chi_0=\eta$, and
\begin{align*}
	S_2 &= \delta_{\chi_0=\eta} \cdot \eta(-1) (q-1) \sideset{}{_{\substack{t,x_i,y_i \\ x_1x_2 = t y_1y_2}}} \sum \chi(y_1) \eta(ty_2) \psi(x_1+x_2-y_1-y_2) \\
	&= \delta_{\chi_0=\eta} \cdot \eta(-1) (q-1) \sideset{}{_{\substack{x_i,y_i \\ x_1x_2 = y_1y_2}}} \sum \chi(y_1) \eta(y_2) \psi(x_1+x_2-y_1) \sideset{}{_{t \neq 0}} \sum \psi(-t^{-1}y_2) \\
	&= -\delta_{\chi_0=\eta} \cdot \eta(-1) (q-1) \sideset{}{_{\substack{x_i,y_i \\ x_1x_2 = y_1y_2}}} \sum \chi(y_1) \eta(y_2) \psi(x_1+x_2-y_1) \\
	&= - \delta_{\chi_0=\eta} \cdot \eta(-1) (q-1) \sideset{}{_{x_1,x_2,y_1}} \sum \chi(y_1) \eta \left( \tfrac{x_1x_2}{y_1} \right) \psi(x_1+x_2-y_1) \\
	&= - \delta_{\chi_0=\eta} \cdot \eta(-1) (q-1) \tau(\eta)^2 \overline{\tau(\eta \chi^{-1})} = - \delta_{\chi_0=\eta} \cdot q(q-1) \overline{\tau(\eta \chi^{-1})}.
\end{align*}
	We summarize the above computation as the following result.
\begin{proposition}
	The two algebraic exponential sums $S(\chi_0,\chi)$ in \eqref{eq: ExpSe=1} and $T(\chi_0,\chi)$ in \eqref{eq: PYExpS} satisfy
	$$ S(\chi_0,\chi) = q^{-\frac{1}{2}} \cdot T(\chi_0,\chi) + \delta_{\chi_0=\eta} \cdot q^{-\frac{1}{2}}(1-q^{-1}) \overline{\tau(\eta \chi^{-1})}. $$
\end{proposition}

\bibliographystyle{acm}
	
\bibliography{mathbib}
	
	
\end{document}